\pdfoutput=1

\documentclass[11pt,a4paper]{article}
\usepackage[margin=2.5cm]{geometry}
\usepackage[T1]{fontenc}
\usepackage{lmodern}
\usepackage[numbers,sort&compress]{natbib}
\usepackage{authblk}

\usepackage[nottoc]{tocbibind}

\usepackage{amsmath}
\usepackage{todonotes}
\usepackage{amssymb}
\usepackage{amsfonts}
\usepackage{mathrsfs}
\usepackage{amsthm}
\usepackage{amsopn}
\usepackage{mathtools}
\usepackage{setspace}
\usepackage{graphicx}
\graphicspath{{figures/}}
\newcommand{\lvPeriodSym}{\ensuremath{17.72}}  % period of the periodic orbit of f_theta
\newcommand{\lvPeriodSpreadSym}{\ensuremath{1.78\times10^{-6}}}  % relative spread of the last five return times
\newcommand{\lvReturnErrorSym}{\ensuremath{3.08\times10^{-6}}}  % max relative distance between successive crossing states
\newcommand{\lvProjErrorSym}{\ensuremath{3.16\times10^{-13}}}  % max |P y(t) - x_theta(t)| on [0, 500]
\newcommand{\lvWindowStartSym}{\ensuremath{437.2}}  % start of the window shown in the population figure
\newcommand{\lvTestFieldMaxSym}{\ensuremath{0.014}}  % max |f_theta - g| on the test points of K
\newcommand{\lvTestFieldRmsSym}{\ensuremath{0.0062}}  % RMS of |f_theta - g| on the test points of K
\newcommand{\lvTestJacMaxSym}{\ensuremath{0.23}}  % max |Df_theta - Dg|_F on the test points of K
\newcommand{\lvTestJacRmsSym}{\ensuremath{0.030}}  % RMS of |Df_theta - Dg|_F on the test points of K
\newcommand{\lvGammaFieldMaxSym}{\ensuremath{0.011}}  % max |f_theta - g| within the radius of Gamma
\newcommand{\lvGammaJacMaxSym}{\ensuremath{0.094}}  % max |Df_theta - Dg|_F within the radius of Gamma
\newcommand{\lvSaddleDistMaxSym}{\ensuremath{0.023}}  % largest of the three distances above
\newcommand{\lvPlaneMinSym}{\ensuremath{0.0026}}  % min over the three faces
\newcommand{\lvSweepPeriodRangeSym}{\ensuremath{3.70\times10^{-6}}}  % (max - min) / reference period
\newcommand{\lvSweepDistanceSym}{\ensuremath{7.67\times10^{-6}}}  % max distance of the tails to the reference orbit
\newcommand{\lvFloquetSecondSym}{\ensuremath{0.412}}  % Floquet multiplier, by decreasing modulus
\newcommand{\lvFloquetThirdSym}{\ensuremath{1.12\times10^{-8}}}  % Floquet multiplier, by decreasing modulus
\newcommand{\lvPeriodAsym}{\ensuremath{21.99}}  % period of the periodic orbit of f_theta
\newcommand{\lvPeriodSpreadAsym}{\ensuremath{1.06\times10^{-6}}}  % relative spread of the last five return times
\newcommand{\lvReturnErrorAsym}{\ensuremath{2.18\times10^{-6}}}  % max relative distance between successive crossing states
\newcommand{\lvProjErrorAsym}{\ensuremath{5.72\times10^{-13}}}  % max |P y(t) - x_theta(t)| on [0, 500]
\newcommand{\lvWindowStartAsym}{\ensuremath{415.8}}  % start of the window shown in the population figure
\newcommand{\lvTestFieldMaxAsym}{\ensuremath{0.017}}  % max |f_theta - g| on the test points of K
\newcommand{\lvTestFieldRmsAsym}{\ensuremath{0.0064}}  % RMS of |f_theta - g| on the test points of K
\newcommand{\lvTestJacMaxAsym}{\ensuremath{0.26}}  % max |Df_theta - Dg|_F on the test points of K
\newcommand{\lvTestJacRmsAsym}{\ensuremath{0.031}}  % RMS of |Df_theta - Dg|_F on the test points of K
\newcommand{\lvGammaFieldMaxAsym}{\ensuremath{0.012}}  % max |f_theta - g| within the radius of Gamma
\newcommand{\lvGammaJacMaxAsym}{\ensuremath{0.093}}  % max |Df_theta - Dg|_F within the radius of Gamma
\newcommand{\lvSaddleDistMaxAsym}{\ensuremath{0.062}}  % largest of the three distances above
\newcommand{\lvPlaneMinAsym}{\ensuremath{0.0026}}  % min over the three faces
\newcommand{\lvSweepPeriodRangeAsym}{\ensuremath{6.97\times10^{-7}}}  % (max - min) / reference period
\newcommand{\lvSweepDistanceAsym}{\ensuremath{2.72\times10^{-5}}}  % max distance of the tails to the reference orbit
\newcommand{\lvFloquetSecondAsym}{\ensuremath{0.118}}  % Floquet multiplier, by decreasing modulus
\newcommand{\lvFloquetThirdAsym}{\ensuremath{1.21\times10^{-10}}}  % Floquet multiplier, by decreasing modulus
\newcommand{\lvTestPoints}{\ensuremath{3000}}  % test points in K
\newcommand{\lvGammaRadius}{\ensuremath{0.1}}  % radius of the neighborhood of Gamma
\newcommand{\lvPlaneGrid}{\ensuremath{401}}  % grid points per side on each face
\newcommand{\lvSweepCount}{\ensuremath{30}}  % initial conditions in the sweep
\newcommand{\lvSweepTailStart}{\ensuremath{480}}  % start of the tail compared with the reference orbit
\newcommand{\lvSweepLow}{\ensuremath{0.02}}  % smallest coordinate of the initial conditions
\newcommand{\lvSweepHigh}{\ensuremath{1}}  % largest coordinate of the initial conditions

\newcommand{\ringCstar}{\ensuremath{2.0547}}  % threshold c*
\newcommand{\ringRangeAtTwoPointOne}{\ensuremath{0.290}}  % peak-to-peak range of the orbit at c = 2.1
\newcommand{\ringPeriodAtTwoPointOne}{\ensuremath{3.63}}  % period of the orbit at c = 2.1
\newcommand{\ringRangeAtFive}{\ensuremath{1.152}}  % peak-to-peak range of the orbit at c = 5
\newcommand{\ringRangeAtTen}{\ensuremath{1.179}}  % peak-to-peak range of the orbit at c = 10
\newcommand{\ringRangeAtForty}{\ensuremath{1.214}}  % peak-to-peak range of the orbit at c = 40
\newcommand{\ringPeriodAtForty}{\ensuremath{2.90}}  % period of the orbit at c = 40
\newcommand{\ringSweepCount}{\ensuremath{22}}  % initial conditions at each of c = 2.1, 5, 10
\newcommand{\ringSweepSpread}{\ensuremath{3.6\times10^{-7}}}  % max spread of the extreme values of x_1 over them

\usepackage{epstopdf}
\usepackage{float}
\usepackage{caption}
\usepackage{subcaption}
\usepackage{algorithm, algpseudocode}
\usepackage{tikz}
\usetikzlibrary{graphs.standard}
\usetikzlibrary{positioning,arrows.meta}
\usepackage{xcolor}
\usepackage{enumitem}
\usepackage{hyperref}
\usepackage{cleveref}

\DeclareGraphicsExtensions{.eps,.pdf,.png,.jpg}

\setlist[enumerate]{leftmargin=.5in}
\setlist[itemize]{leftmargin=.5in}

\DeclareMathOperator{\diag}{diag}
\DeclareMathOperator{\idty}{Id}

\theoremstyle{plain}
\newtheorem{theorem}{Theorem}[section]
\newtheorem{lemma}[theorem]{Lemma}
\newtheorem{corollary}[theorem]{Corollary}
\newtheorem{proposition}[theorem]{Proposition}

\theoremstyle{definition}
\newtheorem{definition}[theorem]{Definition}
\theoremstyle{remark}
\newtheorem{remark}[theorem]{Remark}

\crefname{hypothesis}{Hypothesis}{Hypotheses}
\crefname{fact}{Fact}{Facts}

\numberwithin{equation}{section}
\counterwithin{figure}{section}
\counterwithin{table}{section}
\numberwithin{algorithm}{section}

\makeatletter
\g@addto@macro\appendix{%
  \gdef\thetheorem{\@Alph\c@section.\arabic{theorem}}}
\makeatother

\hypersetup{colorlinks, linkcolor=blue!60!black,
    citecolor=green!45!black, urlcolor=blue!60!black}

\begin{document}

\title{Modeling sequential cognitive states via population-level cortical dynamics\thanks{All authors contributed equally to this work.}}
\author[1]{M. Virginia Bolelli}
\author[1]{Luca Greco}
\author[2]{Dario Prandi\thanks{Corresponding author.}}
\affil[1]{L2S, Universit\'e Paris-Saclay, CentraleSup\'elec, Gif-sur-Yvette, France}
\affil[2]{CNRS, L2S, Universit\'e Paris-Saclay, CentraleSup\'elec, Gif-sur-Yvette, France}
\affil[ ]{\texttt{maria-virginia.bolelli@centralesupelec.fr, luca.greco@centralesupelec.fr, dario.prandi@centralesupelec.fr}}
\date{}
\maketitle

\begin{abstract}
In this work, we propose a mathematical framework for cyclic and sequential patterns of brain activity, combining heteroclinic dynamics with discrete neural-field models. Under explicit structural and positivity assumptions on the prescribed activity patterns, we show that the resulting spatially discrete neural-field system supports neither heteroclinic cycles nor nonconstant periodic orbits. Heteroclinic dynamics, on the other hand, arise in Lotka--Volterra-type systems, but these equations do not directly correspond to neuronal processes. To bridge these descriptions, we use a $C^1$ version of the Universal Approximation Theorem to approximate a target vector field by a neural network interpretable as a high-dimensional Amari-type neural-field system. Each periodic solution of the reduced dynamics admits a unique periodic realization at the population level with the same minimal period. For three-dimensional heteroclinic targets, we establish sufficient conditions ensuring that the approximating vector field admits a unique hyperbolic attracting periodic orbit in a trapping channel surrounding the original heteroclinic cycle. As a proof of concept, we apply this construction to a Lotka--Volterra-type target phenomenologically representing sequential network activations associated with focused-attention meditation. %
\end{abstract}



\setcounter{tocdepth}{2}
\tableofcontents

\section{Introduction}

Experimental evidence shows that neural dynamics often exhibit metastable states: transient, weakly stable configurations in which the system remains for a finite time before transitioning to another state \cite{ alderson2020metastable,capouskova2022modes,rabinovich2006dynamical, van2025large}. Rather than isolated events, these states appear as ordered sequences that can be tracked across different tasks, measurement modalities, and levels of analysis, suggesting a common dynamical motif. Sequential metastable patterns have been reported across several cognitive and perceptual processes, including motor memory consolidation \cite{nicolas2025unraveling}, binocular rivalry \cite{leopold1996activity}, and focused attention meditation \cite{hasenkamp2012effects,hasenkamp2012mind}. The recurrence of such patterns across domains suggests that sequential switching reflects a broader organizing principle of neural activity \cite{rabinovichInformationFlowDynamics2012}.

Several dynamical mechanisms can account for sequential metastable activity, including noise-induced switching between attractors, excitable network dynamics, and dynamics organized by slow manifolds. In excitable networks, transitions may be triggered by inputs or perturbations crossing an excitation threshold, whereas heteroclinic cycles provide an autonomous mechanism for ordered transitions through the unstable directions of saddle equilibria \cite{ashwin2024network,fadera2024arbitrary}. Related models of winnerless competition have also been used to describe input- and noise-dependent perceptual switching \cite{ashwin2010low}. In this work, we focus on asymptotically stable heteroclinic cycles connecting saddle equilibria, which provide a convenient dynamical template for ordered sequential activity.
In such systems, trajectories evolve near a sequence of saddle equilibria connected by orbits along their unstable manifolds and give rise to ordered yet flexible transitions \cite{krupa1995asymptotic,krupa1997robust,rabinovich2001dynamical,afraimovich2004origin,afraimovich2004heteroclinic,castro2025robust}. For a deterministic heteroclinic cycle, the global geometry of the stable and unstable manifolds fixes the order in which the states are visited, whereas residence times near the saddles depend on their local instability rates and increase as trajectories approach an attracting heteroclinic cycle.  Perturbations can prevent this divergence and modulate the resulting sequential dynamics \cite{rabinovich2006dynamical,rabinovich2008transient,stone1990random,jeong2023effect}. Classical models exhibiting heteroclinic behavior include Lotka--Volterra systems \cite{may1975nonlinear} and graph-based dynamical constructions \cite{ashwin2013designing}. Here, heteroclinic dynamics are used as a phenomenological dynamical framework for sequential activity, rather than as the unique mechanism capable of generating metastable sequences.

A neural-population framework is provided by the mean-field equations put forward by Wilson--Cowan \cite{wilson1972excitatory,wilson1973mathematical} and Amari \cite{amari72}. These equations describe the collective activity of interacting neural populations and have been applied to a wide range of perceptual and cognitive phenomena \cite{coombesNeuralFieldsTheory2014,ermentroutMathematical1979,bressloffGeometric2001,tamekueMathematical2024,tamekueReproducibility2025,bolelli2025individuation,bolelli2025neural}. Heteroclinic structures have been studied in neural-field models with specific interaction kernels \cite{schwappach2015metastable}, while reproducible sequential activity has been investigated in random excitatory--inhibitory networks \cite{huerta2004reproducible}. In \cite{rabinovich2008transient}, a three-cluster Wilson--Cowan model is used to illustrate numerically a stable limit cycle within a stable heteroclinic channel. Bifurcation mechanisms in low-dimensional excitatory--inhibitory systems have also been investigated in \cite{nechyporenko2025switching}.

\subsection{Discrete neural-field formulation}

In this work, we propose a mathematical framework for cyclic and sequential patterns of brain activity that combines heteroclinic dynamics with mean-field neural models. We focus on a single-equation, Amari-type discrete neural-field formulation describing $n\ge 3$ interacting neural populations:
\begin{equation}
    \label{eq:wc-intro}
    \dot x = -x + \sigma(Wx+b).
\end{equation}
Here, $x=(x_1,\ldots,x_n)\in\mathbb{R}^n$ collects the activity of $n$ neural populations, $W$ is an $n\times n$ connectivity matrix, $\sigma$ is a nonlinear activation function, and $b$ is a vector encoding biases and external inputs.
Compared with excitatory--inhibitory Wilson--Cowan models, assigning one variable to each neural population keeps the dimensionality of the system tractable while retaining a direct interpretation of the variables in terms of population activity.

Within this setting, we investigate whether system~\eqref{eq:wc-intro} can generate cyclic and transient sequences between prescribed activity states. We focus in particular on equilibrium configurations representing activity patterns in which each population is either inactive or active at a characteristic activity level.

Our main result here is Theorem~\ref{thm:main}, which identifies a structural obstruction to realizing these activity patterns as sequential states of system~\eqref{eq:wc-intro}. More precisely, when the prescribed patterns correspond to \(n\) linearly independent equilibria satisfying an additional positivity condition, we show that the equilibrium constraints force the connectivity matrix to satisfy a detailed-balance condition. This places the resulting system within the classical class of neural networks admitting an energy or Lyapunov function, as in the works of Cohen and Grossberg~\cite{cohen1983absolute} and Hopfield~\cite{hopfield1984neurons}. In our setting, this yields a strict Lyapunov function and therefore rules out both heteroclinic cycles and periodic orbits.

This obstruction is particularly relevant to the three-cluster Wilson--Cowan model considered in \cite{rabinovich2008transient}, which has the same structural form as \eqref{eq:wc-intro}. In that work, sequential metastable activity is associated with saddle states on the coordinate axes, and a stable limit cycle is illustrated numerically. Our result shows that, if three such states are realized as linearly independent equilibria on the coordinate axes and the activation function satisfies our assumptions, the resulting connectivity necessarily satisfies detailed balance, excluding both heteroclinic cycles and nonconstant periodic orbits. This does not contradict the reported numerical observations, since the existence of these prescribed equilibria and the required properties of the activation function are not established for that example.

As a consequence of Theorem~\ref{thm:robust}, Corollary~\ref{cor:no-sequential} shows that this obstruction persists under small autonomous perturbations of the connectivity matrix, provided that the activation function is real-analytic and the prescribed equilibria are nondegenerate. Even when these perturbations break detailed balance, they cannot produce periodic orbits or heteroclinic cycles visiting distinct prescribed states.

\subsection{Aggregated dynamics and universal approximation}

The preceding obstruction motivates an alternative description in which the variables exhibiting sequential dynamics are interpreted as aggregated macroscopic variables rather than as the activities of individual neural populations. This is consistent with the observation that neural activity is often organized in terms of large-scale network dynamics \cite{stringerHighdimensional2019}, while lower-dimensional descriptions may capture selected aspects of its collective behavior \cite{cunninghamDimensionality2014}.

We therefore prescribe a target vector field \(g\) at the macroscopic level and seek a higher-dimensional neural-field system whose projection approximates this dynamics. Related approaches synthesize continuous-time recurrent neural networks by approximating a prescribed vector field \cite{funahashi1993approximation,trischler2016synthesis}. Here we additionally require \(C^1\) approximation and an exact projection from the population dynamics to the macroscopic variables.

More precisely, using a \(C^1\) universal approximation theorem \cite{pinkus1999approximation,hornik1990universal}, Corollary~\ref{cor:UAT} shows that, on a compact domain, a smooth target vector field \(g\) can be approximated, together with its first derivatives, by
\[
    f_\theta(x)=-x+P\sigma(Wx+b).
\]
The reduced dynamics \(\dot x=f_\theta(x)\) is the exact projection of the higher-dimensional neural-field system
\[
    \dot y=-y+\sigma(WPy+b),
    \qquad x=Py.
\]
Periodic activity in the reduced variables is realized by periodic activity of the neural populations. Moreover, the matrix \(P\) can be chosen with a block structure, associating each macroscopic variable with a distinct group of populations.
In this representation, the $N$ units of the approximating network are interpreted as neural populations, whereas the $n$ reduced variables represent their macroscopic aggregates.

We then ask whether the approximating system retains the sequential behavior of the target. If \(g\) has an attracting heteroclinic cycle, small perturbations may break its connections, so \(C^1\) closeness alone is not enough. In Section~\ref{sec:sequential-periodic}, we give conditions under which the original heteroclinic cycle present in the target dynamics is replaced by a periodic orbit that closely follows its geometry. Theorem~\ref{thm:main-periodic} treats the three-dimensional case, which is the relevant setting for our case study, while a companion manuscript in preparation~\cite{bolelli2026periodic} studies the corresponding problem in arbitrary ambient dimension.

\subsection{Application to focused attention meditation}

\begin{figure}
    \begin{subfigure}[b]{0.45\textwidth}
        \includegraphics[width = \textwidth]{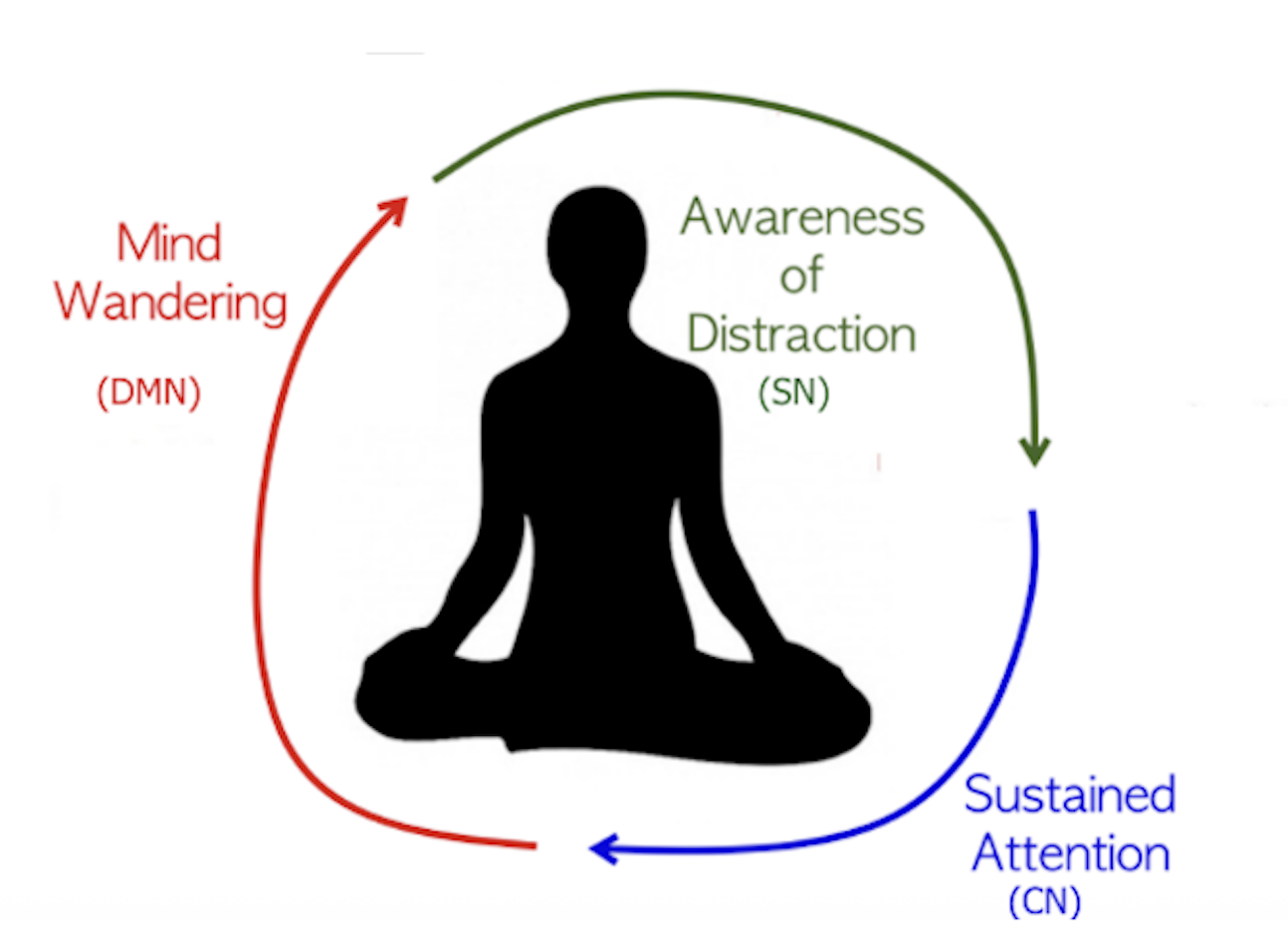}
        \caption{}
    \end{subfigure}
    \hspace{1cm}
    \begin{subfigure}[b]{0.45\textwidth}
        \centering
        \includegraphics[width = \textwidth]{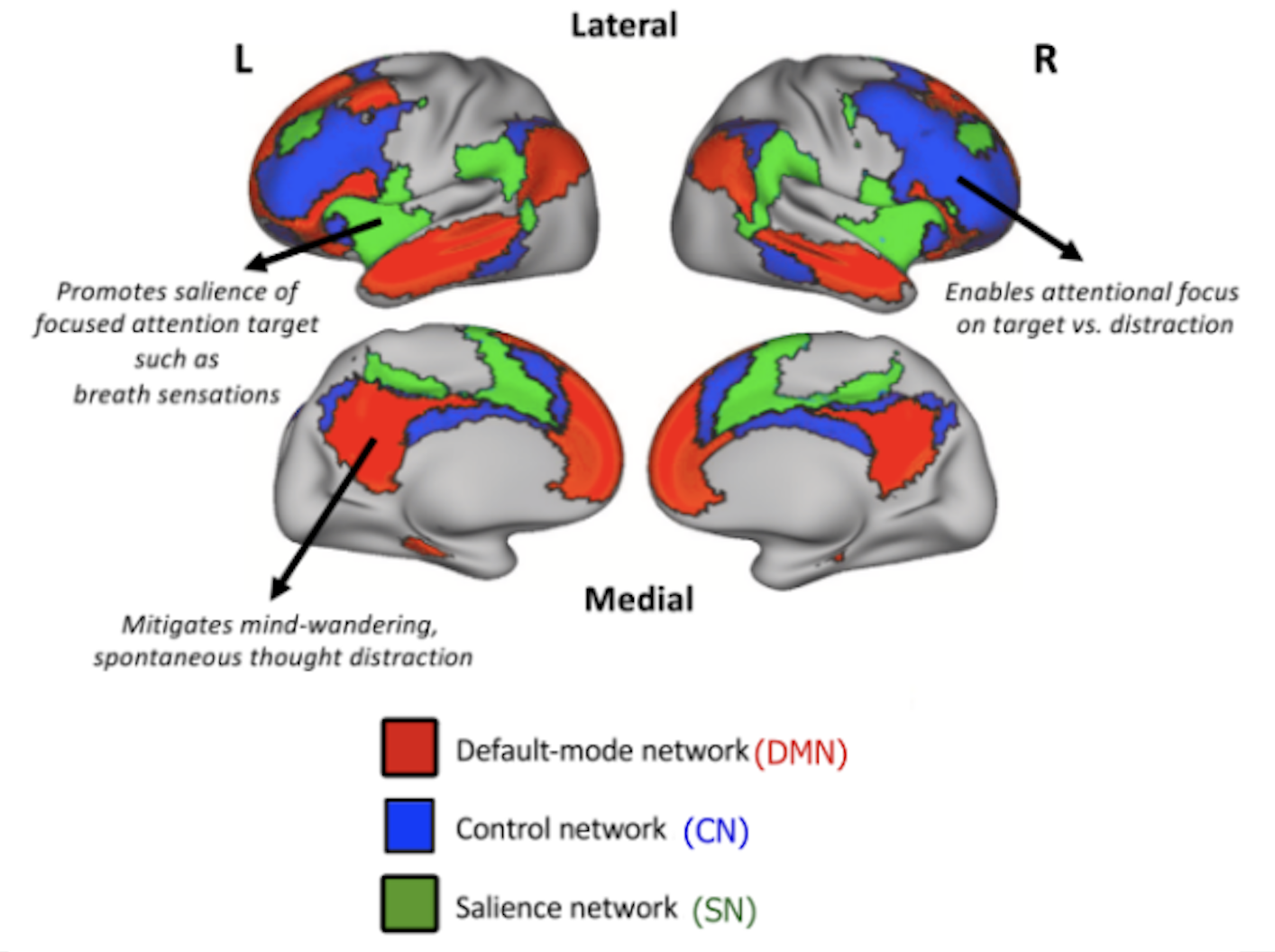}
        \caption{}
    \end{subfigure}
    \caption{(a) Schematic representation of the cognitive phases involved in FAM: sustained attention, mind wandering, and awareness of distraction.
    (b) Corresponding large-scale brain networks typically associated with each phase: the Control Network, the Default Mode Network, and the Salience Network. Schematic representations based on the neurocognitive framework discussed in \cite{ganesan2022focused}.}

    \label{fig:meditation}
\end{figure}

To illustrate our model, we consider the cognitive processes underlying focused attention meditation (FAM). In this practice, the practitioner deliberately directs attention toward a specific object, such as the breath or a visual stimulus, while resisting internal and external distractions. Building on \cite{hasenkamp2013using,hasenkamp2012effects, hasenkamp2012mind}, we adopt a simple three-phase cognitive model in which attention cycles through distinct functional states, each corresponding to the activation of a specific large-scale brain network: sustained attention (associated with the Control Network), mind wandering (associated with the Default Mode Network), and awareness of distraction (associated with the Salience Network). Figure~\ref{fig:meditation} shows a schematic representation of the cognitive phases and their corresponding brain networks.

We encode the sequence of phases in a three-dimensional Lotka--Volterra system, using a parametrization adapted from \cite{horchler2015designing} to ensure that the stable eigenvalues at each saddle are distinct (see Section~\ref{sec:target}). We then train the proposed neural network to approximate this target dynamics. In the simulations, the learned system exhibits recurrent sequential activity and approaches a periodic orbit. Its dynamics also admits an exact realization as the projection of a higher-dimensional discrete neural-field model.

This example is intended as a proof of concept, rather than as a model of a specific neural mechanism for FAM. The projection matrix assigns groups of populations to the three macroscopic variables by construction. The resulting population activity shows that the same recurring sequence at the macroscopic level coexists with heterogeneous temporal profiles of population-level activity patterns.

\subsection{Structure of the paper}
Section~\ref{sec:preliminaries} reviews preliminary results, including heteroclinic dynamics and spatially discrete neural-field equations. Section~\ref{sec:absence} studies the absence of heteroclinic dynamics in the neural-field formulation. It shows that detailed balance excludes heteroclinic cycles and periodic orbits, and gives an example of periodic dynamics without it. It then shows how prescribed activity-pattern equilibria lead to a detailed-balance obstruction, provides a local eigenvalue count at these equilibria, and proves that the exclusion of recurrent transitions between the prescribed states is robust under small perturbations. Section~\ref{sec:UATprocedure} presents our construction for neurally interpretable cyclic and sequential dynamics: we recall a classical $C^1$ universal approximation theorem and adapt it to vector fields of the form $-x+P\sigma(Wx+b)$, realize the reduced dynamics as the exact projection of a higher-dimensional neural system, and establish conditions under which the approximation of an attracting heteroclinic cycle in three dimensions gives rise to an attracting periodic orbit. Section~\ref{sec:FAM} illustrates the application to focused attention meditation, detailing the underlying cognitive process, the chosen target dynamics, the network training, and numerical results.
Technical proofs, additional criteria, and numerical implementation details are provided in the Supplementary Material.

\section{Preliminaries}
\label{sec:preliminaries}
We briefly review the theoretical tools that form the basis of our model.

\subsection{Sequential dynamics and heteroclinic cycles}
\label{sec:heterocline}
Heteroclinic connections provide a natural dynamical framework for describing sequential transitions between metastable states \cite{afraimovich2004heteroclinic,rabinovich2008transient, rabinovich2006dynamical,laurent2001odor,rabinovich2001dynamical, afraimovich2004origin}. Trajectories successively visit neighborhoods of saddle equilibria, where they remain for relatively long intervals, and then move rapidly from one saddle to the next.

Let \(g\in C^1(\mathbb R^n,\mathbb R^n)\), and denote by \(\varphi_t^g\) its flow. We fix a positively invariant set \(\mathcal X\subseteq\mathbb R^n\), which represents the relevant phase space. For example, in population models one typically takes \(\mathcal X=\mathbb R_{\geq0}^n\).

Suppose that \(g\) has \(Q\geq2\) distinct hyperbolic saddle equilibria, denoted by ${\bar{x}_1, \dots, \bar{x}_Q}$, each having unstable index one. Thus the Jacobian \(Dg(\bar x_i)\) has one positive eigenvalue and \(n-1\) eigenvalues with negative real parts. We write
\begin{equation}
    \label{eq:eigenv}
    \lambda_u^i>0, \qquad \lambda_s^i := -\max\bigl\{ \operatorname{Re}\lambda: \lambda\in\operatorname{spec}(Dg(\bar x_i)),\ \operatorname{Re}\lambda<0 \bigr\}>0,
\end{equation}
for the unstable and stable eigenvalue positive real parts, respectively.

The one-dimensional unstable manifold \(W^u(\bar x_i)\setminus\{\bar x_i\}\) has two connected components. We assume that the relevant phase space \(\mathcal X\) contains exactly one of them, which we denote by \(W_+^u(\bar x_i)\). Namely,
\begin{equation}
    \label{eq:selected-unstable-branch}
    W^u(\bar x_i)\cap\mathcal X
    =
    \{\bar x_i\}\cup W_+^u(\bar x_i),
    \qquad i=1,\ldots,Q.
\end{equation}
We then assume that there exists a heteroclinic trajectory
\[
    \Gamma_i^+
    \subset
    W_+^u(\bar x_i)\cap W^s(\bar x_{i+1}),
\]
where indices are understood modulo \(Q\), so that \(\bar x_{Q+1}:=\bar x_1\). Then, the compact invariant set
\begin{equation}
    \label{eq:heteroclinic-cycle}
    \Gamma
    :=
    \bigcup_{i=1}^Q
    \bigl(\{\bar x_i\}\cup\Gamma_i^+\bigr)
\end{equation}
is called a \emph{heteroclinic cycle}. A schematic example is provided in Figure~\ref{fig:heteroclinic-cycle}. The stability of \(\Gamma\) is understood relative to \(\mathcal X\).

\begin{figure}[tb]
    \centering
    \begin{subfigure}[c]{0.42\linewidth}
        \centering
        \begin{tikzpicture}[-latex, auto, node distance=1cm and 2cm, on grid, semithick,
                state/.style={circle, top color=white, bottom color=gray, draw, black, text=black, minimum width=1cm}]

            \node[state] (D) {$\bar{x}_Q$};
            \node[state] (A) [above left=of D] {$\bar{x}_1$};
            \node[state] (B) [above right=of A] {$\bar{x}_2$};
            \node[state] (C) [above right=of D] {$\bar{x}_i$};

            \path (C) edge[bend left=15, dash pattern= on 9pt off 2pt   on 1pt off 2pt on 1pt off 2pt on 1pt off 2pt on 1pt off 2pt on 9pt off 2pt] node[below=0.15cm] {$  $} (D);
            \path (A) edge[bend left=15] node[above] {$\Gamma_1^+$} (B);
            \path (B) edge[bend right=-15, dash pattern= on 9pt off 2pt   on 1pt off 2pt on 1pt off 2pt on 1pt off 2pt on 1pt off 2pt on 9pt off 2pt ] node[below=0.15cm] {$ $} (C);
            \path (D) edge[bend right=-15] node[below=0.15cm] {$\Gamma_Q^+$} (A);

        \end{tikzpicture}
        \caption{}
    \end{subfigure}
    \hfill
    \begin{subfigure}[c]{0.55\linewidth}
        \centering
        \includegraphics[width=\linewidth]{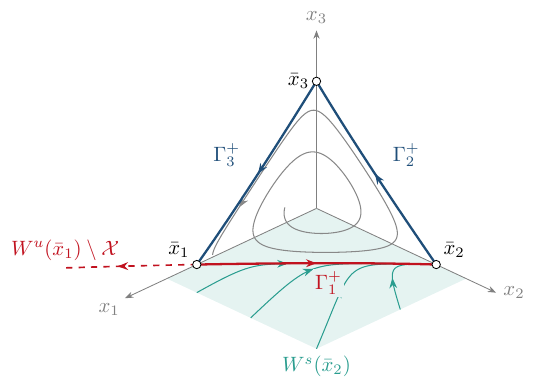}
        \caption{}
    \end{subfigure}
    \caption{(a) Sequential dynamics within a heteroclinic cycle, $\Gamma$, composed of saddle equilibria $\bar{x}_i$ (for $i = 1, \ldots, Q$) and their connecting trajectories $\Gamma_i^+$.
        (b) A heteroclinic cycle $\Gamma$ for $n=Q=3$ and $\mathcal X=\mathbb R^3_{\geq0}$, computed for the Lotka--Volterra system $\dot x_i = x_i(1-x_i-1.3\,x_{i+1}-0.8\,x_{i+2})$, indices modulo $3$.
        The saddles $\bar x_i$ lie on the coordinate axes.
        For $\bar x_1$, the branch $W_+^u(\bar x_1)$ (red) stays in $\partial\mathcal X$, while the other branch of $W^u(\bar x_1)$ (dashed) leaves $\mathcal X$.
        The stable manifold of $\bar x_2$ meets $\mathcal X$ in the face $\{x_3=0\}$ (green, with some of its orbits), which contains $W_+^u(\bar x_1)$, so that $\Gamma_1^+=W_+^u(\bar x_1)\subset W^s(\bar x_2)$.
        A trajectory starting in the interior of $\mathcal X$ (gray) visits the saddles in turn and approaches $\Gamma$.}
    \label{fig:heteroclinic-cycle}
\end{figure}

\begin{definition}
\label{def:relative-asymptotic-stability}
Let $A\subset\mathcal X$ be a compact invariant set. We say that $A$ is \emph{asymptotically stable relative to $\mathcal X$} if:
\begin{enumerate}[label=\textup{(\roman*)}]
    \item for every neighborhood $U$ of $A$ relative to $\mathcal X$, there exists a neighborhood $V$ of $A$ such that $\varphi_t^g(V)\subset U$ for every $t\geq0$;
    \item there exists a neighborhood $V_0$ of $A$ such that $\operatorname{dist}(\varphi_t^g(x),A)\longrightarrow0$ as $t\to+\infty$ for every $x\in V_0$.
\end{enumerate}
\end{definition}

When \(\mathcal X=\mathbb R^n\), this is the usual definition of asymptotic stability.

The local balance between contraction and expansion at \(\bar x_i\) is measured by the saddle value
\begin{equation}
    \label{eq:saddle-value}
    \nu_i
    :=
    \frac{\lambda_s^i}{\lambda_u^i}, \qquad i = 1, \ldots, Q.
\end{equation}

The ratios $\nu_i$ measure the balance between the weakest stable contraction and the unstable expansion at each saddle. Their product,
\[
\nu(\Gamma):=\prod_{i=1}^Q\nu_i,
\]
plays a central role in classical stability criteria for heteroclinic cycles supported by invariant subspaces \cite{krupa1995asymptotic,krupa1997robust}. Under suitable geometric assumptions, the condition $\nu(\Gamma)>1$ ensures that contraction dominates expansion over a complete passage around the cycle. In general, however, stability also depends on the geometry of the connections and the invariant subspaces supporting them. For the Lotka--Volterra system considered in Section~\ref{sec:target}, we will use the classical stability criterion for planar heteroclinic cycles established in \cite[Corollary~2]{hofbauer1994heteroclinic}.

When $\Gamma$ is asymptotically stable relative to $\mathcal X$, trajectories starting sufficiently close to it within $\mathcal X$ remain close and satisfy $\operatorname{dist}(\varphi_t^g(x),\Gamma)\to0$ as $t\to+\infty$. Trajectories following the heteroclinic itinerary successively visit neighborhoods of the saddles in their cyclic order, spending increasingly long intervals near each equilibrium before moving along the next connection.

As illustrated in Figure~\ref{fig:heterocline}(c), the residence time\footnote{Here, the residence time $t_i$ at $\bar{x}_i$ denotes the time interval during which the trajectory remains inside a small neighborhood of the saddle equilibrium before escaping along its unstable manifold.} near a saddle increases as a trajectory approaches the cycle.
Small stochastic perturbations prevent this time from diverging. To leading order, its dependence on the saddle is governed by the inverse unstable rate \(1/\lambda_u^i\), see e.g.~\cite{stone1990random,jeong2023effect}. For deterministic perturbations, the outcome depends on how the heteroclinic connections are broken: perturbations preserving the relevant invariant subspaces may leave the cycle intact, whereas symmetry-breaking perturbations can produce different recurrent dynamics \cite{kirk2008effect}.

\begin{figure}[tb]
    \begin{subfigure}[b]{0.26\textwidth}
        \centering
        \includegraphics[width = 0.9\linewidth]{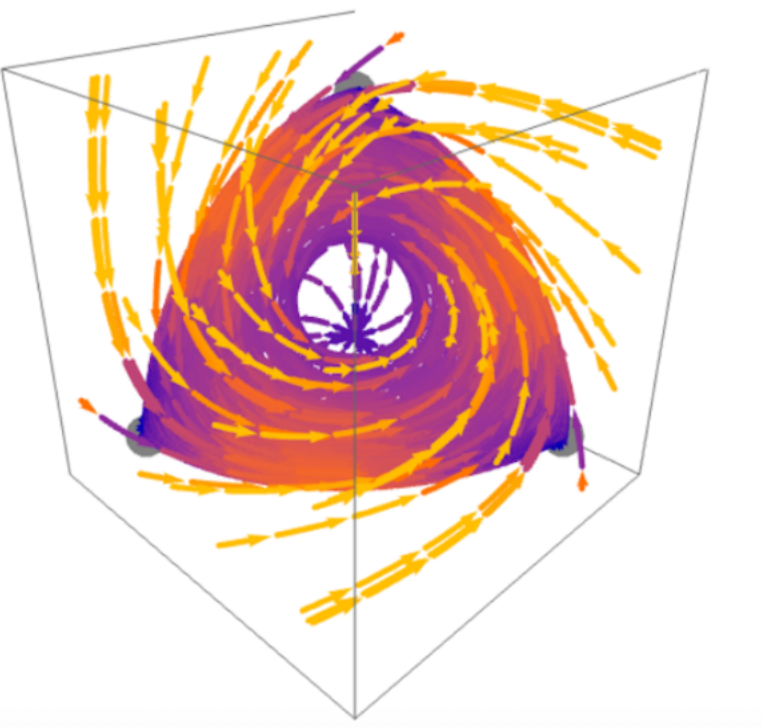}
        \caption{Vector field.}
    \end{subfigure}
    \begin{subfigure}[b]{0.31\textwidth}
        \centering
        \includegraphics[width = 0.8\linewidth]{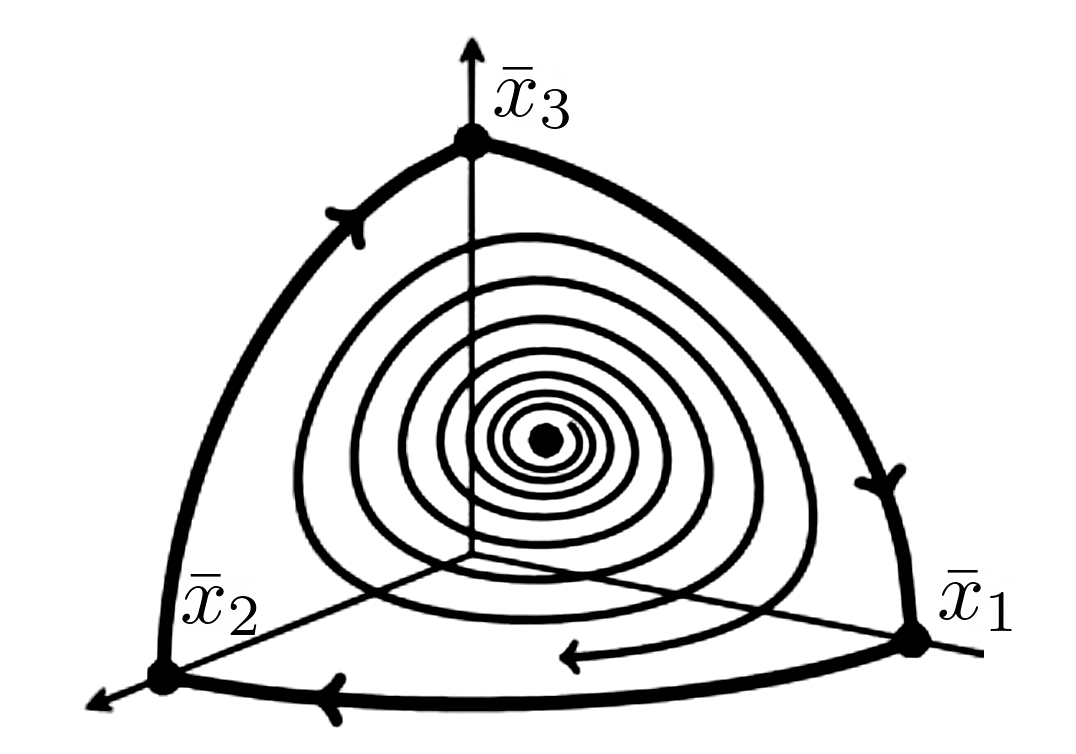}
        \caption{Heteroclinic cycle.}
    \end{subfigure}
    \begin{subfigure}[b]{0.31\textwidth}
        \centering
        \includegraphics[width = \linewidth]{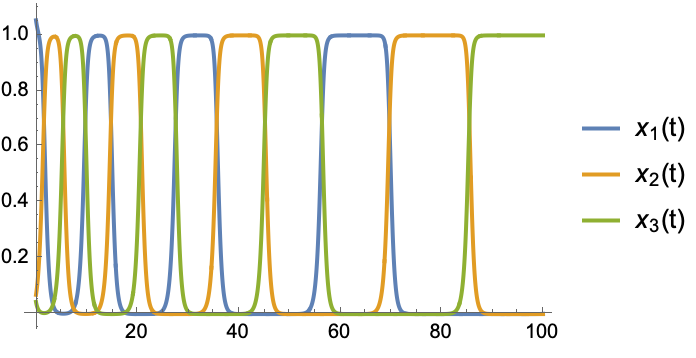}
        \caption{Trajectory evolution. }
    \end{subfigure}
    \caption{Illustration of heteroclinic dynamics in three dimensions: (a) vector field, (b) a trajectory approaching a heteroclinic cycle in phase space, and (c) sequential activations with increasing residence times near the saddle equilibria.}
    \label{fig:heterocline}
\end{figure}

\subsection{Neural-field equations}
\label{ssec:MFeq}
Mean-field equations, first introduced by Wilson and Cowan \cite{wilson1972excitatory} and independently by Amari \cite{amari72}, are widely used to describe the average behavior of interacting neuronal populations.

In this paper, we focus on the spatially discrete case, and the corresponding equation reads
\begin{equation}
    \label{eq:WC}
    \dot x = f(x), \qquad f(x) = -x + \sigma(Wx + b), \qquad x \in \mathbb{R}^n,
\end{equation}
where each component $x_i$ represents the activity of a single population.
The matrix $W \in \mathbb{R}^{n \times n}$ encodes interactions among populations, while $b \in \mathbb{R}^n$ represents external inputs and is assumed to be positive, i.e. $b_i>0$.

The nonlinear activation function $\sigma : \mathbb{R} \to \mathbb{R}$ is applied componentwise and is assumed to satisfy the following properties:
\begin{itemize}
    \item  $\sigma(\mathbb{R}) = (-\alpha, \beta)$ for some $\alpha,\beta>0$, and $\sigma(0) = 0$;
    \item $\sigma \in C^1(\mathbb{R})$ and $\sigma'(s) > 0$ for every $s \in \mathbb{R}$, so that $\sigma$ is strictly increasing.
\end{itemize}
Examples of nonlinearities satisfying these conditions are shown in Figure~\ref{fig:sigmoid}.
\begin{figure}[tb]
    \centering
    % Figure fig:sigmoid: the two example activation functions. Input by main-physicad.tex inside its figure environment.
\begin{tabular}{cc}
\begin{tikzpicture}[scale=1.0]
    % Assi
    \draw[->] (-2.2,0) -- (2.2,0) node[right] {$s$};
    \draw[->] (0,-1.2) -- (0,1.2);

    % Tacche / riferimenti
    \draw[dotted] (-2.2,1) -- (2.2,1);
    \draw[dotted] (-2.2,-1) -- (2.2,-1);
    \node[left] at (0,1) {\scriptsize $1$};
    \node[left] at (0,-1) {\scriptsize $-1$};

    % Curve
    \draw[color=blue, thick, samples=150, domain=-2:2]
        plot (\x,{tanh(\x)});

    % Titolo pannello
    \node at (0,-1.42) {\scriptsize $\sigma(s)=\tanh(s)$};
\end{tikzpicture}
&
\begin{tikzpicture}[xscale=0.65]
    % Assi
    \draw[->] (-3.4,0) -- (3.4,0) node[right] {$s$};
    \draw[->] (0,-1.2) -- (0,1.2);

    % Tacche / riferimenti
    \draw[dotted] (-3.4,1) -- (3.4,1);
    \draw[dotted] (-3.4,-0.582) -- (3.4,-0.582);
    \node[left] at (0,1) {\scriptsize $1$};
    \node[left] at (-3.4,-0.582) {\scriptsize $-\tfrac{1}{e-1}$};

    % Curve
    \draw[color=red!70!black, thick, samples=150, domain=-3.3:3.3]
        plot (\x,{(exp(1 - exp(-\x)) - 1)/(e - 1)});

    % Titolo pannello
    \node at (0,-1.42) {\scriptsize $\sigma(s)=(e^{1-e^{-s}}-1)/(e-1)$};
\end{tikzpicture}
\end{tabular}
    \caption{Two examples of activation functions considered in this paper. On the left, \(\sigma(\mathbb{R})=(-1,1)\), so that \(\alpha=\beta=1\). On the right, a shifted Gompertz function, \(\sigma(\mathbb{R})=(-\tfrac{1}{e-1},1)\), so that \(\alpha=\tfrac{1}{e-1}\) and \(\beta=1\).}
    \label{fig:sigmoid}
\end{figure}

\begin{remark}
    The condition $\sigma(0)=0$ is a normalization. Replacing $\sigma$ by $s\mapsto \sigma(s+s_0)-\sigma(s_0)$ for some $s_0\in\mathbb R$ and translating the state variable accordingly preserves the form of \eqref{eq:WC}, up to a modification of the bias $b$. More precisely, setting $c:=\sigma(s_0)$ and $y=x-c\mathbf 1$, the transformed bias is $\widetilde b=b+cW\mathbf 1-s_0\mathbf 1$. In particular, the standing assumption $\widetilde b_i>0$ has to hold after this transformation.
\end{remark}

\begin{remark}
    Since $\alpha>0$, the components $x_i$ measure activity relative to a baseline and are allowed to be negative. The equilibria prescribed in Section~\ref{sec:absence} are positive, and the case study of Section~\ref{sec:FAM} takes place in the positive octant, but nothing in the analysis of Section~\ref{sec:absence} requires $x$ to be positive.
\end{remark}

We conclude this section by stating a boundedness property related to \eqref{eq:WC}. In particular, the following lemma provides a compact invariant region containing every complete and bounded orbit, and thus, in particular, every periodic orbit or heteroclinic cycle.

\begin{lemma}
    \label{lem:bounded}
    All forward orbits of \eqref{eq:WC} are bounded, and the set $(-\alpha, \beta)^n$ is positively invariant  and globally attractive for the dynamics. Write $\|W\|_\infty = \max_i\sum_j|W_{ij}|$ and $\|b\|_\infty = \max_i|b_i|$, and let
    \begin{equation}
        \label{eq:K-W}
        M_W := \|W\|_\infty \max(\alpha,\beta) + \|b\|_\infty ,
        \qquad
        K_W := \big[\sigma(-M_W),\, \sigma(M_W)\big]^n .
    \end{equation}
    Then $K_W$ is a compact convex subset of $(-\alpha,\beta)^n$, it is positively invariant and globally attracting. Moreover, every orbit of \eqref{eq:WC} that is defined and bounded on the whole of $\mathbb{R}$ is contained in $K_W$.
\end{lemma}

\begin{proof}
    Since $\sigma(\mathbb R)=(-\alpha,\beta)$, we have $\dot x_i<0$ whenever $x_i\geq\beta$ and $\dot x_i>0$ whenever $x_i\leq-\alpha$. Thus all forward orbits are bounded, and $(-\alpha,\beta)^n$ is positively invariant. For any fixed forward orbit, boundedness gives a constant $M>0$ such that $\|Wx(t)+b\|_\infty\leq M$ for all $t\geq0$. Consequently, $\dot x_i\leq\sigma(M)-\beta<0$ whenever $x_i\geq\beta$, and $\dot x_i\geq\sigma(-M)+\alpha>0$ whenever $x_i\leq-\alpha$. Hence every forward orbit enters $(-\alpha,\beta)^n$ in finite time. Since $\sigma$ takes values in $(-\alpha,\beta)$, the set $K_W$ is a compact convex subset of $(-\alpha,\beta)^n$, and $|(Wx+b)_i|\leq M_W$ for every $x\in[-\alpha,\beta]^n$.

    Let now $x$ be any trajectory of \eqref{eq:WC}. Since $(-\alpha, \beta)^n$ is attractive and positively invariant, we can assume without loss of generality that $x(t) \in (-\alpha, \beta)^n$ for all $t \ge s$, for some $s$. The variation of constants formula gives
    \begin{equation}
        \label{eq:duhamel}
        x_i(t) = e^{-(t-s)} x_i(s) + \int_s^t e^{-(t-\tau)} \sigma\big((Wx(\tau)+b)_i\big)\, d\tau, \qquad \forall t \ge s .
    \end{equation}
    For \(t\ge s\), the integrand takes values in $[\sigma(-M_W),\sigma(M_W)]$.
    Since $\int_s^t e^{-(t-\tau)}\,d\tau = 1-e^{-(t-s)},$ the integral term in \eqref{eq:duhamel} belongs to
    \((1-e^{-(t-s)})I_W\), where $I_W:=[\sigma(-M_W),\sigma(M_W)]$. Hence
    \[
    \operatorname{dist}\bigl(x_i(t),I_W\bigr)
    \le
    e^{-(t-s)}
    \operatorname{dist}\bigl(x_i(s),I_W\bigr).
    \]
    Therefore $\operatorname{dist}\bigl(x(t),K_W\bigr)\longrightarrow0$ as $t\to+\infty$, so \(K_W\) is globally attracting.
    The same reasoning shows that \(K_W\) is positively invariant: if \(x(s)\in K_W\), then both \(x_i(s)\) and the normalized integral average in \eqref{eq:duhamel} belong to \(I_W\), and therefore \(x_i(t)\in I_W\) for every \(t\ge s\).

    Assume finally that $x$ is defined and bounded on the whole of $\mathbb{R}$, and put $u_i := x_i-\beta$. Since $\sigma(s)<\beta$ for all $s\in \mathbb{R}$,
    \[
        \frac{d}{dt}\big(e^t u_i(t)\big) = e^t\Big(\sigma\big((Wx(t)+b)_i\big)-\beta\Big) < 0 ,
    \]
    so that $u_i(t) < e^{-(t-s)}u_i(s)$ for $s<t$. Letting $s \to -\infty$ and using that $u_i$ is bounded gives $x_i(t) \le \beta$, and the symmetric argument with $\sigma > -\alpha$ gives $x_i(t) \ge -\alpha$. Therefore $|(Wx(\tau)+b)_i| \le M_W$ for every $\tau \in \mathbb{R}$, and letting $s\to-\infty$ in~\eqref{eq:duhamel} exhibits $x_i(t)$ as an average of values of $\sigma$ in $[\sigma(-M_W),\sigma(M_W)]$ against a probability density. Hence $x(t) \in K_W$ for every $t \in \mathbb{R}$.
\end{proof}

\section{Absence of heteroclinic cycles in neural-field equations}
\label{sec:absence}
Motivated by modeling sequential neural activity across interacting populations, we study whether neural-field equations of the form~\eqref{eq:WC} can exhibit heteroclinic cycles or periodic orbits that would support cyclic-sequential neural activity patterns.

\subsection{Detailed balance and Lyapunov structures}

We first recall that, given an open positively invariant set $U\subseteq\mathbb{R}^n$, a \emph{strict Lyapunov function} for the system $\dot x=f(x)$ is a $C^1$ function $V:U\to\mathbb{R}$ such that $\frac{d}{dt}V(x(t))\le0$ along every trajectory $x(t)$ in $U$, with $\frac{d}{dt}V(x(t)) = 0$ if and only if $f(x(t))=0$.

We will use the following classical consequence.

\begin{proposition}
    \label{prop:lyapunov}
    If a dynamical system admits a strict Lyapunov function on an open positively invariant set $U$, then it supports no heteroclinic cycle and no periodic orbit contained in $U$.
\end{proposition}

\begin{proof}
Suppose that there exists a heteroclinic cycle connecting the equilibria $\bar x_1,\ldots,\bar x_m$, with indices taken modulo $m$. Since $V$ is strictly decreasing along every nonconstant connecting orbit, we have $V(\bar x_{i+1})<V(\bar x_i)$ for every $i$. Summing these inequalities gives
\[
0=\sum_{i=1}^m\bigl(V(\bar x_{i+1})-V(\bar x_i)\bigr)<0,
\]
a contradiction. Similarly, $V$ cannot decrease strictly along a nonconstant periodic orbit and return to its initial value.
\end{proof}

We now identify a structural condition on the connectivity matrix of the system~\eqref{eq:WC} that guarantees the existence of such a Lyapunov function.

\begin{definition}
    The connectivity matrix $W$ is said to satisfy the \emph{detailed balance condition} if there exist $\tau_1,\ldots, \tau_n>0$ for all $i=1,\ldots,n$ such that
    \begin{equation}
        \label{eq:detailed-balance}
        \tau_i W_{ij} = \tau_j W_{ji}, \qquad \forall i,j = 1, \ldots, n.
    \end{equation}
    Equivalently, letting $T = \diag(\tau_1, \ldots, \tau_n)$, the matrix $T W$ is symmetric.
\end{definition}

Condition~\eqref{eq:detailed-balance} expresses the diagonal symmetrizability of $W$ and is closely related to the classical Lyapunov constructions for neural networks \cite{cohen1983absolute,hopfield1984neurons}.

\begin{theorem}
    \label{thm:lyapunov}
    Consider the system~\eqref{eq:WC} with a connectivity matrix $W$ satisfying the detailed balance condition~\eqref{eq:detailed-balance}. Then,
    \begin{equation}
        \label{eq:lyapunov}
        V(x) = \sum_{i=1}^n \tau_i \int_0^{x_i} \sigma^{-1}(s) ds - \frac{1}{2} x^\top TW x - x^\top T b,
    \end{equation}
    is a strict Lyapunov function for \eqref{eq:WC} on the set $(-\alpha, \beta)^n$.
    In particular, the system can support neither heteroclinic cycles nor periodic orbits.
\end{theorem}

\begin{proof}
    Since $\sigma^{-1}$ is defined on $\sigma(\mathbb{R}) = (-\alpha,\beta)$, the function $V$ is well defined exactly on $(-\alpha,\beta)^n$, which is open and positively invariant by Lemma~\ref{lem:bounded}.
    Since $TW$ is symmetric, differentiating \eqref{eq:lyapunov} yields
    \begin{equation}
        \nabla V(x) = T \sigma^{-1}(x) - TW x - Tb.
    \end{equation}
    The continuity of $\sigma^{-1}$ then implies that $V\in C^1$.

    Observe that, by the mean value theorem, for any $j=1,\ldots,n$ there exists $\xi_j(x)$ between $\sigma^{-1}(x_j)$ and $(Wx+b)_j$ such that
    \[
        \begin{split}
            f_j(x)
             & = -\sigma\left( \sigma^{-1}(x_j) \right)+\sigma\left((W x + b)_j\right)  \\
             & =\sigma'(\xi_j(x))\left((W x + b)_j - \sigma^{-1}(x_j)\right)            \\
             & = -\frac{\sigma'(\xi_j(x))}{\tau_j} \frac{\partial V}{\partial x_j}(x) .
        \end{split}
    \]
    Recalling that $\sigma$ is increasing, it follows that if $x(t)$ is a trajectory of \eqref{eq:WC}, we have
    \[
        \frac{d}{dt} V(x(t)) = \nabla V(x(t))^\top \dot x(t) = -\sum_{j=1}^n \frac{\sigma'(\xi_j(x(t)))}{\tau_j} \left(\frac{\partial V}{\partial x_j}(x(t))\right)^2 \le 0,
    \]
    with equality if and only if $\nabla V(x(t)) = 0$, which is equivalent to $f(x(t)) = 0$ and thus that $x(t)$ is an equilibrium of \eqref{eq:WC}.

    Finally, note that heteroclinic cycles and periodic orbits are unions of complete bounded orbits, so the last assertion of Lemma~\ref{lem:bounded} places any of them inside $K_W\subset(-\alpha,\beta)^n$. The last claim thus follows from Proposition~\ref{prop:lyapunov}, applied with $U = (-\alpha,\beta)^n$.
\end{proof}

\begin{remark}
    \label{rem:attribution}
    Setting $y=Wx+b$, any solution of \eqref{eq:WC} gives a solution of $\dot y=-y+W\sigma(y)+b.$
    Multiplying the $i$-th equation by $\tau_i$, we obtain
    \[
    \tau_i\dot y_i=\sum_{j=1}^n \tau_iW_{ij}\sigma(y_j)-\tau_i y_i+\tau_i b_i.
    \]
    This is the graded-response model considered by Hopfield \cite{hopfield1984neurons}, with capacitances $C_i=\tau_i$, resistances $R_i=1/\tau_i$, interaction coefficients $J_{ij}=\tau_iW_{ij}$, and external inputs $I_i=\tau_i b_i$. Thus the detailed balance condition \eqref{eq:detailed-balance} is the symmetry of the interaction matrix $J$ required in Hopfield's Lyapunov construction.

    On the other hand, the detailed balance condition~\eqref{eq:detailed-balance} is closely related to the Lyapunov construction of Cohen and Grossberg \cite{cohen1983absolute}. Indeed, their system specializes formally to the above equation by taking, in their notation,
    \[
        a_i \equiv 1/\tau_i, \quad
        b_i(y) = \tau_i(b_i - y), \quad
        c_{ik} = -\tau_i W_{ik}, \quad
        d_k = \sigma.
    \]
    With this identification, their symmetry condition $c_{ik}=c_{ki}$ is precisely~\eqref{eq:detailed-balance}.
    Observe that this correspondence concerns the Lyapunov function alone, whose monotonicity only requires the symmetry of $(c_{ik})$ together with $a_i\ge0$ and $d_k'\ge0$. The main convergence theorem in \cite{cohen1983absolute} further assumes $c_{ik}\ge0$ and restricts to the positive orthant, neither of which is needed here.
    We give the proof directly in the $x$ variables because it is short, and because it does not require $W$ to be invertible. %
\end{remark}

We have therefore shown that, independently of any prescribed equilibrium structure, a discrete neural-field system of the form~\eqref{eq:WC} whose connectivity matrix satisfies the detailed balance condition cannot support heteroclinic cycles or periodic orbits.
Related constructions of sequential dynamics have been obtained in other neural-network formulations. In particular, Chossat and Krupa studied heteroclinic cycles in regularized Hopfield-type systems with invariant coordinate faces \cite{chossat2016heteroclinic,chossat2017consecutive}, while threshold-linear networks provide another framework for recurrent neural dynamics and sequential attractors \cite{hahnloser2003permitted,parmelee2022sequential}.
Cyclic behavior is not excluded by the neural-field model~\eqref{eq:WC} itself, however. Indeed, a three-dimensional system outside detailed balance can support sustained periodic oscillations.

\begin{theorem}
    \label{thm:ring-main}
    For $n = 3$, let $b = (b_0,b_0,b_0)$ with $b_0>0$ and, for $c>0$,
    \begin{equation}
        \label{eq:ring-W}
        W =
        \begin{pmatrix}
            0  & 0  & -c \\
            -c & 0  & 0  \\
            0  & -c & 0
        \end{pmatrix}.
    \end{equation}
    Then there exists $c_0>0$ such that, for every $c>c_0$, the $\omega$-limit set of every trajectory of~\eqref{eq:WC} starting outside the diagonal $\{x_1=x_2=x_3\}$ is a nonconstant periodic orbit.
\end{theorem}

\begin{remark}
    Detailed balance fails, since $W_{21}=-c\neq0=W_{12}$, which is
    incompatible with $\tau_2W_{21}=\tau_1W_{12}$ for positive
    $\tau_1,\tau_2$.
\end{remark}

We present the proof of Theorem~\ref{thm:ring-main} in Section~\ref{sec:cyclic} of the supplementary material, together with a more precise statement, a discussion of the stability of the periodic orbits and a numerical illustration. This periodic behavior is not sequential activity in the sense of Section~\ref{sec:heterocline}, since the system has a unique equilibrium and therefore no heteroclinic connections between distinct saddle equilibria.

\subsection{Independent equilibria corresponding to neural activity patterns}

We now ask whether prescribing a family of equilibria representing neural activity patterns can force the connectivity matrix to satisfy detailed balance.

We begin with a general observation. Let $ \bar X=[\bar x_1,\ldots,\bar x_n]$ be the matrix whose columns are $n$ linearly independent prescribed equilibria of~\eqref{eq:WC}. Imposing these equilibria a priori forces the connectivity matrix $W$ to have a specific structure. Indeed, the equilibrium condition can be written compactly as
\[
    \bar X = \sigma(W \bar X + b \mathbf{1}^\top),
\]
where $\mathbf{1} = (1, \ldots, 1)^\top$.
From this relation, we obtain
\begin{equation}
    \label{ref:W-structure}
    W = A - b u^\top, \quad
    A = \sigma^{-1}(\bar X)\bar X^{-1},
    \quad
    u^\top = \mathbf{1}^\top \bar X^{-1}.
\end{equation}

We now specialize to equilibria that represent neural activity patterns. In each pattern, the activity of population $i$ is either zero or a fixed nonzero level $a_i$, which is the same across all patterns. The following theorem, the main result of this section, gives conditions under which these prescribed patterns force the connectivity matrix to satisfy detailed balance.

\begin{theorem}
    \label{thm:main}
    Consider the system~\eqref{eq:WC} with $n$ linearly independent equilibrium points $\bar{x}_i$.
    Assume that there exist $a_i\in (-\alpha,\beta)$, $a_i\neq 0$, for $i=1,\ldots,n$, such that
    \begin{equation}
        \label{eq:levels}
        \bar X_{ij} \in \{ 0, a_i\} \qquad \forall j = 1, \ldots, n,
    \end{equation}
    and that the vector $u$ of~\eqref{ref:W-structure} has positive entries,
    \begin{equation}
        \label{eq:u-positive}
        u^\top = \mathbf 1^\top \bar X^{-1} > 0 .
    \end{equation}
    Then, the connectivity matrix $W$ satisfies the detailed balance condition~\eqref{eq:detailed-balance}. In particular, the system~\eqref{eq:WC} admits the strict Lyapunov function~\eqref{eq:lyapunov}, and it supports neither heteroclinic cycles nor periodic orbits.
\end{theorem}

\begin{proof}
    Once detailed balance is established, every assertion follows from Theorem~\ref{thm:lyapunov}, so we only have to verify~\eqref{eq:detailed-balance}.
    To this end, we start by observing that $\sigma^{-1}(\bar X)=D\bar X$ where $D=\diag_i\left(\sigma^{-1}(a_i)/a_i\right)$.
    In particular, by \eqref{ref:W-structure} we have
    \[
        W = D - b u^\top.
    \]
    Finally, since $W_{ij}=-b_iu_j$ for $i\neq j$, condition \eqref{eq:detailed-balance} is satisfied by taking $\tau_i = u_i/b_i$ for $i=1,\ldots,n$.
    Here, $\tau_i>0$ by assumption~\eqref{eq:u-positive}, and the fact that $b_i>0$.
\end{proof}

\begin{remark}
\label{rem:rabinovich}
Consider the three-cluster Wilson--Cowan system discussed in \cite{rabinovich2008transient},
\[
\dot A_j=-A_j+H\left(-\sum_{i=1}^3\rho_{ji}A_i+b_0\right),\qquad b_0>0,
\]
where $H$ is a sigmoidal activation function. The authors describe this system as exhibiting a stable limit cycle within a stable heteroclinic channel, associated with the sequential activation of three neural modes. However, Theorem~\ref{thm:main} reveals a structural obstruction to this scenario when the underlying neural modes are identified with equilibria on the coordinate axes.
More precisely, assume that $H\in C^1(\mathbb R)$ is bounded with $H'>0$, that the coupling is inhibitory ($\rho_{ji}>0$ for $j\neq i$), and that the system admits the equilibria $\bar A_i=a_i e_i$ with $a_i>0$.
Then, the equilibrium conditions force $H$ to vanish at a point $s_0<b_0$, and thus the system above can be put in the form~\eqref{eq:WC} with $\sigma(s)=H(s+s_0)$, satisfying the assumptions of Section~\ref{ssec:MFeq}.
Theorem~\ref{thm:main} therefore applies, and excludes both heteroclinic cycles and nonconstant periodic orbits.
\end{remark}

Condition~\eqref{eq:levels} allows several populations to be active simultaneously in a prescribed pattern (equilibrium). The special case in which each pattern activates exactly one population is $\bar X_{ij}=a_i\delta_{ij},$ so that $\bar X=\operatorname{diag}(a_1,\ldots,a_n)$.

\begin{remark}
    The crucial point of the proof is that $\sigma^{-1}(\bar X)=M\bar X$ for a diagonal matrix $M$. Thus, the assumption could be relaxed: it is enough that, for each row $i$, the ratio ${\sigma^{-1}(\bar X_{ij})}/{\bar X_{ij}}$ be independent of $j$ whenever $\bar X_{ij}\neq0$.
\end{remark}

\begin{remark}
    \label{rem:u-positive}
    Condition~\eqref{eq:u-positive} cannot be dropped.
    For example, for any nonzero $a\in (-\alpha,\beta)$, the equilibria $a(1,0,0)$, $a(1,1,1)$, $a(0,0,1)$ of $\mathbb{R}^3$ satisfy~\eqref{eq:levels} and are linearly independent, but give $u = (1,-1,1)/a$, and the connectivity that they determine is not in detailed balance. Indeed, $W_{12}=b_1/a$ and $W_{21}=-b_2/a$ have opposite signs, so no positive $\tau_1,\tau_2$ can satisfy~\eqref{eq:detailed-balance}.
\end{remark}

A complementary local obstruction can be obtained under slightly more restrictive assumptions on the activation function $\sigma$.
The relevant quantity is the set
\begin{equation}
    \label{eq:support}
    S_i := \{ k : \bar X_{ki} = a_k \},
\end{equation}
namely the set of populations activated by the $i$-th pattern.
In the single-population case described above,
\begin{equation}
    \label{eq:equilibria}
    \bar x_i = a_i e_i, \qquad a_i > 0,
\end{equation}
and hence $S_i=\{i\}$. The bound $a_i<\beta$ is automatic, since the equilibrium condition $\bar x_i=\sigma(W\bar x_i+b)$ together with $\sigma(\mathbb R)=(-\alpha,\beta)$ implies it. The theorem below, in which $|S_i|$ denotes the cardinality of $S_i$, bounds the dimensions of the invariant manifolds at each $\bar x_i$ in terms of that number.

\begin{theorem}
    \label{thm:local}
    Consider the system~\eqref{eq:WC} with $n$ equilibrium points $\bar x_i$ satisfying the assumptions of Theorem~\ref{thm:main}.
    Assume that the activation function $\sigma$ satisfies the assumptions of Section~\ref{ssec:MFeq}, and in addition that
    \[
        \sigma'(s)<\frac{\sigma(s)}{s}<\sigma'(0) \qquad \text{for every } s\in\mathbb{R}\setminus\{0\}.
    \]
    Then, for every $i=1,\ldots,n$, the Jacobian $Df(\bar x_i)$ has real eigenvalues, at least $|S_i|$ of which are negative and at least $n-|S_i|-1$ of which are positive. In particular $\dim\mathcal W^s(\bar{x}_i)\ge |S_i|$ and $\dim\mathcal W^u(\bar{x}_i) \ge n-|S_i|-1$.
\end{theorem}

We present the proof of Theorem~\ref{thm:local} in Section~\ref{app:counting-eigenvalues} of the supplementary material, for completeness. The additional assumption on $\sigma$ holds for both activation functions shown in Figure~\ref{fig:sigmoid}. Whenever $n\ge|S_i|+3$, the theorem gives $\dim\mathcal W^u(\bar x_i)\ge 2$, so $\bar x_i$ is not a codimension-one saddle and condition~\eqref{eq:eigenv} fails. If this holds for every $i$, the heteroclinic cycles introduced in Section~\ref{sec:heterocline} are therefore excluded. In the axis case~\eqref{eq:equilibria} each pattern activates a single population and the requirement is simply $n\ge4$. This local argument says nothing about heteroclinic connections in a broader sense, between equilibria with higher-dimensional unstable manifolds. Excluding heteroclinic cycles built from such connections requires the global argument of Theorem~\ref{thm:main}.

\subsection{Robustness}
\label{sec:robust}

The obstruction of Theorem~\ref{thm:main} relies on the detailed balance condition~\eqref{eq:detailed-balance}, which is not generic and can be destroyed by arbitrarily small perturbations of the connectivity matrix. It is therefore natural to ask whether cyclic sequential dynamics can reappear under small perturbations of $W$.

Let $\mathcal E$ be the set of equilibria of~\eqref{eq:WC}, which is a closed subset of $K_W$ by Lemma~\ref{lem:bounded}, hence compact. For $S \subseteq \mathbb{R}^n$ and $\rho>0$ we write $B(S,\rho) := \{x \in \mathbb{R}^n : d(x,S) < \rho\}$, where $d$ is the Euclidean distance.

\begin{theorem}
    \label{thm:robust}
    Let $W$ satisfy the detailed balance condition~\eqref{eq:detailed-balance}, fix a choice of positive coefficients $\tau_1,\ldots,\tau_n$ satisfying that condition,and let $V$ be the corresponding Lyapunov function~\eqref{eq:lyapunov}. Assume that $V(\mathcal E)$ is finite. Then, for every $r>0$, there exists $\varepsilon > 0$ such that, for every $W'$ with $\|W'-W\|_\infty < \varepsilon$, every heteroclinic cycle, every homoclinic orbit and every periodic orbit of the system
    \begin{equation}
        \label{eq:WC-perturbed-W}
        \dot x = f'(x), \qquad f'(x) = -x + \sigma(W'x+b),
    \end{equation}
    is contained in a single connected component of $B(\mathcal E, r)$.
\end{theorem}

The finiteness of $V(\mathcal E)$ is a mild assumption. It holds whenever the equilibria of~\eqref{eq:WC} are isolated, because $\mathcal E$ is compact and therefore finite, and it holds with no isolation assumption whenever $\sigma$ is real-analytic, as the following corollary shows.

\begin{corollary}
    \label{cor:no-sequential}
    Assume that $\sigma$ is real-analytic and the equilibria $\bar x_1,\ldots,\bar x_n$ satisfy the assumptions of Theorem~\ref{thm:main}.
    Then, for every $r>0$, there exists $\varepsilon > 0$ such that, for every $W'$ with $\|W'-W\|_\infty < \varepsilon$, every heteroclinic cycle, every homoclinic orbit and every periodic orbit of~\eqref{eq:WC-perturbed-W} is contained in a single connected component of $B(\mathcal E, r)$.
    If, in addition, the prescribed equilibria are nondegenerate\footnote{Recall that an equilibrium $\bar x$ is nondegenerate if $Df(\bar x)$ is invertible. This assumption is discussed in Section~\ref{app:nondegeneracy} of the supplementary material, where it is turned into an explicit condition on $\bar X$ and $b$.}, then, for $r$ sufficiently small, no such object  comes within $r$ of more than one prescribed state, so the prescribed states cannot be visited sequentially.
\end{corollary}

\begin{proof}
    By the proof of Theorem~\ref{thm:main}, the connectivity~\eqref{ref:W-structure} determined by the prescribed states satisfies the detailed balance condition~\eqref{eq:detailed-balance}, so that $V$ is a strict Lyapunov function for~\eqref{eq:WC}.

    We first check that $V(\mathcal E)$ is finite. Since $\sigma$ is real-analytic and $\sigma'>0$, the inverse $\sigma^{-1}$ is real-analytic on $(-\alpha,\beta)$, hence so is $V$. The \L ojasiewicz gradient inequality \cite{lojasiewicz1963propriete} provides, around every $p \in \mathcal E$, a neighborhood $U_p$ together with constants $c_p>0$ and $\mu_p \in [1/2,1)$ such that $|\nabla V| \ge c_p\,|V-V(p)|^{\mu_p}$ on $U_p$. Every equilibrium lying in $U_p$ therefore satisfies $V = V(p)$, and finitely many such neighborhoods cover the compact set $\mathcal E$. Theorem~\ref{thm:robust} applies and yields the first assertion.

    For the second assertion, we now use the nondegeneracy of the prescribed equilibria. $\mathcal E$ is a compact real-analytic set, so it has finitely many connected components $\mathcal E_1, \ldots, \mathcal E_\ell$ \cite{lojasiewicz1965ensembles}. Each $\bar x_i$ is a nondegenerate zero of $f$, hence an isolated point of $\mathcal E$ by the inverse function theorem, and therefore a connected component on its own. Let $r<\min_{i\ne j}d(\mathcal E_i,\mathcal E_j)/3$. Each $B(\mathcal E_k,r)$ is connected, being a union of open balls all of which meet the connected set $\mathcal E_k$. The choice of $r$ makes the sets $B(\mathcal E_1,r), \ldots, B(\mathcal E_\ell,r)$ pairwise disjoint, and they are open and cover $B(\mathcal E,r)$, so they are its connected components. Two distinct prescribed states lie in two distinct components of $\mathcal E$, hence in two disjoint members of this family, and the first assertion forbids any of the objects considered to meet both.
\end{proof}

In terms of the sequential dynamics considered in Section~\ref{sec:heterocline}, this corollary shows that small perturbations cannot generate cyclic trajectories visiting the prescribed states one after another.

We now prove Theorem~\ref{thm:robust}.

\begin{proof}[Proof of Theorem~\ref{thm:robust}]
    We may assume $\varepsilon\leq1$, so that $\|W'-W\|_\infty<1$ throughout.
    The perturbed system~\eqref{eq:WC-perturbed-W} is of the form~\eqref{eq:WC}, so Lemma~\ref{lem:bounded} applies to it. Put
    \begin{equation}
        \label{eq:Kstar}
        M^\ast := M_W + \max(\alpha,\beta),
        \qquad
        K^\ast := \big[\sigma(-M^\ast),\, \sigma(M^\ast)\big]^n ,
    \end{equation}
    with $M_W$ as in~\eqref{eq:K-W}. Then $K^\ast$ is a compact convex subset of $(-\alpha,\beta)^n$ containing $K_W$, and hence $\mathcal E$. Since $\|W'\|_\infty \le \|W\|_\infty+1$, we have $M_{W'}\le M^\ast$ and $K_{W'}\subseteq K^\ast$. Equilibria, homoclinic orbits, periodic orbits and the orbits composing a heteroclinic cycle are all defined and bounded on the whole of $\mathbb{R}$, so the last assertion of Lemma~\ref{lem:bounded} places every one of them inside $K^\ast$.

    Write $v_1 < \cdots < v_q$ for the values taken by $V$ on $\mathcal E$ and
    \[
        E_i := \mathcal E \cap \{V = v_i\}, \qquad i = 1, \ldots, q,
    \]
    for the associated \emph{level pieces}, which are compact and pairwise disjoint. Write also $\varepsilon' := \max_{K^\ast}|f'-f|$. Since $f'-f = \sigma(W'x+b) - \sigma(Wx+b)$ and $\sigma$ is continuous, $\varepsilon' \to 0$ as $\|W'-W\|_\infty \to 0$, so it suffices to make $\varepsilon'$ small.

    If $K^\ast\subseteq B(\mathcal E,r/2)$, then every object considered lies in $B(\mathcal E,r)$ and, being connected, in a single connected component of it, so the conclusion of the theorem is immediate. We may therefore assume that $K^\ast\setminus B(\mathcal E,r/2)$ is nonempty.

    \emph{Step 1.}
    For $\rho>0$ set
    \[
        \kappa(\rho) := -\max\big\{ \nabla V(x)\cdot f(x) \ : \ x \in K^\ast\setminus B(\mathcal E,\rho) \big\} .
    \]
    The set $K^\ast \setminus B(\mathcal E,\rho)$ is compact and carries no equilibrium, so $\kappa(\rho)>0$ by Theorem~\ref{thm:lyapunov}. Since $|\nabla V|$ is bounded on $K^\ast$ and $|\nabla V \cdot f' - \nabla V \cdot f| \le |\nabla V|\,\varepsilon'$, there is $\varepsilon_\rho>0$ such that
    \begin{equation}
        \label{eq:off-E}
        \nabla V \cdot f' \ \le \ -\tfrac12 \kappa(\rho)
        \qquad \text{on } K^\ast \setminus B(\mathcal E,\rho),
        \qquad \text{whenever } \varepsilon' < \varepsilon_\rho .
    \end{equation}
    In particular $V$ is strictly decreasing along the perturbed flow outside $B(\mathcal E,\rho)$.
    Observe that every zero of $f'$ in $K^\ast$ lies in $B(\mathcal E,\rho)$.

    \emph{Step 2.}
    Let $\zeta(\rho) := \rho \max\{|\nabla V(x)| : x \in K^\ast, \ d(x,\mathcal E)\le\rho\}$. The gradient $\nabla V$ is continuous and vanishes on $\mathcal E$, so $\zeta(\rho) \to 0$ as $\rho \to 0$. If $d(x,E_i)\le\rho$, integrating $\nabla V$ along the segment joining $x$ to a nearest point of $E_i$, a segment contained in $K^\ast$ by convexity and lying within distance $\rho$ of $E_i$, gives
    \begin{equation}
        \label{eq:near-E}
        |V(x)-v_i| \ \le \ \zeta(\rho).
    \end{equation}

    \emph{Step 3.}
    We now claim that there exists $\eta\in (0,r/2)$ such that letting
    \[
        C_i := [\,v_i-\zeta(\eta),\ v_i+\zeta(\eta)\,], \qquad i=1,\ldots,q,
    \]
    we have that
    \begin{equation}
        \label{eq:gain-corridor}
        \nabla V(x)\cdot f'(x) > 0
        \quad\Longrightarrow\quad
        V(x) \in C_1\cup\cdots\cup C_q .
    \end{equation}
    That is, along the perturbed flow $V$ can increase only while its value lies in one of the intervals $C_1,\ldots, C_q$.

    To this aim, let $F := 1+\max_{K^\ast}|f|$, so that $|f'|\le F$ on $K^\ast$ as soon as $\varepsilon'\le1$, and set
    \[
        \theta := \frac{\kappa(r/2)}{2}\cdot\frac{r}{2F} \ > \ 0 .
    \]
    Choose $\eta \in (0,r/2)$ so small that
    \begin{equation}
        \label{eq:eta-choice}
        2\zeta(\eta) < \theta
        \qquad\text{and}\qquad
        2\zeta(\eta) < \tfrac13\min_{i\neq j}|v_i-v_j| ,
    \end{equation}
    the second condition being void when $q=1$, and then let $\varepsilon>0$ be such that $\|W'-W\|_\infty<\varepsilon$ implies $\varepsilon' < \min(\varepsilon_\eta, \varepsilon_{r/2},1)$. The intervals $C_1,\ldots, C_q$ are then pairwise disjoint, and \eqref{eq:off-E} and \eqref{eq:near-E} give \eqref{eq:gain-corridor}.

    \emph{Step 4.}
    We claim that if $\Gamma$ is a heteroclinic cycle, a homoclinic orbit or a periodic orbit of~\eqref{eq:WC-perturbed-W}, then there exists a unique $i \in \{1,\ldots,q\}$ such that $V(\Gamma) \subset C_i$.
    In particular this immediately implies that
    \begin{equation}
        \label{eq:small-range}
        \max_\Gamma V - \min_\Gamma V \ \le \ 2\zeta(\eta) \ < \ \theta .
    \end{equation}

    We start by observing that $\Gamma$ is a union of complete bounded orbits, so Lemma~\ref{lem:bounded} places it in $K^\ast$.
    Moreover, $V(\Gamma)$ meets $C_1\cup\cdots\cup C_q$. Indeed, if $\Gamma$ is periodic and did not meet $B(\mathcal E,\eta)$, then \eqref{eq:off-E} would make $V$ strictly decreasing along one period, which is impossible. If $\Gamma$ is heteroclinic or homoclinic, its equilibrium points are zeros of $f'$ and hence belong to $B(\mathcal E,\eta)$ by Step~1. In either case, \eqref{eq:near-E} shows that $V(\Gamma)$ meets one of the intervals $C_i$.

    Since $V(\Gamma)$ is connected, it suffices to show that for any $c \notin C_1\cup\cdots\cup C_q$, $\Gamma$ cannot meet both $\{V<c\}$ and $\{V>c\}$.

    Observe first that $\nabla V(x)\cdot f'(x)<0$ at every $x \in K^\ast$ with $V(x)=c$.
    Indeed, if $x\in B(\mathcal E,\eta)$, then $x$ lies within distance $\eta$ of some $E_i$, and~\eqref{eq:near-E} would give $V(x)\in C_i$, contrary to the choice of $c$. Thus $x\notin B(\mathcal E,\eta)$, and~\eqref{eq:off-E} applies with $\rho=\eta$.
    In particular $\Gamma$ contains no equilibrium of $f'$ at which $V$ equals $c$.

    Suppose $\Gamma$ is periodic and meets both sides. Then $t\mapsto V(x(t))$ is periodic and takes values both below and above $c$, so there is a time $s$ with $V(x(s))=c$ and $V(x(t))>c$ for $t$ slightly larger than $s$. This forces $\frac{d}{dt}V(x(t))\big|_{t=s}\ge0$, contradicting the observation.

    Suppose now that $\Gamma$ is a heteroclinic cycle, with equilibria $q_1,\ldots,q_m$ of $f'$ and connections $\gamma_k$ from $q_k$ to $q_{k+1}$, indices modulo $m$. A homoclinic orbit is the case $m=1$. By the observation, $V(q_k)\neq c$ for every $k$. If some point of $\gamma_k$ satisfies $V<c$, then $V<c$ along the whole forward orbit from that point: at a first time $s$ with $V=c$ one would need $\frac{d}{dt}V\big|_{t=s}\ge0$.
    Passing to the limit gives $V(q_{k+1})\le c$. Since $V(q_{k+1})\neq c$, we have $V(q_{k+1})<c$, and the same applies to $\gamma_{k+1}$. Going around the cycle, if any point of $\Gamma$ has $V<c$ then every point does.
    Thus $\Gamma$ cannot meet both $\{V<c\}$ and $\{V>c\}$. Since $V(\Gamma)$ is connected, meets $C_1\cup\cdots\cup C_q$, and the intervals $C_i$ are pairwise disjoint, it follows that $V(\Gamma)\subset C_i$ for a unique $i$, proving the claim.

    \emph{Step 5.}
    Suppose, by contradiction, that some $z \in \Gamma$ has $d(z,\mathcal E)\ge r$.
    As observed in Step~4, $\Gamma$ meets $B(\mathcal E,\eta)$.

    Let $\gamma$ be the orbit of $\Gamma$ through $z$. If $\Gamma$ is periodic then $\gamma=\Gamma$ meets $B(\mathcal E,\eta)$. Otherwise the ends of $\gamma$ are zeros of $f'$, which lie in $B(\mathcal E,\eta)$ by Step~1. In both cases $t\mapsto d(x(t),\mathcal E)$ is continuous along $\gamma$ and takes a value smaller than $\eta<r/2$ and a value at least $r$, so there are times $s_1<s_2$ with
    \[
        \big\{ d(x(s_1),\mathcal E),\ d(x(s_2),\mathcal E) \big\} = \{r/2,\ r\},
        \qquad
        d(x(t),\mathcal E) \in [r/2,r] \ \text{ for } t \in [s_1,s_2],
    \]
    the two values occurring in one order or the other according to whether $\gamma$ passes through $z$ after or before it visits $B(\mathcal E,\eta)$. Then $|x(s_2)-x(s_1)| \ge r/2$, so $s_2-s_1 \ge r/(2F)$, and \eqref{eq:off-E} with $\rho=r/2$ gives
    \[
        V(x(s_2)) - V(x(s_1)) \ \le \ -\frac{\kappa(r/2)}{2}\,(s_2-s_1) \ \le \ -\theta ,
    \]
    which contradicts~\eqref{eq:small-range}. Hence $\Gamma \subseteq B(\mathcal E,r)$ and, being connected, $\Gamma$ lies in a single connected component of $B(\mathcal E,r)$.
\end{proof}

\begin{remark}
    Theorem~\ref{thm:robust} does not exclude a cycle confined to a small neighborhood of a single connected component of the equilibrium set, and this is unavoidable without further assumptions.
    If, in addition, every equilibrium of~\eqref{eq:WC} is hyperbolic, then each equilibrium has a sufficiently small isolating neighborhood, and this property persists under sufficiently small perturbations. Combined with Theorem~\ref{thm:robust}, this shows that, for $\varepsilon$ small, the perturbed system has no heteroclinic cycle, no homoclinic orbit, and no periodic orbit whatsoever.
\end{remark}

\begin{remark}
    A complementary statement is available in the literature.
     Theorem~1 of \cite{hirsch1989convergent} shows that if a system has isolated equilibria and a strict Lyapunov function, then for every $T, \eta > 0$ there is $\delta_0 > 0$ such that every limit point of a $(T,\delta)$-perturbed trajectory with $\delta < \delta_0$ lies within $\eta$ of an equilibrium.
    In particular, \cite{hirsch1989convergent} controls the forward limit set of very general, even discontinuous, perturbed trajectories, whereas Theorem~\ref{thm:robust} controls the whole periodic, homoclinic, or heteroclinic object for the more specific perturbations induced by a nearby connectivity matrix. Observe that \cite{hirsch1989convergent} assumes that all equilibria are isolated, while Theorem~\ref{thm:robust} does not.
\end{remark}

\section{Neural-population representation of sequential states}
\label{sec:UATprocedure}
The obstruction identified in the previous section arises when the states of the heteroclinic sequence are prescribed directly as equilibria of the neural-field system~\eqref{eq:WC}. We now adopt a different representation, in which the sequential dynamics is described in a low-dimensional space and its states are realized as aggregates of a higher-dimensional neural population system.

\subsection{Universal approximation theorem for aggregated dynamics}
\label{sec:UATresults}
We begin by recalling a classical $C^1$ universal approximation theorem for neural networks. A general overview of the mathematical theory of deep learning can be found in \cite{petersen2024mathematical}.

For an open set $X\subseteq\mathbb R^n$, let $C^1(X,\mathbb R^m)$ denote the space of continuously differentiable functions from $X$ to $\mathbb R^m$. For a compact set $K\subset X$, we write, by a slight abuse of notation, $f\in C^1(K,\mathbb R^m)$ when $f$ is defined and continuously differentiable on an open
neighborhood of $K$ contained in $X$.
We use the norm
\begin{equation}
\label{def:C1norm}
\|f\|_{C^1(K)}:=\sup_{x\in K}\bigl(\|f(x)\|+\|Df(x)\|\bigr),
\end{equation}
where the vector norm is Euclidean and the matrix norm is the induced operator norm.

\begin{theorem}[Universal Approximation Theorem, {\cite[Theorem~4.1]{pinkus1999approximation}}]
    \label{theo:UAT}
    Suppose that $\sigma \in C^1(\mathbb{R},\mathbb{R})$ is not a polynomial. Then, for every $n,m \in \mathbb{N}$, compact $K \subseteq \mathbb{R}^{n}$, $\tilde g \in C^1(K,\mathbb{R}^{m})$, and $\varepsilon > 0$, there exist $N \in \mathbb{N}$, $W \in \mathbb{R}^{N \times n}$, $b \in \mathbb{R}^{N}$, and $P \in \mathbb{R}^{m \times N}$ such that
    \begin{equation}
        \| \tilde g - \tilde f \|_{C^1(K)} < \varepsilon ,
        \qquad\text{where}\qquad
        \tilde f(x) = P \sigma(Wx + b).
    \end{equation}
\end{theorem}
The vector-valued statement of the theorem follows from the scalar result by extending each component of $\tilde g$ to a $C^1$ function on $\mathbb R^n$ using a smooth cutoff, applying the approximation theorem componentwise, and concatenating the corresponding hidden units. Earlier results on the simultaneous approximation of functions and their derivatives were obtained in \cite{hornik1990universal}. For approximation in the uniform norm, see \cite{leshno1993multilayer}.

We are interested in an extension of the universal approximation theorem, with approximating function $f(x) = -Lx + \tilde f(x)$, for $L \in \mathbb{R}^{m \times n}$.

\begin{corollary}
    \label{cor:UAT}
    Suppose that $\sigma \in C^1(\mathbb{R},\mathbb{R})$ is not a polynomial. Then, for every $n,m \in \mathbb{N}$, compact $K \subseteq \mathbb{R}^{n}$, $g \in C^1(K,\mathbb{R}^{m})$, $L \in \mathbb{R}^{m\times n}$ and $\varepsilon > 0$, there exist $N \in \mathbb{N}$, $W \in \mathbb{R}^{N \times n}$, $b \in \mathbb{R}^{N}$, and $P \in \mathbb{R}^{m \times N}$ such that
    \begin{equation*}
        \| g - f \|_{C^1(K)} < \varepsilon ,
    \end{equation*}
    where  $f(x) = -Lx + P \, \sigma(Wx+b)$.
\end{corollary}

\begin{proof}
    Apply Theorem~\ref{theo:UAT} to the function $\tilde g(x) := g(x)+Lx$ over $K$. The statement follows noting that $\|\tilde f-\tilde g\|_{C^1(K)} = \|f-g\|_{C^1(K)}$, with $f = \tilde f - Lx$.
\end{proof}

In the following, we set $m=n$ and $L=\idty$, where $\idty$ denotes the identity matrix on $\mathbb{R}^n$. Moreover, introducing the parameter space
\begin{equation}
    \label{eq:parameter-set}
    \Theta \coloneqq \{\theta = (P, W, b) \mid P \in \mathbb{R}^{n \times N},\, W \in \mathbb{R}^{N \times n},\, b \in \mathbb{R}^N \} \simeq \mathbb{R}^k,
\end{equation}
with $k = N(1+2n)$, we define for every $\theta \in \Theta$ the function
\begin{equation}
    \label{eq:univ-approx}
    f_\theta(x) := -x + P \, \sigma(Wx+b), \qquad x \in \mathbb{R}^n.
\end{equation}

\subsubsection{Interpretation in terms of neural populations}
\label{sec:interpretation}
We now interpret the network defined by~\eqref{eq:univ-approx} in terms of neural populations and their collective dynamics.

Fix $N \in \mathbb{N}$ and $\theta = (P,W,b) \in \Theta$. Let $y(t) \in \mathbb{R}^N$ evolve according to
\begin{equation}
    \label{eq:WC-approx}
    \dot{y} = H(y), \qquad H(y)= -y + \sigma(W P y + b).
\end{equation}
Here, $\sigma\in C^1(\mathbb R,\mathbb R)$ is a non-polynomial activation function applied componentwise and satisfying the activation-function assumptions of Section~\ref{ssec:MFeq}. The bias $b\in\mathbb R^N$ is not required to be positive. This system describes the activity of $N$ interacting neural populations.

\begin{figure}[tb]
    \centering
    \includegraphics[width=0.9\linewidth]{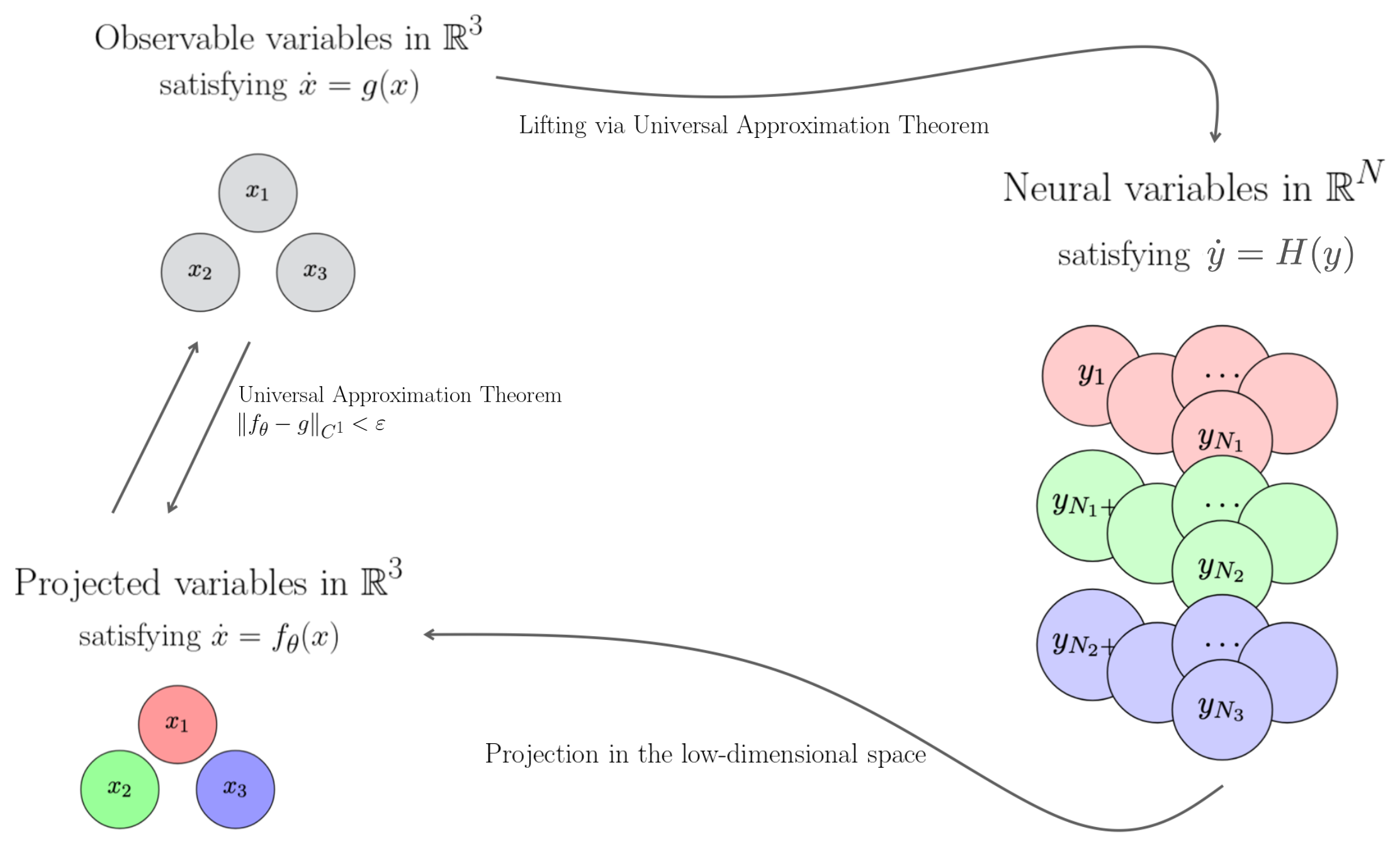}
    \caption{Schematic of neural interpretability. The image illustrates how the network-approximated dynamics can be linked to a neural model simulating population activity. We start from the target dynamics $g$ defined in $\mathbb{R}^n$, the low-dimensional space of observable variables (here $n = 3$). By the universal approximation theorem, there exists a network $f_{\theta}$ that approximates $g$ in the $C^{1}$ norm. This allows us to lift the dynamics into a higher-dimensional neural space $\mathbb{R}^N$, where the system follows neural-field-type dynamics $\dot{y} = H(y)$. Projecting these neural dynamics back into the low-dimensional space maps neural activity to interpretable variables satisfying $\dot{x} = f_{\theta}(x)$.}
    \label{fig:neural-interpretability}
\end{figure}

Projecting the high-dimensional dynamics $y(t)$ through $P$ yields the low-dimensional trajectory in $\mathbb{R}^n$
\begin{equation}
    \label{eq:projection-dynamics}
    x(t) := P y(t) \quad \text{ satisfying }\quad\dot{x} = f_\theta(x).
\end{equation}
Thus, approximating a target $g$ by $f_\theta$ realizes $g$, up to an error $\varepsilon$, as the projection of a higher-dimensional neural-field system of the form~\eqref{eq:WC}, without necessarily satisfying the additional assumptions imposed in Section~\ref{sec:absence}.

The relation between the population dynamics
\eqref{eq:WC-approx} and the aggregated dynamics
$\dot x=f_\theta(x)$ can be made more precise for periodic solutions. Throughout, we call a nonconstant solution $z:\mathbb R\to\mathbb R^n$ \emph{$T$-periodic} if $T>0$ is its minimal period, namely the smallest positive number such that
\[
    z(t+T)=z(t)
    \qquad\text{for every }t\in\mathbb R.
\]

\begin{proposition}
\label{prop:periodic-lifting}
Let $x_*$ be a $T$-periodic solution of $\dot x=f_\theta(x)$. Then system~\eqref{eq:WC-approx} admits a unique $T$-periodic solution $y_*$ satisfying
\[
    Py_*(t)=x_*(t) \qquad\text{for every }t\in\mathbb R.
\]
Moreover, every solution $y$ of~\eqref{eq:WC-approx} with $Py(0)=x_*(0)$ satisfies
\begin{equation}
\label{eq:fibre-attraction}
    y(t)-y_*(t)  =e^{-t}\bigl(y(0)-y_*(0)\bigr), \qquad t\geq0.
\end{equation}
\end{proposition}

\begin{proof}
Set $h(t):=\sigma(Wx_*(t)+b)$, which is a $T$-periodic function. The linear equation $\dot y = -y +h$ then admits a unique $T$-periodic solution $y_*$. The difference $d = Py_*-x_*$ satisfies the equation $\dot d = -d$ and then $d(t) = e^{-t}d(0)$. However, since $d$ is $T$-periodic, we have $d(T)=d(0)=e^{-T}d(0)$, which requires $d(0)=0$. This implies $Py_* = x_*$.

Let $y$ be a solution of~\eqref{eq:WC-approx} with $Py(0)=x_*(0)$. By uniqueness for the projected system, $Py(t)=x_*(t)$ for all $t\geq0$. Thus $y$ and $y_*$ solve the same equation $\dot y=-y+h(t)$, and their difference satisfies $\frac{d}{dt}(y-y_*)=-(y-y_*)$.
Therefore, \eqref{eq:fibre-attraction} holds.
\end{proof}

\begin{remark}
If the reduced periodic orbit is hyperbolic and asymptotically stable, then its lifted periodic orbit has the same properties. Indeed, the linearized dynamics projects onto that of the reduced system, while perturbations in $\ker P$ decay as $e^{-t}$.
\end{remark}

\begin{remark}
The existence of periodic orbits in system~\eqref{eq:WC-approx} does not contradict Theorem~\ref{thm:main}. Indeed, the bias parameters in the lifted system are not required to be positive and, more importantly, the system is not constrained to possess the prescribed activity-pattern equilibria considered in that theorem. Even if the biases are positive, the structural conditions forcing detailed balance need not hold.
\end{remark}

To improve interpretability, we impose a block-diagonal structure on $P$
\begin{equation}
    \label{eq:block-structure}
    P =
    \begin{pmatrix}
        P_1    & 0      & \cdots & 0      \\
        0      & P_2    & \cdots & 0      \\
        \vdots & \vdots & \ddots & \vdots \\
        0      & 0      & \cdots & P_n
    \end{pmatrix},
    \quad P_i \in \mathbb{R}^{1 \times N_i}, \quad \sum_{i=1}^n N_i = N
\end{equation}
so that each low-dimensional component $x_i$ depends only on a specific subset of neural variables, corresponding to a distinct subpopulation.

With these constraints on $P$, the relationship between the lifted neural dynamics and the projected observable system follows the scheme illustrated in Figure~\ref{fig:neural-interpretability}.
The following result shows that this does not affect the approximation capability.

\begin{corollary}
    \label{prop:blockP}
    Corollary~\ref{cor:UAT}, with $n=m$ and $L=\idty$, holds under the additional constraint~\eqref{eq:block-structure} on $P$.
\end{corollary}

\begin{proof}
    Under the block-diagonal constraint~\eqref{eq:block-structure}, each scalar component $g_i(x)$ can be approximated independently by $f_i(x) = -x_i + P_i \sigma(W_i x + b_i)$, with $W_i \in \mathbb{R}^{N_i \times n}$  and $b_i \in \mathbb{R}^{N_i}$. The claim follows applying Corollary~\ref{cor:UAT} to each component and concatenating the results.
\end{proof}

\subsection{From heteroclinic cycles to periodic orbits}
\label{sec:sequential-periodic}

We now turn to the case in which the target vector field \(g\) has an asymptotically stable heteroclinic cycle \(\Gamma\), as in Section~\ref{sec:heterocline}. We aim to understand how much of this sequential structure is preserved by an approximation of the form \(f_\theta\), and in particular which dynamical features of the heteroclinic behavior survive under \(C^1\)-small perturbations.

In a companion paper~\cite{bolelli2026periodic}, in preparation, we show that under structural assumptions on the cycle, an appropriate class of $C^1$-small perturbations breaks the heteroclinic connections and produces a unique attracting periodic orbit in a channel around $\Gamma$. That result holds in arbitrary ambient dimension and assumes neither symmetry nor robustness of the connections, and its proof is outside the scope of the present paper.
Here we prove a three-dimensional version, which suffices for the case study of Section~\ref{sec:FAM}. Part of the proof is given in the Supplementary Material.

Throughout this subsection, consider $\bar x_i \in \mathbb{R}^3$ for $i = 1, \ldots, Q$, and let $g\in C^3(\mathbb R^3,\mathbb R^3)$ have a heteroclinic cycle
\begin{equation}
    \label{eq:cycle}
    \Gamma
    =
    \bigcup_{i=1}^Q
    \bigl(\{\bar x_i\}\cup\Gamma_i^+\bigr),
    \qquad
    \Gamma_i^+
    \subset
    W^u_+(\bar x_i)\cap W^s(\bar x_{i+1}),
\end{equation}
where indices are understood modulo \(Q\). Let \(K\subset\mathbb R^3\) be a compact set containing \(\Gamma\) in its interior. All perturbed vector fields $f$ are assumed to belong to \(C^3(\mathbb R^3,\mathbb R^3)\)\footnote{The \(C^3\)-regularity requirement for approximating functions of the form \eqref{eq:univ-approx} translates into a corresponding regularity requirement on the activation function \(\sigma\).}, while the \(C^r\)-norms below are computed on \(K\).

The argument has three steps. We first obtain uniform estimates for the local passage near each saddle. We then compose each local passage with the global transition along the heteroclinic connection. Finally, we construct rectangles that are mapped strictly into one another. The orbit segments joining successive rectangles form a trapping channel, and the resulting first-return map is a contraction. Figure~\ref{fig:local}(a) illustrates the geometry of the
local neighborhoods and the periodic regime sought near a
reference heteroclinic cycle.

We start with the following assumptions for the dynamics described by $g$.
\begin{enumerate}[label=\textup{(A\arabic*)}]
    \item\label{A:spectrum}
    Each \(\bar x_i\) is a hyperbolic saddle with real, simple eigenvalues satisfying
    \begin{equation}
        \label{eq:spectrum}
        -\lambda_{ss}^i<-\lambda_s^i<0<\lambda_u^i .
    \end{equation}
    We let $W^{ss}(\bar x_i)\subset W^s(\bar x_i)$ denote the strong stable manifold of $\bar x_i$, i.e., the manifold tangent to the eigenspace of $-\lambda_{ss}^i$ at $\bar x_i$.

    \item\label{A:product}
    Every saddle is individually contracting:
    \begin{equation}
        \label{eq:saddle-values}
        \nu_i:=\frac{\lambda_s^i}{\lambda_u^i}>1,
        \qquad i=1,\ldots,Q.
    \end{equation}

    \item\label{A:incoming}
    No incoming connection lies in the strong stable manifold, that is $\Gamma_{i-1}^+\not\subset W^{ss}(\bar x_i)$ for every $i=1,\ldots,Q$.
\end{enumerate}

By Assumption~\ref{A:incoming}, the incoming connection \(\Gamma_{i-1}^+\) approaches \(\bar x_i\) tangentially to the leading stable eigendirection.
We orient the corresponding eigenvector \(e_i^s\) so that \(\Gamma_{i-1}^+\) approaches \(\bar x_i\) along \(+e_i^s\), and orient the unstable eigenvector \(e_i^u\) so that \(\Gamma_i^+\) leaves \(\bar x_i\) along \(+e_i^u\).
Moreover, since \(Q<\infty\) and \(\nu_i>1\) for every \(i\), we can fix $0<\kappa<
\min_{1\le i\le Q}\left\{\lambda_u^i,\frac{\lambda_s^i-\lambda_u^i}{3}\right\}$.
Then
\begin{equation}
    \label{eq:kappa-choice}
    \beta_i
    :=
    \frac{\lambda_s^i-2\kappa}
         {\lambda_u^i+\kappa}
    >1,
    \qquad i=1,\ldots,Q.
\end{equation}
Let us begin with the local passage near each equilibrium.

\begin{figure}[tb]
    \centering
    \begin{subfigure}{0.48\textwidth}
        \centering
        \includegraphics[width=\textwidth]{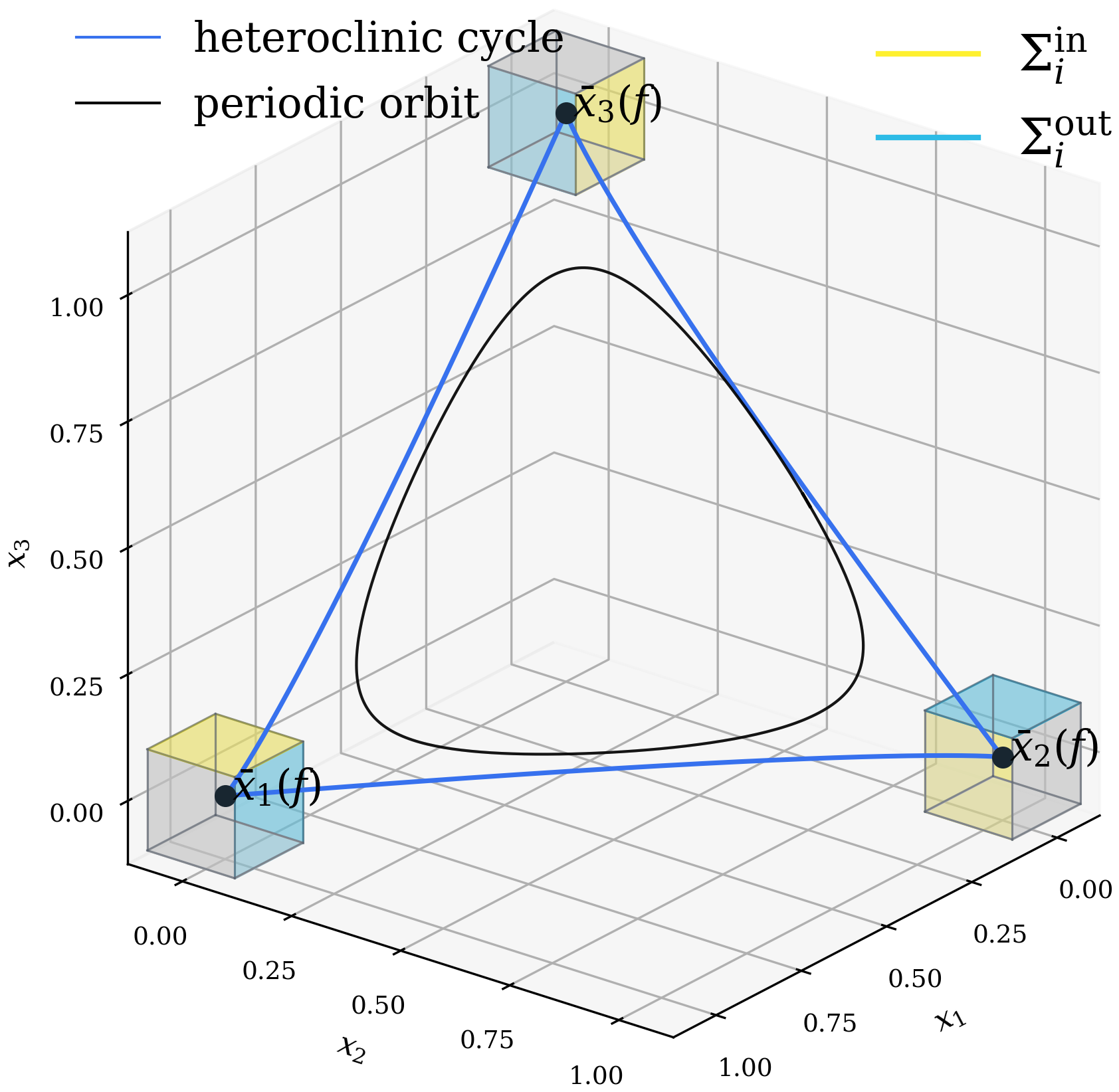}
        \caption{}
    \end{subfigure}
    \begin{subfigure}{0.48\textwidth}
        \centering
        \includegraphics[width=\textwidth]{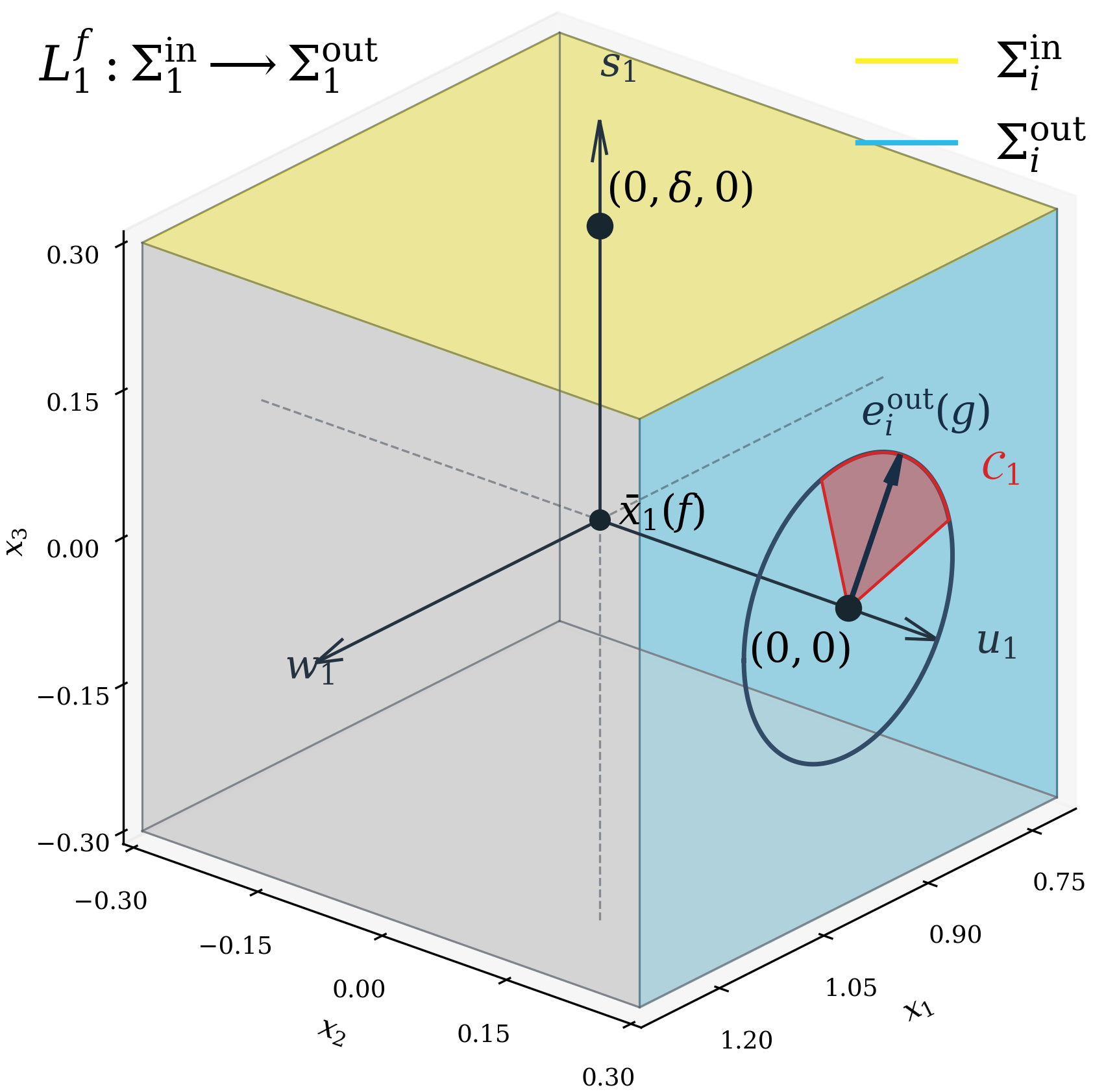}
        \caption{}
    \end{subfigure}
    \caption{ (a) Schematic overview of the construction proposed in Section~\ref{sec:sequential-periodic}. The blue curve represents a possible heteroclinic cycle, while the black curve illustrates a possible periodic orbit. The boxes represent local coordinate neighborhoods of the continued saddles $\bar x_i(f)$. Yellow and blue identify the incoming and outgoing sections, respectively. (b) Local passage near $\bar x_1(f)$ in adapted coordinates $(u, s, w)$.  The local map $L_1^f$ transports incoming points to the outgoing section. In the plane $\Sigma_1^{\rm out}$, the origin corresponds to the intersection with the local unstable branch. The red sector represents an outgoing cone $\mathcal C_1$ around the reference limiting direction $e_1^{\mathrm{out}}(g)$.}
    \label{fig:local}
\end{figure}

\begin{lemma}
    \label{lem:passage}
    Assume \ref{A:spectrum}--\ref{A:incoming}, and let \(R>0\). There exist \(\delta_0>0\) and a map $\varepsilon_0\colon(0,\delta_0]\longrightarrow(0,\infty)$ such that, for every \(\delta\in(0,\delta_0]\) and every \(f\in C^3(\mathbb R^3,\mathbb R^3)\) satisfying $\|f\|_{C^3(K)}\le R$ and $\|f-g\|_{C^1(K)}<\varepsilon_0(\delta)$, the following statements hold for every \(i=1,\ldots,Q\).

    \begin{enumerate}[label=\textup{(\roman*)}]
        \item\label{pass:persist}
        The vector field \(f\) has a unique equilibrium \(\bar x_i(f)\) in a fixed neighborhood of \(\bar x_i\).
        It is a hyperbolic saddle whose eigenvalues
        \[
            -\lambda_{ss}^i(f)
            <
            -\lambda_s^i(f)
            <
            0
            <
            \lambda_u^i(f)
        \]
        are real, simple and converge to the corresponding eigenvalues of \(Dg(\bar x_i)\) as \(f\to g\) in \(C^1(K)\).

        \item\label{pass:coord}
        There exist \(C^3\) local coordinates \((u,s,w)\), centered at \(\bar x_i(f)\), whose coordinate domain contains $ Q(\delta_0):=\{|u|,|s|,|w|\le\delta_0\},$ and in which
        \begin{equation}
            \label{eq:straightened}
            W^u_{\rm loc}=\{s=w=0\},
            \qquad
            W^s_{\rm loc}=\{u=0\},
            \qquad
            W^{ss}_{\rm loc}=\{u=s=0\},
        \end{equation}
        where $W^{\bullet}_{\rm loc} = W^{\bullet}_{\rm loc}(\bar x_i(f))$,  for  $\bullet \in \{u, s, ss\}$.
       The charts are oriented so that the positive $u$- and $s$-axes point in the directions of the unstable and leading stable eigenvectors $e_i^u(f)$ and $e_i^s(f)$, respectively, chosen continuously with respect to $f$ and agreeing with the orientations fixed above for $g$. The $w$-axis is tangent to the strong stable eigendirection.

        \item\label{pass:map}
        Set
        \begin{equation}
            \label{eq:sections}
            \Sigma_i^{\rm in}
            :=
            \{s=\delta,\ 0<u\le\delta,\ |w|\le\delta\},
            \quad
            \Sigma_i^{\rm out}
            :=
            \{u=\delta,\ |s|\le\delta,\ |w|\le\delta\}.
        \end{equation}
        Every orbit starting at \((u,\delta,w)\in\Sigma_i^{\rm in}\) reaches \(\Sigma_i^{\rm out}\) for the first time in a finite local passage time $\tau_i^f\geq0$. Writing the resulting local passage map as
        \begin{equation}
            \label{eq:local-map}
            L_i^f(u,w)
            =
            \bigl(\delta,z_i^f(u,w)\bigr),
            \qquad
            z_i^f(u,w)\in\mathbb R^2,
        \end{equation}
        there exist constants \(C_{0,i},C_{1,i}>0\), independent of \(\delta,f,u,w\), such that
        \begin{align}
            0< \|z_i^f(u,w)\|
            &\le
            C_{0,i}\delta
            \left(\frac{u}{\delta}\right)^{\beta_i},
            \label{eq:local-size}\\
            \|Dz_i^f(u,w)\|
            &\le
            C_{1,i}
            \left(\frac{u}{\delta}\right)^{\beta_i-1}.
            \label{eq:local-derivative}
        \end{align}

        \item\label{pass:time}
        The local passage time satisfies
        \begin{equation}
            \label{eq:passage-time}
            \tau_i^f(u,w)
            =
            \frac{1}{\lambda_u^i(f)}
            \log\frac{\delta}{u}
            +O(\delta),
        \end{equation}
        where the remainder is uniform in \(f,u,w\), for all sufficiently small $\delta$.

        \item\label{pass:direction}
        There exist constants $\rho_i\in(0,1)$, $\mu_i>0$, and
        $C_{\mathrm{dir},i}>0$, independent of $\delta$ and $f$, such that,
        for each admissible $\delta$ and $f$, there exists a unit vector
        $e_i^{\rm out}(f)$\footnote{The dependence on $\delta$ is omitted from the notation.}
        satisfying
        \begin{equation}
            \label{eq:exit-direction}
            \sup_{|w|\le\rho_i\delta}
            \left\|
                \frac{z_i^f(u,w)}{\|z_i^f(u,w)\|}
                -e_i^{\rm out}(f)
            \right\|
            \le
            C_{\mathrm{dir},i}
            \left(\frac{u}{\delta}\right)^{\mu_i},
            \qquad 0<u\le\delta.
        \end{equation}
        For each fixed $\delta\in(0,\delta_0]$, the limiting direction $e_i^{\rm out}(f)$ depends continuously on $f$ in the $C^1(K)$ topology, within the admissible class of perturbations.

    \end{enumerate}
\end{lemma}

\begin{proof}
    Estimates of this kind are classical for a single vector field, or for a smooth finite-dimensional family of vector fields. They follow from the analysis of the Shilnikov boundary-value problem near a saddle \cite[Section~2.8]{shilnikov1998methods}, which yields estimates of the form needed here in \cite[Section~13.8]{shilnikov2001methods}, and from the strong $\lambda$-lemma \cite{deng1989silnikov}, and they are discussed in terms of Shilnikov variables in \cite[Sections~2.3 and~3]{homburg2010homoclinic}. Lemma~\ref{lem:passage} adds that all the constants are uniform over the $C^1$-neighborhood of $g$ considered here, which is not a finite-dimensional family. We give the full proof in Section~\ref{app:local-passage} of the supplementary material.
\end{proof}

From now on, we identify the incoming and outgoing sections~\eqref{eq:sections} with their local coordinate representations: points of $\Sigma_i^{\rm in}$ are written as $(u,w)$, with $s=\delta$, while points of $\Sigma_i^{\rm out}$ are written as $z=(s,w)$, with $u=\delta$. The local sections, the limiting exit direction, and the outgoing cones $\mathcal C_i$ used in the global-transit argument are illustrated in Figure~\ref{fig:local}(b).

To define the global transition map, we extend the incoming section across the local stable manifold $u=0$ and set
\begin{equation}
    \label{eq:sec-incoming-ext}
    \widehat\Sigma_i^{\rm in}
    :=
    \{s=\delta,\ |u|<\delta,\ |w|<\delta\}.
\end{equation}
Using $(u, w)$ as coordinates on $\widehat\Sigma_i^{\rm in}$, we note that $u$ is the signed coordinate relative to the stable manifold. The section $\Sigma_i^{\rm in}$ corresponds to its part with $u>0$.

For \(f=g\), the orbit \(\Gamma_i^+\) crosses \(\Sigma_i^{\rm out}\) and \(\widehat\Sigma_{i+1}^{\rm in}\) transversally. Hence, for \(f\) sufficiently \(C^1\)-close to \(g\), the flow defines a \(C^1\) global transition map
\begin{equation}
    \label{eq:global-map}
    \Psi_i^f=(U_i^f,W_i^f)
    \colon
    V_i\subset\Sigma_i^{\rm out}
    \longrightarrow
    \widehat\Sigma_{i+1}^{\rm in},
\end{equation}
where \(V_i\) is a neighborhood of \((0, 0)\), and \(U_i^f,W_i^f\) are the \(u\)- and \(w\)-components on the target section. We write
\begin{equation}
    \label{eq:global-landing}
    \Psi_i^f(0,0)
    =
    \bigl(d_i(f),\omega_i(f)\bigr).
\end{equation}
Since \((0,0)\) lies on the continued unstable branch \(W_+^u(\bar x_i(f))\), \(d_i(f)\) is the signed \(u\)-coordinate of its landing point on the next incoming section, see Figure~\ref{fig:global}(a). For \(f=g\), this branch follows \(\Gamma_i^+\) and lands in \(W^s(\bar x_{i+1})=\{u=0\}\), hence $d_i(g)=U_i^g(0,0)=0$.

Although the value of $d_i(f)$ depends on the adapted coordinates and on the choice of the incoming section, its sign has an intrinsic geometric meaning. Indeed, $d_i(f)>0$ means that the continued unstable branch $W_+^u(\bar x_i(f))$ reaches the positive side of $W^s_{\rm loc}(\bar x_{i+1}(f))$, oriented by $+e_{i+1}^u(f)$. This condition is independent of the choice of adapted coordinates preserving the fixed orientations. It is also independent of the choice of admissible incoming section, since the stable manifold is invariant under the flow and the continued unstable branch cannot cross it.

We finally impose the following orientation condition on the global transition.

\begin{enumerate}[label=\textup{(A\arabic*)},start=4]
    \item\label{A:transit}
    Every connection has positive transit:
    \begin{equation}
        \label{eq:positive-transit}
        DU_i^g(0,0)
        \bigl[e_i^{\rm out}(g)\bigr]
        >0,
        \qquad i=1,\ldots,Q.
    \end{equation}
\end{enumerate}

Here \(e_i^{\rm out}(g)\) is the limiting exit direction of Lemma~\ref{lem:passage} when $f=g$. Condition~\ref{A:transit} means that a positive displacement from $(0,0)$ in the limiting exit direction produces, under the global transition, a positive first-order displacement in the $u$-coordinate on the next incoming section. For nearby fields, this yields $U_i^f(z)>d_i(f)$ for sufficiently small nonzero $z$ whose direction $z/\|z\|$ is close to the limiting exit direction. The following lemma gives the corresponding uniform estimate.

As in the case of the splitting $d_i(f)$, the positivity in Assumption~\ref{A:transit} is independent of admissible choices of sections and adapted coordinates preserving the prescribed orientations. It expresses the fact that the limiting outgoing direction is transported by the flow toward the positive side of the stable manifold of the next saddle.

\begin{lemma}
    \label{lem:transit}
    Under the standing Assumptions~\ref{A:spectrum}--\ref{A:incoming}, assume in addition \ref{A:transit}. After decreasing the \(C^1\)-neighborhood of \(g\) if necessary, there exist, for every \(i\), constants $\eta_i,r_i,c_i,C_i>0$ such that
    \begin{equation}
        \label{eq:transit-cone}
        c_i\|z\|
        \le
        U_i^f(z)-d_i(f)
        \le
        C_i\|z\|
    \end{equation}
    whenever
    \[
        0<\|z\|\le r_i,
        \qquad
        \left \|
            \frac{z}{\|z\|}-e_i^{\rm out}(g)
        \right \|
        \le\eta_i.
    \]
    Moreover, we also have
    \begin{equation}
        \label{eq:eout-close}
        \|e_i^{\rm out}(f)-e_i^{\rm out}(g)\|
        \le\frac{\eta_i}{2}
    \end{equation}
    for every admissible \(f\).
\end{lemma}

\begin{proof}
    Fix \(i\). By Assumption~\ref{A:transit}, $  DU_i^g(0,0)\bigl[e_i^{\rm out}(g)\bigr] >0$. By continuity of \(DU_i^g\), and by the \(C^1\)-continuous dependence of the global transition on \(f\), there exist \(\eta_i,r_i,c_i,C_i>0\) such that
    \begin{equation}
        \label{eq:transit-derivative-bounds}
        c_i
        \le
        DU_i^f(x)[v]
        \le
        C_i
    \end{equation}
    whenever $\|x\|\le r_i$, $\|v\|=1$, and $\|v-e_i^{\rm out}(g)\|\le\eta_i$, uniformly for \(f\) sufficiently close to \(g\).

    Since \(d_i(f)=U_i^f(0)\), the fundamental theorem of calculus gives
    \[
    \begin{aligned}
        U_i^f(z)-d_i(f)
        &=
        \int_0^1 DU_i^f(tz)[z]\,dt =
        \|z\|
        \int_0^1
        DU_i^f(tz)
        \left[\frac{z}{\|z\|}\right]dt.
    \end{aligned}
    \]
    Applying \eqref{eq:transit-derivative-bounds} along the segment \(t\mapsto tz\) proves \eqref{eq:transit-cone}.

    Finally, \eqref{eq:eout-close} follows from the continuous dependence of \(e_i^{\rm out}(f)\) on \(f\) established in Lemma~\ref{lem:passage}. Since \(Q<\infty\), the neighborhood of \(g\) can be chosen independently of \(i\).

\end{proof}

In particular, if \(d_i(f)>0\), then \eqref{eq:transit-cone} implies $U_i^f(z)>0,$ so that the corresponding points are mapped into the positive incoming half-section and can continue around the oriented cycle.

\begin{figure}[tb]
    \centering
    \begin{subfigure}{0.43\textwidth}
        \centering
        \includegraphics[width=\textwidth]{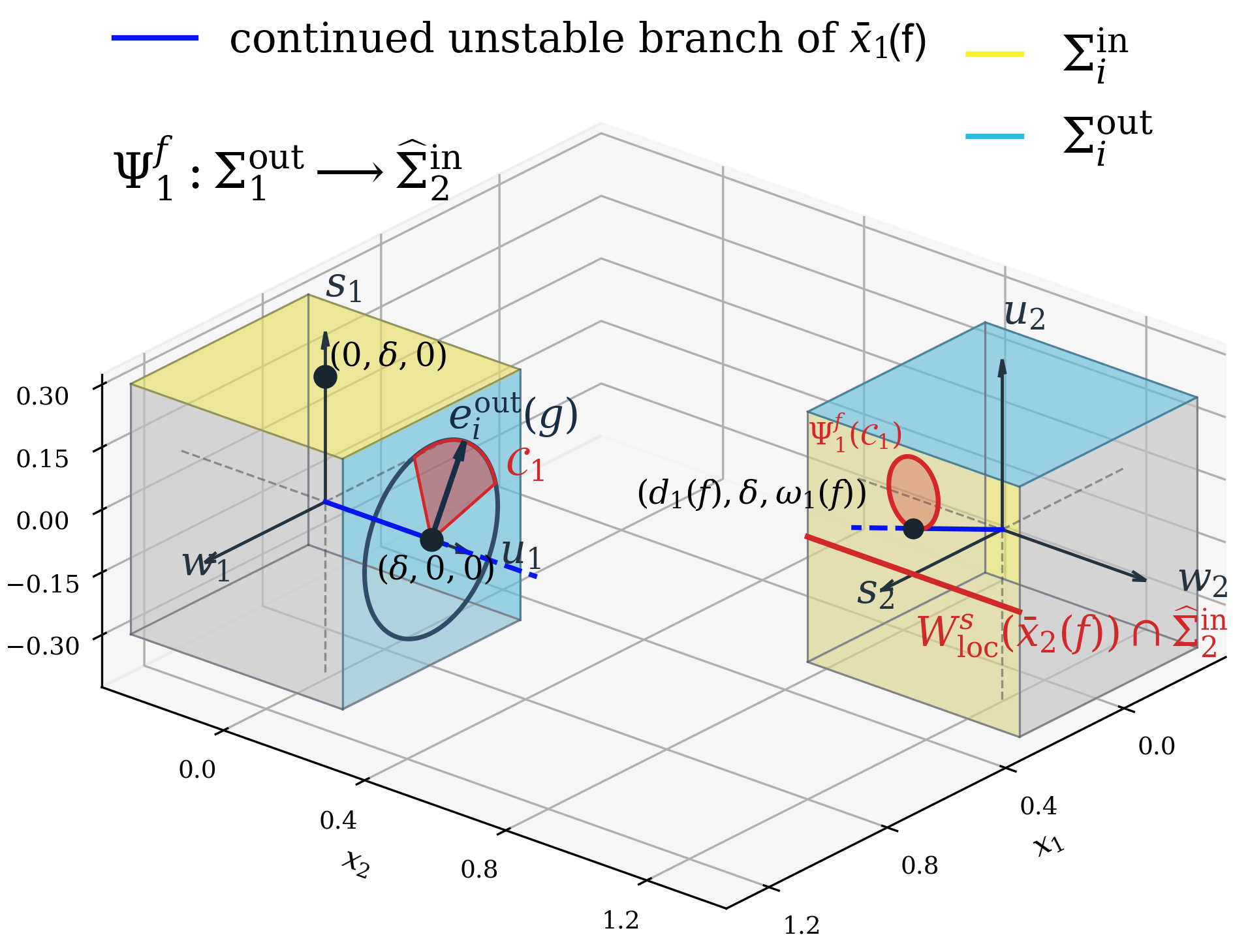}
        \caption{}
    \end{subfigure}
    \hspace{0.05\textwidth}
    \begin{subfigure}{0.43\textwidth}
        \centering
        \includegraphics[width=\textwidth]{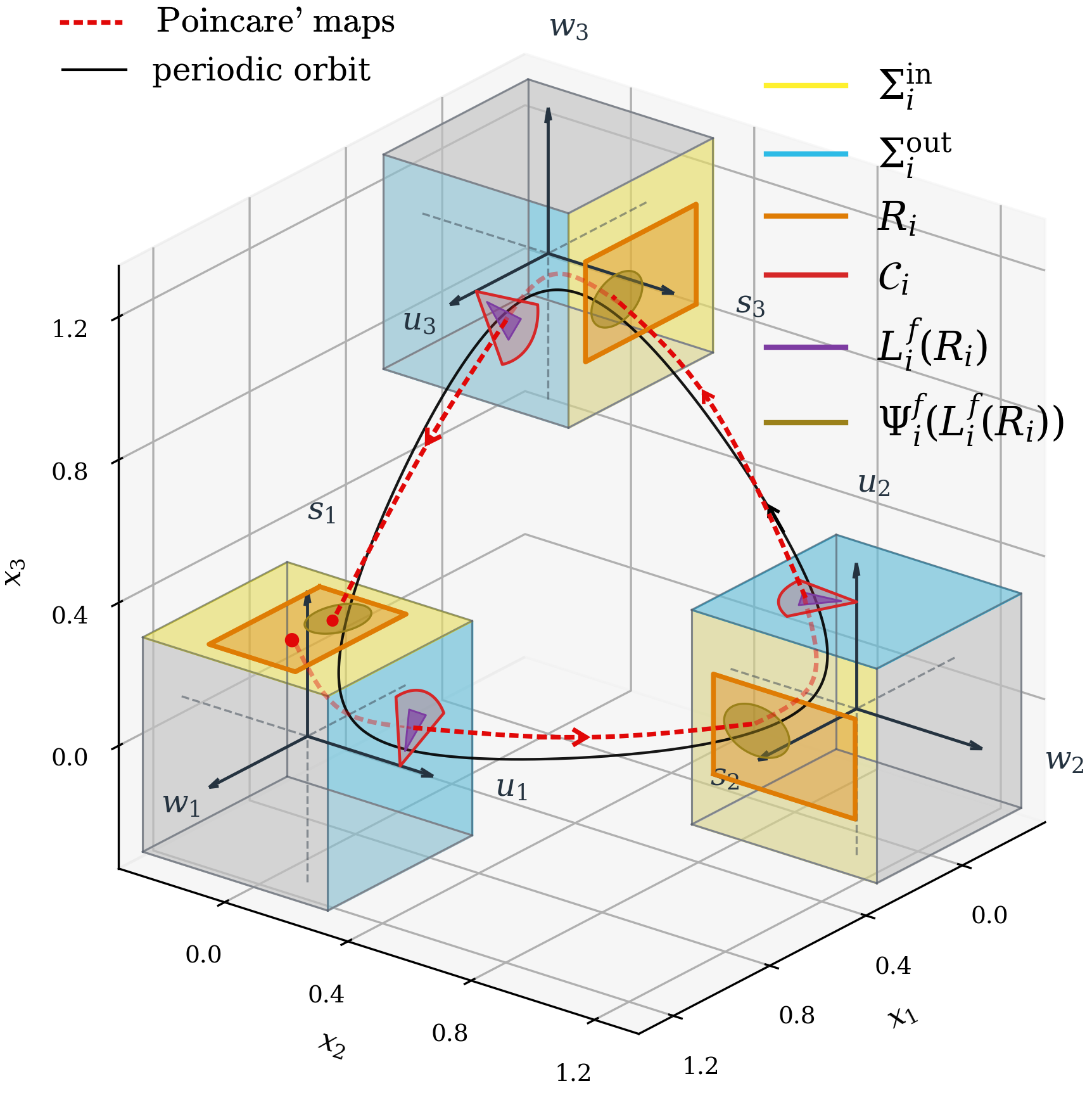}
        \caption{}
    \end{subfigure}
    \caption{  (a) Schematic global transition from $\Sigma_1^{\rm out}$ to $\hat \Sigma_2^{\rm in}$. The landing point of the continued unstable branch is $(d_1(f),\delta,\omega_1(f))$.  The global transition $\Psi_1^f$ maps  $\mathcal C_1$ into $\Psi_1^f(\mathcal C_1)$, a set whose points lie above the red set $W_{\rm loc}^s(\bar x_2(f))\cap\hat \Sigma_2^{\rm in}$. The blue segments and their dashed extensions schematically indicate the continued unstable branch of $\bar x_1(f)$. (b) Geometry underlying the trapping channel. The orange rectangles $R_i$ lie in $\hat\Sigma_i^{\rm in}$. Their local images $L_i^f(R_i)$ (purple) lie in the outgoing cones $\mathcal C_i$ (red). The global images $\Psi_i^f(L_i^f(R_i))=\Phi_i^f(R_i)$ (olive) are strictly contained in the next rectangles $R_{i+1}$. The flow segments joining successive rectangles form $\mathscr C_f$. Contraction of the Poincaré map yields the attracting periodic orbit illustrated in black.  }
    \label{fig:global}
\end{figure}

We now combine the local-passage estimates of Lemma~\ref{lem:passage} with the preceding global-transit estimate of Lemma~\ref{lem:transit}.

\begin{lemma}
    \label{lem:one-step}
    Assume~\ref{A:spectrum}--\ref{A:transit}.
    After possibly shrinking the \(C^1\)-neighborhood of \(g\), for each \(i\), there exist constants $u_i^*\in(0,\delta]$ and $C_{\Delta,i},C_{\Omega,i},C_{\Phi,i}>0$, independent of \(f\), such that the one-step map $\Phi_i^f := \Psi_i^f\circ L_i^f$ is well defined on $\{(u,w): 0<u\le u_i^*,\ |w|\le \rho_i\delta\}$, takes values in $\widehat\Sigma_{i+1}^{\rm in}$, and has the form
    \begin{equation}
        \label{eq:one-step-form}
        \Phi_i^f(u,w)
        =
        \bigl(
            d_i(f)+\Delta_i^f(u,w),
            \omega_i(f)+\Omega_i^f(u,w)
        \bigr).
    \end{equation}
     Moreover, for every $0<u\le u_i^*$ and $|w|\le\rho_i\delta$
    \begin{align}
        0<\Delta_i^f(u,w)
        &\le
        C_{\Delta,i}\delta
        \left(\frac{u}{\delta}\right)^{\beta_i},
        \label{eq:one-step-u}\\
        |\Omega_i^f(u,w)|
        &\le
        C_{\Omega,i}\delta
        \left(\frac{u}{\delta}\right)^{\beta_i},
        \label{eq:one-step-w}\\
        \|D\Phi_i^f(u,w)\|
        &\le
        C_{\Phi,i}
        \left(\frac{u}{\delta}\right)^{\beta_i-1}.
        \label{eq:one-step-derivative}
    \end{align}
\end{lemma}

\begin{proof}
    Fix \(i\). Choose \(u_i^*\in(0,\delta]\), independently of \(f\),
    sufficiently small that
    \begin{equation}
        \label{eq:ustar-conditions}
        C_{\mathrm{dir},i}
        \left(\frac{u_i^*}{\delta}\right)^{\mu_i}
        \le\frac{\eta_i}{2},
        \qquad
        C_{0,i}\delta
        \left(\frac{u_i^*}{\delta}\right)^{\beta_i}
        \le r_i,
    \end{equation}
    where the constants are as in Lemma~\ref{lem:passage}.

    If \(0<u\le u_i^*\) and \(|w|\le\rho_i\delta\), then \eqref{eq:exit-direction}, \eqref{eq:eout-close}, and \eqref{eq:ustar-conditions} give
    \[
    \begin{aligned}
        \left \|
            \frac{z_i^f(u,w)}{\|z_i^f(u,w)\|}
            -
            e_i^{\rm out}(g)
        \right \|
        &\le
        \left \|
            \frac{z_i^f(u,w)}{\|z_i^f(u,w)\|}
            -
            e_i^{\rm out}(f)
        \right \|
        +
        \|e_i^{\rm out}(f)-e_i^{\rm out}(g)\|  \\
        &\le
        C_{\mathrm{dir},i}
        \left(\frac{u}{\delta}\right)^{\mu_i}
        +\frac{\eta_i}{2}
        \le\eta_i.
    \end{aligned}
    \]
    Moreover, \eqref{eq:local-size} gives $\|z_i^f(u,w)\| \le C_{0,i}\delta  \left(\frac{u}{\delta}\right)^{\beta_i} \le r_i$.
    Therefore \(z_i^f(u,w)\) satisfies the hypotheses of Lemma~\ref{lem:transit}.

    Writing \(\Psi_i^f=(U_i^f,W_i^f)\), define
    \[
        \Delta_i^f(u,w)
        :=
        U_i^f\bigl(z_i^f(u,w)\bigr)-d_i(f), \quad \Omega_i^f(u,w)
        :=
        W_i^f\bigl(z_i^f(u,w)\bigr)-\omega_i(f).
    \]
    This gives \eqref{eq:one-step-form}. By Lemma~\ref{lem:transit}, $0  < \Delta_i^f(u,w) \le  C_i\|z_i^f(u,w)\|.$
    Combining this with \eqref{eq:local-size} proves  \eqref{eq:one-step-u}, with $C_{\Delta,i}:=C_iC_{0,i}.$

    The \(C^1\)-continuous dependence of the global transition on \(f\) gives a uniform bound for \(DW_i^f\) on \(\|z\|\le r_i\). Since \(\omega_i(f)=W_i^f(0)\), the mean value theorem and \eqref{eq:local-size} give \eqref{eq:one-step-w}, after choosing \(C_{\Omega,i}\) sufficiently large.

    Finally, \(D\Psi_i^f\) is uniformly bounded on \(\|z\|\le r_i\). By the chain rule and \eqref{eq:local-derivative},
    \[
    \begin{aligned}
        \|D\Phi_i^f(u,w)\|
        &\le
        \left\|
            D\Psi_i^f\bigl(z_i^f(u,w)\bigr)
        \right\|
        \|Dz_i^f(u,w)\|\\
        &\le
        C_{\Phi,i}
        \left(\frac{u}{\delta}\right)^{\beta_i-1},
    \end{aligned}
    \]
    after increasing \(C_{\Phi,i}\) if necessary. This proves \eqref{eq:one-step-derivative}.
\end{proof}

The one-step estimates identify the landing point \((d_i(f),\omega_i(f))\) of each transition and control the displacement from it. When \(d_i(f)>0\), they allow us to choose small rectangles in the incoming sections such that each rectangle is mapped into the next one. The orbit segments joining these rectangles form the trapping channel used below.

Let \(\varphi_t^f(x)\) denote the flow generated by \(f\), whenever it is defined.
With indices taken modulo \(Q\), let \(T_i^f(x)\) be the first hitting time of the next incoming section from \(x\in\Sigma_i^{\rm in}\).
Namely,
\[
    T_i^f(x)
:=\inf\left\{t>0:\varphi_t^f(x)\in\Sigma_{i+1}^{\rm in}\right\}\in(0,+\infty].
\]
In general \(T_i^f(x)\) may be infinite, and \(\varphi_{T_i^f(x)}^f(x)\) may not coincide with $\Phi_i^f(x)$.

\begin{definition}
    \label{def:channel}
    A \emph{trapping channel} for \(f\) around \(\Gamma\) is a set of the form
    \begin{equation}
        \label{eq:channel}
        \mathscr C_f
        := \bigcup_{i=1}^Q\left\{\varphi_t^f(x): x \in R_i,  \quad 0\le t\le T_i^f(x)
        \right\},
    \end{equation}
    where, for \(i=1,\ldots,Q\) and with \(R_{Q+1}:=R_1\), the sets \(R_i\subset\Sigma_i^{\rm in}\) satisfy:
    \begin{enumerate}[label=\textup{(\roman*)}]
        \item \(R_i\) is compact, has nonempty relative interior, and is contained in the domain of the one-step map \(\Phi_i^f\);
        \item \(T_i^f(x)<+\infty\) and \(\Phi_i^f(x)=\varphi_{T_i^f(x)}^f(x)\) for every \(x\in R_i\);
        \item \(\Phi_i^f(R_i)\subset R_{i+1}\).
    \end{enumerate}
\end{definition}

We can now state the main result of this subsection.

\begin{theorem}
    \label{thm:main-periodic}
    Assume \ref{A:spectrum}--\ref{A:transit}, and let \(R>0\).
    For every neighborhood \(U\) of \(\Gamma\), there exists
    \(\varepsilon_U>0\) with the following property.

    If $f\in C^3(\mathbb R^3,\mathbb R^3)$, $\|f\|_{C^3(K)}\le R$, $ \|f-g\|_{C^1(K)}<\varepsilon_U$, and
    \begin{equation}
    \label{eq:positive-splitting}
        d_i(f)>0 \qquad i=1,\ldots,Q
    \end{equation}
    then there exists a trapping channel \(\mathscr C_f\subset U\) such that:
    \begin{enumerate}[label=\textup{(\roman*)}]
        \item\label{main:unique}
        \(f\) has a unique periodic orbit \(\gamma_f\) in \(\mathscr C_f\);

        \item\label{main:stable}
        \(\gamma_f\) is hyperbolic and asymptotically stable, and every forward orbit entering \(\mathscr C_f\) converges to \(\gamma_f\);

        \item\label{main:period}
        if \(f_k\to g\) in \(C^1(K)\), with $\|f_k\|_{C^3(K)}\le R$ and $d_i(f_k)>0$  for every \(i,k\), then the corresponding periods satisfy $T_{f_k}\longrightarrow+\infty$.
    \end{enumerate}
\end{theorem}

\begin{proof}

    Fix a neighborhood $U$ of $\Gamma$. Replacing it by $U\cap\operatorname{int}K$, we may assume that $U\subset\operatorname{int}K$. Choose $\delta\in(0,\delta_0]$ sufficiently small that the local boxes and the selected global transition segments are contained in $U$ and
    \begin{equation}
    \label{eq:omega-small}
    \|\omega_{i-1}(g)\|\leq\frac{\rho_*\delta}{4},
    \qquad i=1,\ldots,Q,
    \end{equation}
    where $\rho_*=\min_i\rho_i>0$ is independent of $\delta$ by Lemma~\ref{lem:passage}. This is possible by Assumption~\ref{A:incoming}, since the incoming connections approach the leading stable directions. For this fixed $\delta$, choose the constants and a common $C^1$-neighborhood of $g$ as in Lemmas~\ref{lem:passage}, \ref{lem:transit}, and \ref{lem:one-step}, and set $h=\rho_*\delta/4$.

    Since \(\beta_i>1\) and \(Q<\infty\), we can choose \(\bar q\in(0,1]\) such that, for each \(i\),
    \begin{equation}
        \label{eq:q-smallness}
        \delta\bar q\le u_i^*,
        \qquad
        C_{\Delta,i}\bar q^{\beta_i-1}\le\frac14,
        \qquad
        C_{\Omega,i}\bar q^{\beta_i}\le\frac{\rho_*}{8},
        \qquad
        C_{\Phi,i}\bar q^{\beta_i-1}\le\frac12.
    \end{equation}

    By continuity of \(d_i(f)\) and \(\omega_i(f)\) at \(f=g\), we may choose \(\varepsilon_U>0\) sufficiently small that
    \begin{equation}
        \label{eq:parameter-smallness}
        |d_i(f)|
        \le
        \frac{\delta\bar q}{2},
        \qquad
        \|\omega_i(f)-\omega_i(g)\|
        \le
        \frac{\rho_*\delta}{4}
    \end{equation}
    for every \(i\). After decreasing \(\varepsilon_U\) further if necessary, continuous dependence of solutions and transverse hitting times on the vector field ensures that all orbit segments defining the global transitions \(\Psi_i^f\) remain in \(U\).

    Set $ m_f:=\max_{1\le i\le Q}d_i(f)>0$ and define the compact rectangles with nonempty relative interior as
    \begin{equation}
        \label{eq:rectangles}
        R_i
        :=
        \left\{
            (u,w)\in\Sigma_i^{\rm in}:
            d_{i-1}(f)\le u\le d_{i-1}(f)+m_f,\
            |w-\omega_{i-1}(f)|\le h
        \right\}.
    \end{equation}
    If \((u,w)\in R_i\), then $0<u\le2m_f$ and $\frac{u}{\delta}\le\frac{2m_f}{\delta}\le\bar q$. Moreover, by \eqref{eq:omega-small} and \eqref{eq:parameter-smallness}, we have
    \[
    \begin{aligned}
        |w|
        &\le
        |w-\omega_{i-1}(f)|
        +
        |\omega_{i-1}(f)-\omega_{i-1}(g)|
        +
        |\omega_{i-1}(g)| \le
        \frac{3\rho_*\delta}{4}
        <
        \rho_i\delta.
    \end{aligned}
    \]
    Thus \(R_i\subset\Sigma_i^{\rm in}\), and Lemma~\ref{lem:one-step} applies on \(R_i\).

    Put $q_f:=\frac{2m_f}{\delta}\le\bar q.$ For \((u,w)\in R_i\), \eqref{eq:one-step-u} and \eqref{eq:q-smallness} give
    \[
    \begin{aligned}
        0<\Delta_i^f(u,w)
        &\le
        C_{\Delta,i}\delta q_f^{\beta_i} =
        2C_{\Delta,i}m_fq_f^{\beta_i-1}
        \le
        \frac{m_f}{2},\\
         |\Omega_i^f(u,w)|
        &\le
        C_{\Omega,i}\delta q_f^{\beta_i}
        \le
        \frac{\rho_*\delta}{8}
        =
        \frac h2.
    \end{aligned}
    \]
    Therefore
    \begin{equation}
        \label{eq:rectangle-invariance}
        \Phi_i^f(R_i)
        \subset
        \operatorname{int}_{\Sigma_{i+1}^{\rm in}}R_{i+1}, \qquad  i=1,\ldots,Q.
    \end{equation}
    By choosing the section patches sufficiently small, the reference orbit segments have no earlier intersections with the next incoming section. By compactness, transversality, and continuous dependence on the vector field, the same holds for nearby trajectories, so that $\Phi_i^f(x)$ is the first intersection with $\Sigma_{i+1}^{\rm in}$ for every $x\in R_i$. The rectangles and the flow pieces joining them consequently form a trapping channel \(\mathscr C_f\subset U\). The cyclic inclusion of the section rectangles and their local and global images is illustrated in Figure~\ref{fig:global}(b).

    By \eqref{eq:one-step-derivative} and \eqref{eq:q-smallness},
    \[
        \|D\Phi_i^f(u,w)\|
        \le\frac12
        \qquad\text{on }R_i.
    \]
    Since each \(R_i\) is convex, the mean value theorem shows that \(\Phi_i^f\) is \(1/2\)-Lipschitz on \(R_i\). Hence the first-return map
    \[
        P^f
        :=
        \Phi_Q^f\circ\cdots\circ\Phi_1^f
        \colon R_1\longrightarrow R_1
    \]
    is a contraction, with Lipschitz constant at most \(2^{-Q}\).
    Banach's fixed-point theorem gives a unique fixed point \(p^*\in R_1\). Since \(P^f\) is the first-return map to \(R_1\), the orbit through \(p^*\) is periodic. Moreover, \eqref{eq:rectangle-invariance} and the definition of the trapping channel imply that $\gamma_f\subset\mathscr C_f$.

    Since \(\Phi_i^f(R_i)\) is contained in the relative interior of \(R_{i+1}\), the fixed point \(p^*\) is an interior point of the domain of the Poincaré map. The two nontrivial Floquet multipliers of \(\gamma_f\) are the eigenvalues of \(DP^f(p^*)\). Since
    \[
        \|DP^f(p^*)\|\le 2^{-Q}<1,
    \]
    both multipliers have modulus strictly smaller than one. Therefore \(\gamma_f\) is hyperbolic and asymptotically stable in $\mathbb R^3$ \cite[Chapter~1]{GuckenheimerHolmes1983}.

    Every forward orbit entering \(\mathscr C_f\) reaches one of the rectangles and thereafter crosses them successively. Its intersections with \(R_1\) are iterates of \(P^f\), and therefore converge to \(p^*\). Hence the corresponding trajectory converges to \(\gamma_f\).

    If another periodic orbit were contained in \(\mathscr C_f\), one of its intersections with \(R_1\) would be a fixed point of \((P^f)^m\) for some \(m\ge1\). Since \((P^f)^m\) is also a contraction and \(p^*\) is one of its fixed points, the two periodic orbits coincide.

    Finally, write $(u_i^{\rm per},w_i^{\rm per}):=\gamma_f\cap\Sigma_i^{\rm in}.$ Since this point belongs to \(R_i\), $0<u_i^{\rm per}\le2m_f$. The local passage-time estimate therefore gives
    \begin{equation}
        \label{eq:period-lower}
        T_f
        \ge
        \left(
            \sum_{i=1}^Q\frac{1}{\lambda_u^i(f)}
        \right)
        \log\frac{\delta}{2m_f}
        -
        C_\tau,
    \end{equation}
    where \(C_\tau\) is independent of \(f\).

    If \(f_k\to g\), then $d_i(f_k)\longrightarrow d_i(g)=0$. Since \(Q<\infty\), $ m_{f_k} =\max_i d_i(f_k) \longrightarrow0$. Furthermore, \(\lambda_u^i(f_k)\to\lambda_u^i>0\). It follows from \eqref{eq:period-lower} that $T_{f_k}\longrightarrow+\infty$.
\end{proof}

\begin{remark}
    \label{rem:companion}
    Assumption~\ref{A:product} requires every saddle to be individually contracting. The companion result \cite{bolelli2026periodic} requires only the global contraction condition
    \[
        \nu(\Gamma)=\prod_i\nu_i >1
    \]
    and allows arbitrary ambient dimension and complex nonleading eigenvalues.
    Under this condition some individual passages may expand, but the cyclic product of the one-step maps is still contracting. The proof uses a different construction of the trapping channel and a weighted product estimate for the iterates of the full return map.
\end{remark}

\begin{remark}
    \label{rem:transit-needed}
   The conditions \(d_i(f)>0\) and Assumption~\ref{A:transit} play distinct roles. The first places the continued unstable branch on the side from which trajectories can proceed to the next equilibrium. The second ensures that displacements in the limiting exit direction are transported to the same side by the global transition. If either condition fails, the conclusion of Theorem~\ref{thm:main-periodic} is no longer guaranteed:
   bifurcations and perturbations of heteroclinic networks and polycycles may create or destroy long-period periodic orbits, or produce several limit cycles \cite{krupa1997robust,kirk2012resonance,dukov2022polycycles}.
\end{remark}

Many of the local ingredients used above are classical, but the results in which they appear rely on additional structure.
See, e.g., \cite{krupa1995asymptotic,krupa1997robust,krupa2004asymptotic,homburg2010homoclinic,postlethwaite2010resonance,kirk2012resonance,dukov2022polycycles}, which exploit symmetry, invariant subspaces, or parametrized families of perturbations.
Theorem~\ref{thm:main-periodic} requires none of these: it applies to every \(C^1\)-small perturbation with a uniform \(C^3\) bound, without symmetry, invariant subspaces or a parametrized family, and it replaces the broken connections by a unique attracting periodic orbit under the sign conditions \(d_i(f)>0\).

\section{Case study: focused attention meditation}
\label{sec:FAM}
We illustrate here the application of the proposed model to the cognitive processes underlying focused attention meditation (FAM).

FAM is a common contemplative practice in which the meditator directs their attention toward a specific object, such as the breath or a visual stimulus, while attempting to resist both internal distractions (e.g., thoughts) and external stimuli (e.g., sounds).
Drawing on the literature on meditation and its neural correlates \cite{ricard2014mind,hasenkamp2012effects,hasenkamp2013using,ganesan2022focused}, we describe FAM as an alternation between three main cognitive phases (see Figure~\ref{fig:meditation}(a)):
\begin{enumerate}
    \item \emph{Sustained attention}, when the practitioner successfully maintains focus on the meditation object;
    \item \emph{Mind wandering}, when attention drifts away from the intended object;
    \item \emph{Awareness of distraction}, when the practitioner realizes the loss of focus and redirects attention back to the meditation object.
\end{enumerate}
Each phase has been associated with a distinct large-scale brain network \cite{hasenkamp2012mind, ricard2014mind, ganesan2022focused} (see Figure~\ref{fig:meditation}(b)):
\begin{itemize}
    \item \emph{Control Network} (CN), active during sustained focus;
    \item \emph{Default Mode Network} (DMN), engaged during mind wandering;
    \item \emph{Salience Network} (SN), recruited when detecting distractions.
\end{itemize}

\begin{remark}
The shifting of attention following awareness of distraction is identified as a separate cognitive phase in \cite{hasenkamp2012mind}. Since this phase engages executive-network regions also involved in sustained attention \cite{hasenkamp2012mind, ricard2014mind, ganesan2022focused}, we incorporate it into the return to sustained focus. %
\end{remark}

We represent this sequence by a heteroclinic cycle connecting three saddle equilibria, each associated with the dominance of one brain network. The connections prescribe the activation order, while the unstable eigenvalues determine the local departure rates from each state.
In the following, we implement this structure in a Lotka–Volterra system and use the resulting vector field as the target for the neural-population construction.

\subsection{Target dynamics: a competitive Lotka--Volterra system}
\label{sec:target}

To encode the cyclic sequence $\mathrm{CN}\longrightarrow\mathrm{DMN} \longrightarrow\mathrm{SN}\longrightarrow\mathrm{CN}$, we introduce the three-dimensional competitive system
\[
    \dot x=g(x),\qquad x=(x_1,x_2,x_3)\in\mathbb R_{\ge0}^3.
\]
Here, the coordinates \(x_1,x_2,x_3\) represent the activities of the CN, DMN, and SN, respectively.
The dynamics are governed by the Lotka--Volterra vector field $g$, whose components are given by
\begin{equation}
    \label{eq:LV}
    g_k(x) = x_k \left(1-\sum_{j=1}^3\rho_{kj}x_j
    \right),\qquad k=1,2,3,
\end{equation}
with
\begin{equation}
    \label{eq:rho}
    \rho_{kj} =
    \begin{cases}
        \dfrac{1}{a_j},& k=j,\\[2mm]
        \dfrac{1-\lambda_u^j}{a_j},& k=j+1,\\[2mm]
        \dfrac{1+\lambda_s}{a_j},& k=j+2,
    \end{cases}
    \qquad\text{indices modulo \(3\),}
\end{equation}
where
\begin{equation}
    \label{eq:LV-parameters}
    a_i>0,\qquad 0<\lambda_u^i<\lambda_s<1, \qquad  i=1,2,3.
\end{equation}
All coefficients $\rho_{kj}$ are positive, and hence \eqref{eq:LV} is competitive. The parametrization is adapted from \cite{horchler2015designing}, where the two stable eigenvalues at each saddle are chosen to coincide. Here, we introduce the parameter $0<\lambda_s<1$ to distinguish the stable eigenvalues, which become $-1$ and $-\lambda_s$. This choice ensures the spectral gap required by Assumption~\ref{A:spectrum}. The original parametrization with repeated stable eigenvalues is recovered when $\lambda_s=1$.

Heteroclinic dynamics in three-species competitive Lotka--Volterra systems are classical, dating back to May and Leonard \cite{may1975nonlinear}. The following proposition verifies the corresponding properties and the additional geometric assumptions required by Theorem~\ref{thm:main-periodic} for our parametrization.
The proof is a direct verification of the assumptions, and is given in Section~\ref{sec:proof-LV-hypotheses} of the supplementary material.

\begin{proposition}
    \label{prop:LV-hypotheses}
    Let \(g\) be given by \eqref{eq:LV}--\eqref{eq:rho} and assume \eqref{eq:LV-parameters}. Then the equilibria $\bar x_i:=a_i e_i$, for $i=1,2,3,$ are connected by a heteroclinic cycle $\Gamma=\bigcup_{i=1}^3\bigl(\{\bar x_i\}\cup\Gamma_i^+\bigr)$, where $\Gamma_i^+\subset W_+^u(\bar x_i)\cap W^s(\bar x_{i+1})$. The cycle is asymptotically stable relative to \(\mathbb R_{\ge0}^3\).
    Moreover, Assumptions~\ref{A:spectrum}--\ref{A:transit} hold, with parameters
    \[
    \lambda_{ss}^i=1,\quad
    \lambda_s^i=\lambda_s, \quad
    \nu_i={\lambda_s}/{\lambda_u^i}>1.
    \]
\end{proposition}

For the field $g$, the parameters $a_i$ determine the locations of the axial equilibria, while the $\lambda_u^i$ govern the local departure rates. Indeed, the passage time
near $\bar x_i=a_ie_i$ satisfies $\tau_i(u)=\frac{1}{\lambda_u^i}\log\frac{\delta}{u}+O(\delta)$, where $u$ is the incoming unstable coordinate. For fixed $\delta$ and $u$, the leading term increases as $\lambda_u^i$ decreases. This motivates the comparison of equal and unequal unstable eigenvalues in Figure~\ref{fig:results1}.

Our goal is to use the metastable dynamics generated by \(g\) as a target for an approximating vector field \(f\) whose dynamics exhibits a nearby cyclic metastable sequence. Indeed, Theorem~\ref{thm:main-periodic} shows that suitable deterministic perturbations replace the heteroclinic cycle by an attracting periodic orbit. The sequence then recurs with finite residence times near the saddles, without stochastic forcing.

To apply Theorem~\ref{thm:main-periodic}, however, the approximating field must satisfy the positive-splitting condition~\eqref{eq:positive-splitting}.
This condition is not easy to implement numerically, since it involves the computation of local and global transition maps of Section~\ref{sec:sequential-periodic}.
For the Lotka--Volterra target considered here, this condition can be replaced by a much simpler criterion expressed directly in the original coordinates, as the following result shows.
The proof is given in Section~\ref{sec:proof-LV-positive-splittings} of the supplementary material.

\begin{proposition}
\label{prop:LV-positive-splittings}
Let $g$ be the Lotka--Volterra vector field~\eqref{eq:LV}, with heteroclinic cycle $\Gamma$ of Proposition~\ref{prop:LV-hypotheses}, and let $K$ be a compact neighborhood of $\Gamma$. For every $R>0$, there exists $\varepsilon_0=\varepsilon_0(K,R)>0$ such that, if $f\in C^3(\mathbb R^3,\mathbb R^3)$ satisfies $\|f\|_{C^3(K)}\le R$, $\|f-g\|_{C^1(K)}<\varepsilon_0$, and
\begin{equation}
\label{eq:positive-plane-condition}
    f_k(x)>0 \qquad \text{for every }x\in K\cap\{x_k=0\}, \quad k=1,2,3,
\end{equation}
then the splittings $d_i(f)$ defined in~\eqref{eq:global-landing} satisfy $d_i(f)>0$ for $i=1,2,3$.
\end{proposition}

Condition~\eqref{eq:positive-plane-condition} asks the approximating field to point into the positive octant along the coordinate planes, which contain the connections of $g$.
It is a three-dimensional counterpart of the net inflow condition of Reyn for separatrix polygons in the plane, under which an attracting polygon generates a stable limit cycle \cite{reyn1980generation}, see also \cite{shaw2012phase}.
The same mechanism underlies the birth of a stable periodic orbit from a homoclinic loop that is split inward \cite[Theorem~13.6]{shilnikov2001methods} and, for cycles whose invariant planes come from a symmetry, the analysis of symmetry-breaking perturbations through Melnikov integrals \cite{sandstede1995forced}.

\subsection{Approximation of the target dynamics}

For the approximation argument, we fix a compact neighborhood of the heteroclinic cycle. More precisely, we take
\[
    K=[-\eta,A+\eta]^3,\qquad\eta>0,\qquad A>\max_i a_i.
\]
It follows that $\Gamma\subset\operatorname{int}K.$

We rely on the Universal Approximation Theorem~\ref{theo:UAT} and Corollary~\ref{cor:UAT} to approximate the target vector field \(g\) on \(K\). We consider approximating fields of the form
\[
    f_\theta(x) := -x+P\sigma(Wx+b),
\]
where
\[
    \theta=(P,W,b),\qquad P\in\mathbb R^{3\times N},\quad W\in\mathbb R^{N\times3}, \quad b\in\mathbb R^N.
\]
By Corollary~\ref{prop:blockP}, the matrix \(P\) may be chosen block-diagonal without affecting the approximation property.

Our aim is to approximate $g$ on $K$ while promoting condition~\eqref{eq:positive-plane-condition}. Proposition~\ref{prop:LV-positive-splittings} and Theorem~\ref{thm:main-periodic} motivate the plane penalty introduced below. The numerical experiments investigate whether the trained fields exhibit recurrent activity with the prescribed activation order.

The network architecture, training procedure, and hyperparameters are described in Section~\ref{sec:training} of the Supplementary Material, and the code is available at~\cite{bolelli2026code}. The training loss combines an approximation term for the vector field and its Jacobian with a penalty promoting the positivity condition~\eqref{eq:positive-plane-condition}.
That is, we minimize the loss
\begin{equation}
\label{eq:numerical-training-loss}
\mathcal L
=
\mathcal L_{C^1}
+
\mathcal L_T.
\end{equation}
Here:
\begin{itemize}
    \item $\mathcal L_{C^1}$ is a mean-square surrogate for the $C^1$ approximation error:
    \begin{equation*}
    \mathcal L_{C^1} = \operatorname{mean}\bigl|f_\theta-g\bigr|^2 + \lambda_J\, \operatorname{mean}\bigl|Df_\theta-Dg\bigr|^2,
    \end{equation*}
    where the means are taken over the sampled points and over the entries of the corresponding vectors and matrices.
    \item $\mathcal L_T$ is a penalty for the violation of the coordinate plane constraints:
    At each optimization step we generate points lying on the coordinate planes, with one coordinate set exactly equal to zero, and add the penalty
    \begin{equation*}
    \mathcal L_T = \lambda_T\, \operatorname{mean} \left[ \max\left\{ m-(f_\theta)_k(x),0 \right\}^2 \right].
    \end{equation*}
    Here $m>0$ is a small positive margin, while $\lambda_T>0$ controls the relative weight of the plane penalty.
\end{itemize}

\subsubsection{Simulation results}
\label{sec:simulation}

We now present the numerical results obtained with the training procedure of Section~\ref{sec:training} of the supplementary material, and relate the observed dynamics to the theoretical framework of Section~\ref{sec:sequential-periodic}.

\begin{figure}[tb]
    \centering
        \makebox[\textwidth][c]{%
        \begin{minipage}{19cm}
            \centering
            \begin{subfigure}[t]{0.48\linewidth}
                \centering
                \includegraphics[width=\linewidth]{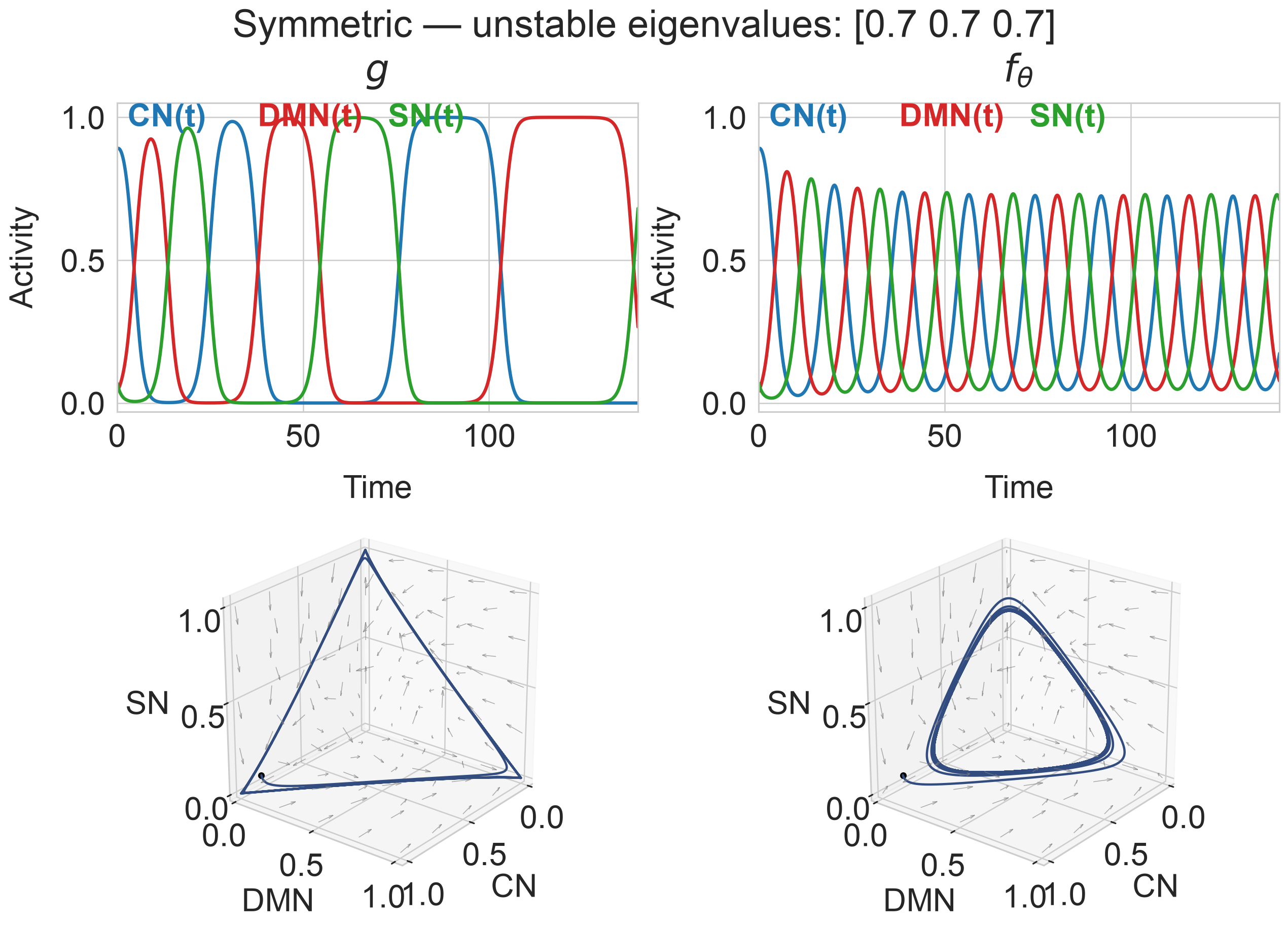}
                \caption{}
            \end{subfigure}\hfill%
            \begin{subfigure}[t]{0.48\linewidth}
                \centering
                \includegraphics[width=\linewidth]{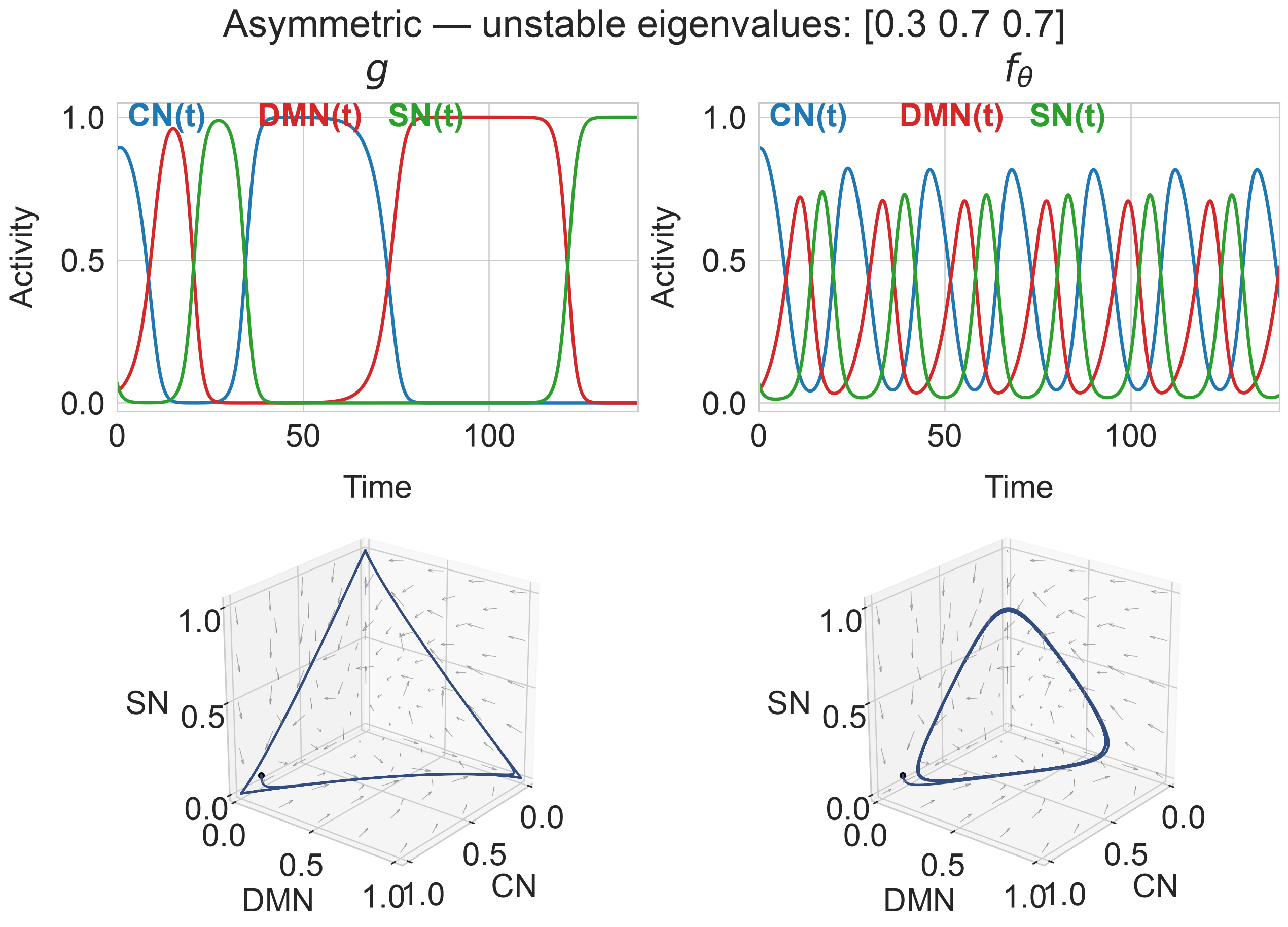}
                \caption{}
            \end{subfigure}
        \end{minipage}%
    }
    \caption{Comparison of the target dynamics $g$ and the learned dynamics $f_\theta$ in the symmetric (a) and asymmetric (b) configurations, with $\lambda_s=0.95$, $a_i=1$, and unstable eigenvalues $(0.7,0.7,0.7)$ and $(0.3,0.7,0.7)$, respectively. All trajectories start from $x(0)=(0.89,0.04,0.07)^T$. Top: network activities on $0\le t\le140$. Bottom: trajectories in phase space and vector-field directions. Black dots indicate the initial condition.}
    \label{fig:results1}
\end{figure}

We consider the Lotka--Volterra field $g$ defined by~\eqref{eq:LV}, with coefficients specified by~\eqref{eq:rho}, $\lambda_s=0.95$, and $a_j=1$ for $j=1,2,3$. We examine two configurations of unstable eigenvalues: the symmetric case $(\lambda_u^1,\lambda_u^2,\lambda_u^3)=(0.7,0.7,0.7)$ and the asymmetric case $(0.3,0.7,0.7)$. The approximation domain is $K=[-0.05,1.15]^3$, corresponding to $\eta=0.05$ and $A=1.10$.

For each configuration, we train a separate network $f_\theta$ with $N=90$ hidden units organized into three blocks of $30$, with the procedure and hyperparameters of Section~\ref{sec:training} of the supplementary material. We use $\sigma=\tanh$ as the activation function.

Figure~\ref{fig:results1} compares the trajectories generated by the original field $g$ and the learned field $f_\theta$ in the symmetric~(a) and asymmetric~(b) configurations. Both systems are initialized at $x(0)=(0.89,0.04,0.07)^T$ and integrated over $0\le t\le500$ using the DOP853 method, with relative tolerance $10^{-9}$, absolute tolerance $10^{-11}$, and maximum step size $0.1$. Solutions are sampled at intervals of $0.05$. The temporal panels in Figure~\ref{fig:results1} display only $0\le t\le140$.

The target dynamics $g$ exhibits sequential activation with increasing residence times near the saddles. By contrast, the learned dynamics $f_\theta$ displays sustained oscillations with approximately repeating activation sequences. The learned trajectories reproduce the activation order and the qualitative geometry of the target trajectories, although differences in amplitude remain.

In the symmetric configuration, the three activation phases have comparable durations. In the asymmetric configuration, the smaller unstable eigenvalue at the CN saddle produces longer CN-dominated phases along the target trajectory and a longer period in the learned dynamics, approximately $\lvPeriodAsym$ compared with $\lvPeriodSym$ in the symmetric case.
Additional numerical checks on the trained fields are reported in Section~\ref{sec:checks} of the supplementary material: approximation errors near $\Gamma$, the equilibria of $f_\theta$ near the saddles, the plane condition~\eqref{eq:positive-plane-condition} on a grid, a sweep of initial conditions, and the Floquet multipliers of the periodic orbit.

These simulations are intended to illustrate the proposed construction rather than to verify the hypotheses of Theorem~\ref{thm:main-periodic} for the trained networks. In particular, the training procedure does not establish the uniform $C^1$ approximation bounds or the positive-splitting conditions required by the theorem.

\subsection{Neural interpretation of the approximated dynamics}

The learned dynamics $\dot x=f_\theta(x)$ has an exact realization as the projection of the higher-dimensional mean-field system
\begin{equation*}
    \dot y=-y+\sigma(WPy+b),
    \qquad y\in\mathbb R^N,
\end{equation*}
as described in Section~\ref{sec:interpretation}.
Here, $y(t)$ describes the activities of the $N$ populations, while $x_\theta(t):=Py(t)$ denotes the corresponding trajectory of the learned reduced dynamics.

To associate these variables with the three functional
networks, we take $P$ with three blocks of
$30$ populations each. The index sets are $I_{\mathrm{CN}}=\{1,\ldots,30\}$, $I_{\mathrm{DMN}}=\{31,\ldots,60\}$, and $I_{\mathrm{SN}}=\{61,\ldots,90\}$. Thus,
\[
 x_i(t)=\sum_{j\in I_i}P_{ij}y_j(t),
 \qquad i\in\{\mathrm{CN},\mathrm{DMN},\mathrm{SN}\}.
\]

\begin{figure}[H]
    \centering
    \makebox[\textwidth][c]{%
        \begin{minipage}{19cm}
            \centering
    \begin{subfigure}{\textwidth}
        \centering
        \includegraphics[width=\textwidth]{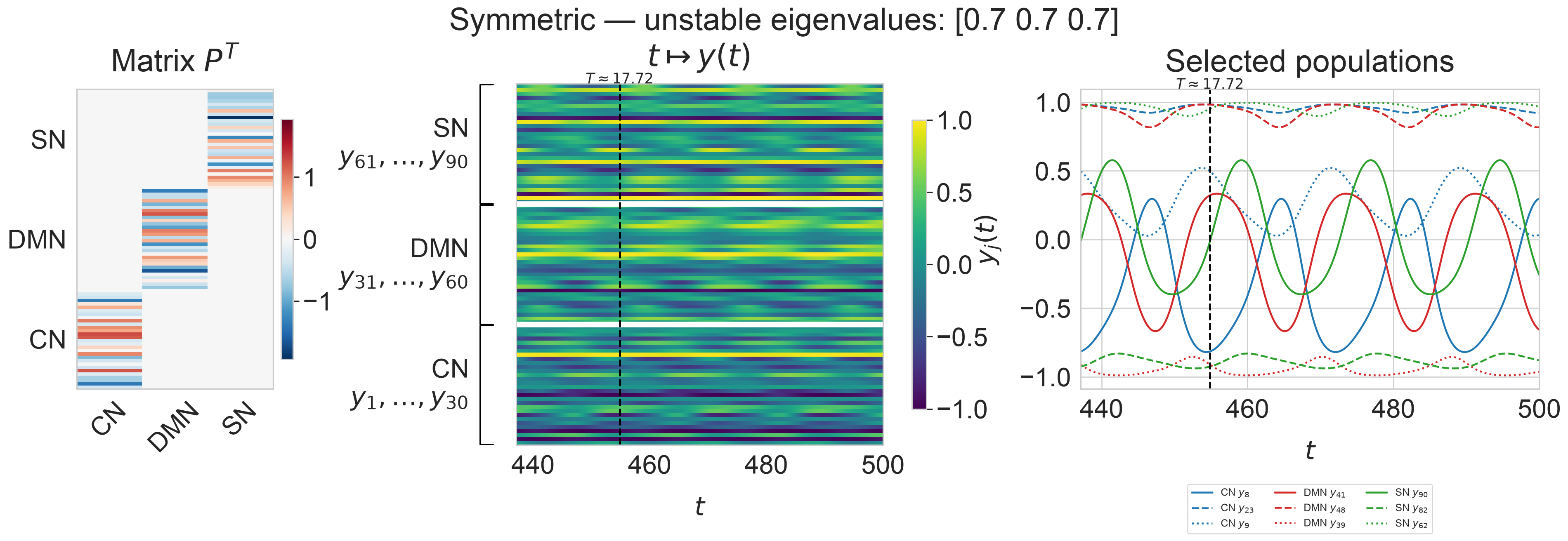}
        \caption{}
    \end{subfigure}
    \begin{subfigure}{\textwidth}
        \centering
        \includegraphics[width=\textwidth]{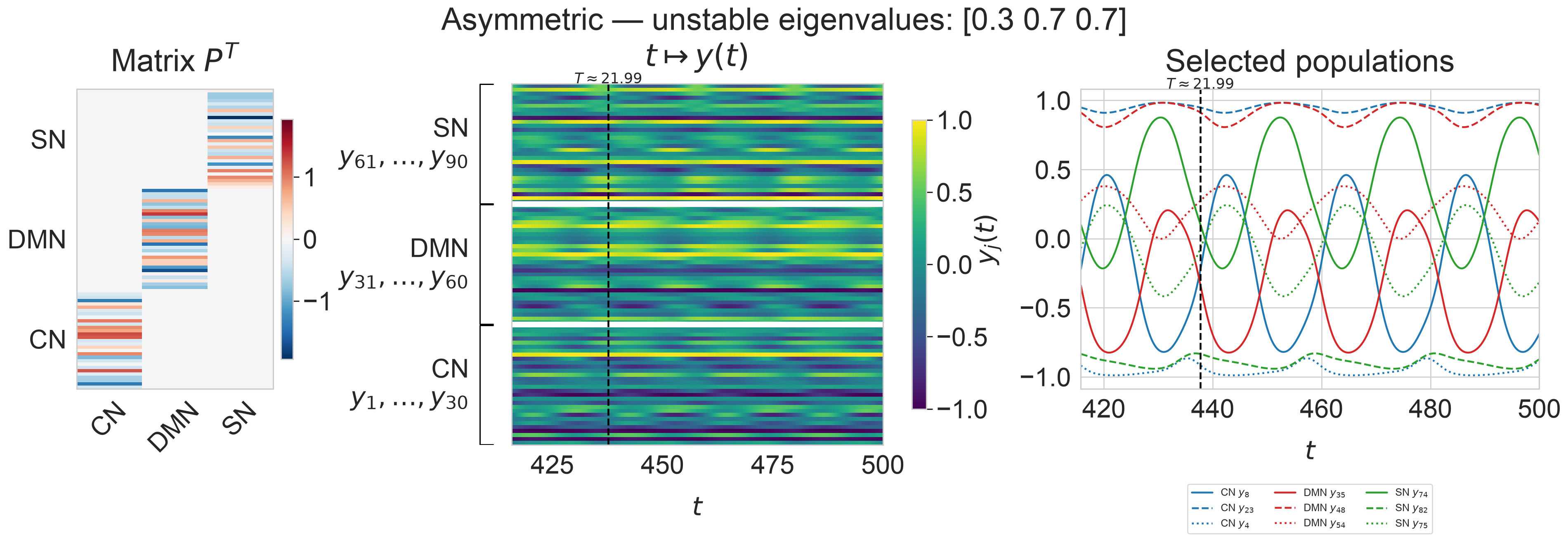}
        \caption{}
    \end{subfigure}
    \end{minipage}
    }
    \caption{ Population realization of the learned dynamics: (a) symmetric and (b) asymmetric configurations of the unstable eigenvalues. Left: the transpose \(P^T\) of the projection matrix, with a diverging color scale indicating signed projection coefficients. Center: population activity \(t\mapsto y(t)\), with one row for each population and the three blocks corresponding to CN, DMN, and SN. Right: representative population trajectories selected from each block, with different line styles distinguishing populations within the same network.
    The dashed vertical line marks $t_0+T$, where $t_0$ is a reference upward section crossing and $T$ is the numerically estimated recurrence period ($T\approx\lvPeriodSym$ in (a) and $T\approx\lvPeriodAsym$ in (b)).}
    \label{fig:neural}
\end{figure}

Figure~\ref{fig:neural} shows the projection coefficients of $P$, the population activities $t\mapsto y(t)$, and representative population trajectories of the two simulations of Section~\ref{sec:simulation}, initialized so that $Py(0)=x_\theta(0)$. The individual populations within each block have different amplitudes, phases, and waveforms, while their weighted projections recover the sequential activity. The maximum numerical discrepancy between $Py(t)$ and the independently integrated reduced trajectory $x_\theta(t)$ is approximately $\lvProjErrorSym$ in the symmetric case and $\lvProjErrorAsym$ in the asymmetric case.
The CN and SN blocks of the two networks are identical, as the matrices $P^T$ of Figure~\ref{fig:neural} show, and the longer CN phase of the asymmetric configuration is thus produced entirely by the DMN block.

For the lifted periodic orbit, Proposition~\ref{prop:periodic-lifting} ensures that the period of the projected dynamics is also the period of the full population state.
To estimate the recurrence period shown in Figure~\ref{fig:neural}, we consider the component of $x_\theta(t)$ with the largest peak-to-peak amplitude on $[325,500]$. We record its upward crossings of the midpoint between its minimum and maximum on that interval and take the median of the last five return intervals. This gives $T\approx\lvPeriodSym$ in the symmetric case and $T\approx\lvPeriodAsym$ in the asymmetric case. The relative spread of the return intervals is $\lvPeriodSpreadSym$ and $\lvPeriodSpreadAsym$, respectively, while the maximum relative difference between successive crossing states is $\lvReturnErrorSym$ and $\lvReturnErrorAsym$. These small errors provide numerical evidence that the trajectories have approached a periodic regime.

This example illustrates how a recurring macroscopic sequence can coexist with varied activity across the populations associated with each network. Determining which physiological constraints and empirical observations can distinguish among these realizations, and whether their temporal asymmetries reflect changes in attentional state, remains a direction for future research.

\section{Summary and conclusions}

This work proposes a dynamical framework for the study of cyclic and sequential neural activity patterns. Our aim is to bridge phenomenological descriptions of cyclic sequential activity, in particular heteroclinic dynamics, with spatially discrete neural-field equations. To this end, we combine dynamical-systems analysis, universal approximation, and machine-learning techniques.

Our first contribution concerns the direct realization of prescribed sequential states within a spatially discrete neural-field model. Under the structural and positivity assumptions of Theorem~\ref{thm:main}, prescribing the desired activity patterns as equilibria forces the connectivity matrix to satisfy a detailed-balance condition. This implies the existence of a strict Lyapunov function that excludes heteroclinic cycles and nonconstant periodic orbits. Theorem~\ref{thm:robust} and Corollary~\ref{cor:no-sequential} further show that, under suitable additional assumptions, sufficiently small autonomous perturbations of the connectivity matrix cannot produce periodic orbits or heteroclinic cycles visiting distinct prescribed states. On the other hand, the explicit cyclic-inhibition construction of Theorem~\ref{thm:ring-main} shows that spatially discrete neural-field equations can support periodic dynamics without detailed balance. The obstruction therefore concerns the prescribed equilibrium configurations covered by Theorem~\ref{thm:main}, rather than the neural-field formulation itself.

In the second part of the paper, we overcome this obstruction by representing the sequential states as low-dimensional aggregates of a higher-dimensional population system. The construction relies on a \(C^1\)-version of the Universal Approximation Theorem (Section~\ref{sec:UATresults}), which allows a prescribed target vector field to be approximated by a dynamics admitting an exact realization as an Amari-type neural-field system. The projection can moreover be organized into distinct population blocks, preserving a natural interpretation in terms of neural populations. Periodic solutions of the reduced dynamics then lift to periodic solutions of the population system with the same minimal period.

The situation is subtler for heteroclinic target dynamics. Here we focus on the three-dimensional case. Under additional geometric and stability assumptions on the target system, Theorem~\ref{thm:main-periodic} shows that a sufficiently small \(C^1\) perturbation, subject to a uniform \(C^3\) bound and suitable positive-splitting conditions, admits a unique hyperbolic attracting periodic orbit contained in a trapping channel surrounding the original heteroclinic cycle. Thus, although the heteroclinic cycle itself is not structurally stable under generic perturbations, its cyclic ordering persists in the approximating dynamics through a nearby attracting periodic orbit. For a Lotka--Volterra target system, positivity of the appropriate normal components along the coordinate planes provides a sufficient condition for positive splitting and motivates the plane penalty used in the training procedure.

The FAM example illustrates how the proposed framework can connect a phenomenological sequence of large-scale network activations with a higher-dimensional neural-population realization. The numerical experiments show that the sequential activity is reproduced at the aggregated level, while individual populations exhibit heterogeneous temporal profiles. By construction, the block structure of the projection matrix $P$ associates each macroscopic variable with a distinct group of populations, and the sequential dynamics emerges from their weighted activities.

More broadly, the framework suggests that macroscopic functional states need not correspond to equilibria of the underlying population dynamics. Sequential behavior may instead emerge at the level of suitable low-dimensional observables of a higher-dimensional neural system.

Future work will focus on incorporating empirical constraints into the population level realization. In particular, it will be important to determine which realizations are compatible with measured neural activity, how biologically motivated connectivity constraints restrict the class of admissible approximations, and how variations in the target dynamics are reflected in the organization of the underlying population system.

\section*{Code availability}
The code that produces the computed figures and numerical values of this paper is available at~\cite{bolelli2026code}.

\section*{Acknowledgments}
The authors wish to thank Hugues Mounier for helpful discussions.
This work was supported by the French National Research Agency (ANR)
under grants ANR-17-CE40-0005 (MindMadeClear) and ANR-20-CE48-0003
(RubinVase).

\section*{Declaration of generative AI and AI-assisted technologies in the manuscript preparation process}
During the preparation of this work the authors used Claude (Anthropic) and ChatGPT (OpenAI) in order to edit the text, check proofs, and to assist in writing the code for the numerical experiments.
After using these tools, the authors reviewed and edited the content as needed and take full responsibility for the content of the published article.

\clearpage
\phantomsection
\addcontentsline{toc}{section}{Supplementary material}
\begin{center}
    {\LARGE\bfseries Supplementary material\par}
    \medskip
    {\large for ``Modeling sequential cognitive states via population-level cortical dynamics''\par}
\end{center}
\medskip

\setcounter{section}{0}
\renewcommand{\thesection}{S\arabic{section}}
\renewcommand{\theHsection}{S.\arabic{section}}

\section{Cyclic behavior without detailed balance assumptions}
\label{sec:cyclic}

We prove here Theorem~\ref{thm:ring-main}. It is a consequence of the following more precise statement, which we then discuss and illustrate numerically.

\begin{theorem}
    \label{thm:ring}
    For $n = 3$, let $c>0$, $b = (b_0,b_0,b_0)$ with $b_0>0$, and $W$ as in~\eqref{eq:ring-W}.
    Then, the only equilibrium of~\eqref{eq:WC} is $\bar x = (x^\ast,x^\ast,x^\ast)$, where $x^\ast\in (0,\beta)$ is the unique solution of $x = \sigma(-cx+b_0)$.
    Set
    \begin{equation}
        \label{eq:ring-psi}
        \psi(c) := c\,\sigma'\big(\sigma^{-1}(x^\ast)\big) .
    \end{equation}
    Moreover, it holds that
    \begin{enumerate}[label=\textup{(\roman*)}]
        \item If $\psi(c) < 2$ then $\bar x$ is locally asymptotically stable.
        \item If $\psi(c) > 2$ then, for every initial condition outside the diagonal $\Delta := \{x_1=x_2=x_3\}$, the $\omega$-limit set of the solution is a nonconstant periodic orbit. In particular, \eqref{eq:WC} admits a nonconstant periodic orbit.
    \end{enumerate}
    Finally, if $\sigma$ is odd and concave on $[0,\infty)$, there exists a unique $c^\ast>0$ with $\psi(c^\ast) = 2$, which separates the regimes~(i) and~(ii).
\end{theorem}

\begin{proof}[Proof of Theorem~\ref{thm:ring}]

    Set $g(t):=\sigma(-ct+b_0)$. Since $c>0$ and $\sigma'>0$, the function $g$ is strictly decreasing, and hence so is $g(t)-t$. Moreover,
    \[
        g(0)=\sigma(b_0)>0,
        \qquad
        g(\beta)-\beta<0,
    \]
    because $\sigma$ takes values in $(-\alpha,\beta)$. Thus $g$ has a unique fixed point $x^\ast\in(0,\beta)$.

    An equilibrium satisfies $x_i=g(x_{i-1})$, with indices modulo $3$, and therefore $x_1$ is a fixed point of $g\circ g\circ g$. Since $g\circ g\circ g$ is strictly decreasing, it has at most one fixed point. As $x^\ast$ is a fixed point of both $g$ and $g\circ g\circ g$, the only equilibrium is $\bar x =(x^\ast,x^\ast,x^\ast)$.

    Next, we claim that the eigenvalues of $Df(\bar x)$ are
    \begin{equation}
        \label{eq:eigenvalues-psi}
        -1-\psi(c)
        \qquad\text{and}\qquad
        -1+\tfrac12\psi(c) \pm i\,\tfrac{\sqrt3}{2}\,\psi(c) .
    \end{equation}
    Indeed, $f_i(x) = -x_i + \sigma(-cx_{i-1}+b_0)$, so at $\bar x$ we have $\partial_{x_i}f_i = -1$, $\partial_{x_{i-1}}f_i = -c\,\sigma'(\sigma^{-1}(x^\ast)) = -\psi(c)$, and every other partial derivative vanishes. Hence $Df(\bar x) = -\idty - \psi(c)P$, where $P$ is the cyclic permutation matrix with $P_{i,i-1}=1$. The eigenvalues of $P$ are the cube roots of unity $1$ and $e^{\pm2\pi i/3} = -\tfrac12 \pm i\tfrac{\sqrt3}{2}$, which gives \eqref{eq:eigenvalues-psi}.

    Point \emph{(i)} follows, since the real eigenvalue is negative and the complex pair has negative real part exactly when $\psi(c)<2$.

    For point \emph{(ii)}, assume that $\psi(c)>2$.
    We start by showing that $W^s(\bar x) = \Delta$. The diagonal $\Delta$ is invariant, and on it the dynamics reduces to the scalar equation $\dot s=h(s)$ with $h(s):=-s+g(s).$ The function $h$ is strictly decreasing and vanishes only at $x^\ast$.
    Hence every point of $\Delta$ converges to $\bar x$ in forward time, so $\Delta\subset W^s(\bar x)$.
    By~\eqref{eq:eigenvalues-psi}, $\bar x$ is hyperbolic and its stable eigenspace is one-dimensional, spanned by $(1,1,1)$. The local stable manifold is therefore one-dimensional and, by uniqueness of the local stable manifold, coincides locally with $\Delta$. By invariance, this gives $W^s(\bar x)=\Delta$. Moreover, every nonconstant orbit on $\Delta$ leaves any sufficiently small neighborhood of $\bar x$ in backward time. Hence no nonconstant orbit in $\Delta$ is homoclinic to $\bar x$. Since a homoclinic orbit would lie in $W^s(\bar x)=\Delta$, the system has no homoclinic orbit to $\bar x$.

    System~\eqref{eq:WC} with the connectivity~\eqref{eq:ring-W} is a monotone cyclic negative feedback system, since $f_i$ depends only on $x_i$ and $x_{i-1}$ and $\partial_{x_{i-1}}f_i = -c\,\sigma'(\cdot) < 0$ everywhere. Its forward orbits are bounded. By the Poincaré--Bendixson theorem for such systems~\cite{malletparet1990poincare}, and since $\bar x$ is the unique equilibrium, the $\omega$-limit set of a bounded forward orbit is either $\{\bar x\}$, a nonconstant periodic orbit, or $\{\bar x\}$ together with homoclinic orbits to $\bar x$. The last possibility has just been excluded. Moreover,
    \[
        \omega(x_0)=\{\bar x\}
        \quad\Longrightarrow\quad
        x_0\in W^s(\bar x)=\Delta.
    \]
    Hence, if $x_0\notin\Delta$, its $\omega$-limit set is a nonconstant periodic orbit.

    Finally, assume that $\sigma$ is odd and concave on $[0,\infty)$, and write $x^\ast=x^\ast(c)$. Consider $ F(c,x):=x-\sigma(-cx+b_0)$. At $(c,x^\ast(c))$, we have $\partial_xF =1+c\,\sigma'(-cx^\ast+b_0)>0,$ so the implicit function theorem shows that $x^\ast(c)$ is continuously differentiable. Moreover,
    \[
        \frac{d x^\ast}{dc}
        =
        -\frac{x^\ast\,\sigma'(-cx^\ast+b_0)}
        {1+c\,\sigma'(-cx^\ast+b_0)}
        <0.
    \]
    Thus $x^\ast(c)$, and hence $\sigma^{-1}(x^\ast(c))$, is strictly decreasing and remains positive. Since $\sigma$ is odd, $\sigma'$ is even, and since $\sigma$ is concave on $[0,\infty)$, $\sigma'$ is nonincreasing there. Consequently $ c\longmapsto \sigma'\bigl(\sigma^{-1}(x^\ast(c))\bigr)$ is nondecreasing and positive, and therefore $\psi$ is continuous and strictly increasing. Furthermore, $0<\psi(c)\le c\,\sigma'(0)\longrightarrow0$ as $ c\to0$. On the other hand, the equilibrium equation gives $ \sigma^{-1}(x^\ast)=b_0-cx^\ast>0,$    and hence $0<x^\ast(c)<\frac{b_0}{c}.$ Therefore $x^\ast(c)\to0$ as $c\to\infty$, so $\sigma^{-1}(x^\ast(c))\to0$ and $\sigma'\bigl(\sigma^{-1}(x^\ast(c))\bigr) \longrightarrow\sigma'(0)>0$. It follows that $\psi(c)\to\infty$ as $c\to\infty$. By continuity and strict monotonicity, there exists a unique $c^\ast>0$ such that $\psi(c^\ast)=2$.

\end{proof}

\begin{proof}[Proof of Theorem~\ref{thm:ring-main}]
    By Theorem~\ref{thm:ring}\,(ii), it suffices to show that $\psi(c)>2$ for every $c$ large enough. Since $x^\ast>0$ and $\sigma(0)=0$, the equilibrium equation gives $\sigma^{-1}(x^\ast)=b_0-cx^\ast>0$, hence $0<x^\ast<b_0/c$. Therefore $\sigma^{-1}(x^\ast)\to0$ as $c\to\infty$, and the continuity of $\sigma'$ yields $\sigma'\bigl(\sigma^{-1}(x^\ast)\bigr)\to\sigma'(0)>0$. Thus $\psi(c)\to\infty$ as $c\to\infty$, and there exists $c_0>0$ such that $\psi(c)>2$ for every $c>c_0$.
\end{proof}

\begin{remark}
   The system~\eqref{eq:WC}--\eqref{eq:ring-W} is a three-dimensional dissipative, irreducible and competitive system. It has a unique equilibrium $\bar x$, with $\det Df(\bar x)=-(1+\psi(c)^3)<0$. Since $\bar x$ is unstable when $\psi(c)>2$, the results of Zhu and Smith~\cite{zhuSmith1994} imply the existence of an orbitally stable periodic orbit. If $\sigma$ is real-analytic, as for $\sigma=\tanh$, an orbitally asymptotically stable periodic orbit exists. These results do not imply uniqueness.
\end{remark}

\begin{figure}[tb]
    \centering
    \includegraphics[width=\linewidth]{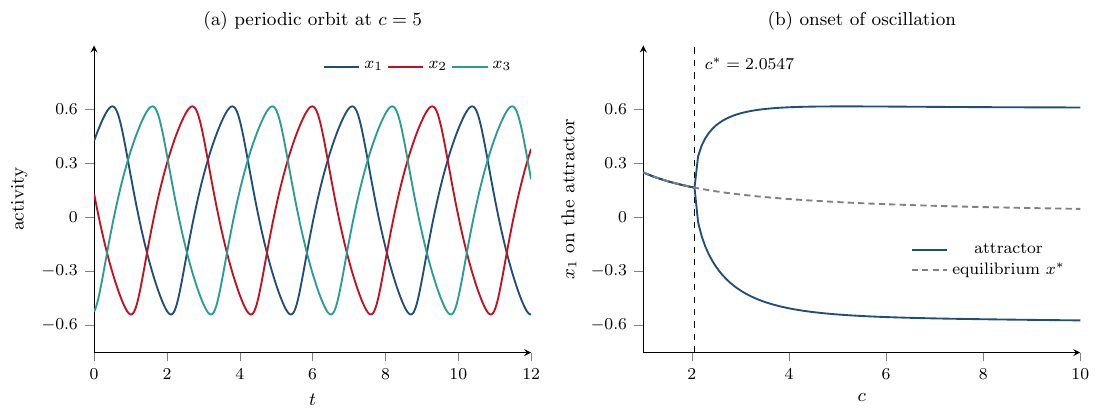}
    \caption{Cyclic inhibition with $\sigma=\tanh$ and $b_0=1/2$, in the setting of Theorem~\ref{thm:ring}. (a) The three coordinates on the numerically observed periodic orbit at $c=5$. Numerically, the coordinates are shifted by approximately one third of the period, consistent with the cyclic symmetry of the system. (b) Minimum and maximum of $x_1$ on the numerically observed periodic orbit as $c$ varies, together with the equilibrium branch $x^\ast(c)$. The equilibrium is locally asymptotically stable for $c<c^\ast$ and loses stability at $c=c^\ast=\ringCstar$. For $c>c^\ast$, Theorem~\ref{thm:ring} implies that every off-diagonal trajectory has a nonconstant periodic orbit as its $\omega$-limit set.}
    \label{fig:ring-cycle}
\end{figure}

Figure~\ref{fig:ring-cycle} illustrates the statement for $\sigma = \tanh$ and $b_0 = 1/2$, where the threshold of Theorem~\ref{thm:ring} is $c^\ast = \ringCstar$. At each of $c = 2.1$, $c = 5$ and $c = 10$ we integrated $\ringSweepCount$ initial conditions off the diagonal. Numerically, all of them converged to the same periodic orbit, the extreme values of $x_1$ agreeing to within $\ringSweepSpread$. The peak-to-peak range of that orbit in each coordinate grows from $\ringRangeAtTwoPointOne$ at $c = 2.1$ through $\ringRangeAtFive$ at $c = 5$ and $\ringRangeAtTen$ at $c = 10$ to $\ringRangeAtForty$ at $c = 40$, while its period falls from $\ringPeriodAtTwoPointOne$ to $\ringPeriodAtForty$ over the same range.

\section{Local spectral analysis}
\label{app:counting-eigenvalues}
We prove here Theorem~\ref{thm:local} by a direct spectral analysis of the Jacobian of the vector field~\eqref{eq:WC} at the prescribed equilibria. We use throughout the set $S_i$ of populations activated by the $i$-th pattern, defined in~\eqref{eq:support}, and write $|S_i|$ for its cardinality.

\begin{lemma}
    \label{lem:eigenvalues}
    Under the assumptions of Theorem~\ref{thm:local}, the Jacobian matrix $Df(\bar{x}_i)$ has real eigenvalues, at least $|S_i|$ of which are negative and at least $n-|S_i|-1$ of which are positive.
\end{lemma}

\begin{proof}
    The Jacobian of the vector field $f$ of~\eqref{eq:WC} at $\bar{x}_i$ is
    \begin{equation*}
        Df(\bar x_i)
        = -\idty + \Sigma W,
        \qquad
        \Sigma = \diag_j \left(\sigma'(\sigma^{-1}(\bar X_{ji}))\right),
    \end{equation*}
    where the second identity holds because $(W\bar x_i + b)_j = \sigma^{-1}(\bar X_{ji})$ at an equilibrium.
    Since $W = A - b u^\top$ (see~\eqref{ref:W-structure}) and the assumptions of Theorem~\ref{thm:main} make $A$ diagonal, $Df(\bar x_i)$ is the difference of the diagonal matrix $-\idty + \Sigma A$ and the rank one matrix $\Sigma bu^\top$. We let $\lambda_j$ be the diagonal entries of the matrix $\Sigma A$, namely
    \begin{equation}
        \label{eq:eigen-sigmaA}
        \lambda_j =
        \frac{\sigma^{-1}(a_j) \sigma'(\sigma^{-1}(\bar X_{ji}))}{a_j}.
    \end{equation}
    By the assumptions on $\sigma$, we have $\lambda_j<1$ for $j \in S_i$ and $\lambda_j>1$ for $j \notin S_i$. Indeed, writing $s_j := \sigma^{-1}(a_j)$, so that $\sigma(s_j)/s_j = a_j/s_j>0$, the inequality $\sigma'(s_j)<\sigma(s_j)/s_j$ gives $\lambda_j = \sigma'(s_j)\,s_j/a_j<1$ when $\bar X_{ji}=a_j$, while $\sigma(s_j)/s_j<\sigma'(0)$ gives $\lambda_j = \sigma'(0)\,s_j/a_j>1$ when $\bar X_{ji}=0$. Neither computation is sensitive to the sign of $a_j$. In particular no $\lambda_j$ equals $1$, so the diagonal matrix $-\idty+\Sigma A$ is invertible.

    Since $Df(\bar x_i)$ is a diagonal matrix minus a rank one matrix, its characteristic polynomial $p(\lambda) = \det (Df(\bar x_i)-\lambda\idty)$ follows from the Matrix Determinant Lemma\footnote{It states that $\det(M+uv^\top)=(1+v^\top M^{-1}u) \det (M)$, whenever $M$ is an invertible square matrix and $u,v$ are vectors.}. For $\lambda+1\notin\{\lambda_1,\ldots,\lambda_n\}$, the matrix $\Sigma A-(\lambda+1)\idty$ is invertible, and we obtain
    \begin{equation}
        p(\lambda) = \varphi(\lambda+1)\det \left(\Sigma A -(\lambda+1)\idty \right),
    \end{equation}
    where we let
    \begin{equation}
        \label{eq:phi}
        \varphi(\mu) = 1 + \sum_{j=1}^n \frac{\alpha_j}{\mu - \lambda_j},
        \qquad
        \alpha_j = u_j b_j \,\sigma'(\sigma^{-1}(\bar X_{ji})).
    \end{equation}
    Observe that $\alpha_j> 0$, since $b_j> 0$ and $u_j>0$ by~\eqref{eq:u-positive}.

    Let us denote by $\tilde\lambda_1 < \tilde\lambda_2<\ldots<\tilde\lambda_k$ the $k$ distinct quantities in $\{\lambda_1,\ldots,\lambda_n\}$, by $m_\ell$ the number of indices $j$ with $\lambda_j = \tilde\lambda_\ell$, and by $k'$ the number of those values that are smaller than $1$, so that $m_1+\cdots+m_{k'} = |S_i|$.
    Since $\alpha_j>0$ for every $j$, for each distinct value $\tilde\lambda_\ell$ we have
    \[
        \beta_\ell
        :=
        \sum_{\{j:\lambda_j=\tilde\lambda_\ell\}}\alpha_j
        >0.
    \]
    Grouping equal values of the $\lambda_j$, we can write
    \[
        \varphi(\mu)
        =
        1+\sum_{\ell=1}^k
        \frac{\beta_\ell}{\mu-\tilde\lambda_\ell}.
    \]
    Therefore $\varphi$ is the quotient of two polynomials whose numerator has degree exactly $k$, so it has exactly $k$ zeros, possibly complex.
    Since $\beta_\ell>0$ for every $\ell$, we have
    \begin{equation}
        \lim_{\mu\to \pm\infty} \varphi(\mu) = 1,
        \qquad
        \lim_{\mu\to \tilde \lambda_\ell^-} \varphi(\mu) = -\infty,
        \qquad\text{and}\qquad
        \lim_{\mu\to \tilde \lambda_\ell^+} \varphi(\mu) = +\infty.
    \end{equation}
    By continuity, $\varphi(\mu)$ has one zero $z_1 \in (-\infty,\tilde\lambda_1)$ and one zero $z_\ell \in (\tilde\lambda_{\ell-1},\tilde\lambda_\ell)$ for $\ell=2,\ldots, k$, and these are all of its zeros.

    The determinant identity above was obtained away from the poles of $\varphi$. After putting its right-hand side over a common denominator, it extends as an identity of polynomials. At a pole $\tilde\lambda_\ell$, the simple pole of $\varphi$ cancels one of the $m_\ell$ factors coming from $\det(\Sigma A-(\lambda+1)\idty)$. Therefore $\tilde\lambda_\ell-1$ is an eigenvalue of $Df(\bar x_i)$ with multiplicity $m_\ell-1$.

    The spectrum of $Df(\bar x_i)$ is therefore made of the $k$ numbers $z_\ell-1$ together with $\tilde\lambda_\ell - 1$ repeated $m_\ell-1$ times, which accounts for $k + \sum_\ell (m_\ell-1) = n$ eigenvalues. Among them, $z_\ell-1<0$ for $\ell \le k'$, because $z_\ell<\tilde\lambda_\ell\le\tilde\lambda_{k'}<1$, and $\tilde\lambda_\ell-1<0$ for $\ell\le k'$: this gives $k' + \sum_{\ell\le k'}(m_\ell-1) = |S_i|$ negative eigenvalues. Likewise $z_\ell-1>0$ for $\ell\ge k'+2$, because $z_\ell>\tilde\lambda_{\ell-1}\ge\tilde\lambda_{k'+1}>1$, and $\tilde\lambda_\ell-1>0$ for $\ell>k'$: this gives $(k-k'-1) + \sum_{\ell>k'}(m_\ell-1) = n-|S_i|-1$ positive ones. If $k'=k$, that is, if the $i$-th pattern activates every population, this already accounts for all $n$ eigenvalues, which are then negative. Otherwise only the sign of the remaining zero $z_{k'+1}$ is left undetermined, namely the one in the interval $(\tilde\lambda_{k'},\tilde\lambda_{k'+1})$ that contains $1$. In particular, all the eigenvalues described above are real, which completes the proof.
\end{proof}

\begin{remark}
    \label{rem:phi-one}
    Assume $k'<k$, so that one eigenvalue is left undetermined by Lemma~\ref{lem:eigenvalues}. Since $\varphi$ has a single zero in the interval $(\tilde\lambda_{k'},\tilde\lambda_{k'+1})$ containing $1$, and tends to $+\infty$ at its left endpoint, the undetermined eigenvalue $z_{k'+1}-1$ is positive precisely when $\varphi(1)>0$. Under that condition there are exactly $|S_i|$ negative and $n-|S_i|$ positive eigenvalues, and $n-1$ positive ones in the axis case. The condition $\varphi(1)>0$ thus fixes the sign of the remaining eigenvalue, but Theorem~\ref{thm:local} does not need it.
\end{remark}

Theorem~\ref{thm:local} follows from Lemma~\ref{lem:eigenvalues}, together with the stable and unstable manifold theorems, which give the corresponding lower bounds on the dimensions of the strong stable and unstable manifolds.
Whenever $n \ge |S_i|+3$, the equilibrium $\bar x_i$ has at least two unstable directions and therefore cannot satisfy condition~\eqref{eq:eigenv}. If this holds for every prescribed equilibrium, the heteroclinic cycles of the type introduced in Section~\ref{sec:heterocline} are excluded.

\begin{remark}
Assume in addition that the prescribed equilibria are hyperbolic.
The qualitative distribution of stable and unstable eigenvalues is robust under small perturbations of the connectivity matrix, that is, under passing from~\eqref{eq:WC} to~\eqref{eq:WC-perturbed-W}.
As a consequence, the absence of codimension-one saddle equilibria among the prescribed states, and hence of heteroclinic cycles of the type considered above, persists under such perturbations.
\end{remark}

\section{Nondegeneracy of the prescribed equilibria}
\label{app:nondegeneracy}

Corollary~\ref{cor:no-sequential} asks that the prescribed equilibria be nondegenerate. We record here what this amounts to in terms of the data $\bar X$ and $b$. Throughout, $\bar x_1,\ldots,\bar x_n$ satisfy the assumptions of Theorem~\ref{thm:main}, so that $W = A - bu^\top$ with $A = \diag_k\big(\sigma^{-1}(a_k)/a_k\big)$ diagonal and detailed balance holds with $\tau_k = u_k/b_k>0$, and we set
\begin{equation}
    \label{eq:delta-i}
    \delta^i_k := \frac{1}{\sigma'(\sigma^{-1}(\bar X_{ki}))} - \frac{\sigma^{-1}(a_k)}{a_k},
    \qquad i,k = 1,\ldots,n .
\end{equation}

\begin{proposition}
    \label{prop:nondegeneracy}
    For every $i=1,\ldots,n$, the following hold.
    \begin{enumerate}[label=\textup{(\roman*)}]
        \item The spectrum of $Df(\bar x_i)$ is real. In particular, $\bar x_i$ is nondegenerate if and only if it is hyperbolic.
        \item The equilibrium $\bar x_i$ is nondegenerate if and only if
              \begin{equation}
                  \label{eq:det-nondeg}
                  \prod_{k=1}^n \delta^i_k \ + \ \sum_{k=1}^n u_kb_k \prod_{j\neq k}\delta^i_j \ \neq \ 0 ,
              \end{equation}
              a condition that reads
              \begin{equation}
                  \label{eq:nondeg-criterion}
                  1 + \sum_{k=1}^n \frac{u_kb_k}{\delta^i_k} \ \neq \ 0
              \end{equation}
              when every $\delta^i_k$ is nonzero. If two of the $\delta^i_k$ vanish, then $\bar x_i$ is degenerate for every choice of $b$. Otherwise, the patterns being fixed, it is degenerate only for $b$ in a hyperplane of $\mathbb{R}^n$.
        \item Assume in addition that $\sigma$ satisfies the hypothesis of Theorem~\ref{thm:local}. Then $\delta^i_k>0$ for $k \in S_i$ and $\delta^i_k<0$ for $k \notin S_i$, and the left-hand side of~\eqref{eq:nondeg-criterion} is the number $\varphi(1)$, with $\varphi$ as in~\eqref{eq:phi}. It is positive as soon as
              \begin{equation}
                  \label{eq:nondeg-sufficient}
                  \sum_{k \notin S_i} \frac{u_kb_k}{|\delta^i_k|} \ < \ 1 ,
              \end{equation}
              in which case $Df(\bar x_i)$ has exactly $|S_i|$ negative and $n-|S_i|$ positive eigenvalues.
    \end{enumerate}
\end{proposition}

\begin{proof}
    Since $\bar x_i$ is an equilibrium, $(W\bar x_i+b)_k = \sigma^{-1}(\bar X_{ki})$ for every $k$, so that
    \begin{equation}
        \label{eq:jacobian-prescribed}
        Df(\bar x_i) = -\idty + \Sigma_i W = -\Sigma_i\big(\Lambda_i - W\big),
        \
        \Sigma_i := \diag_k\big(\sigma'(\sigma^{-1}(\bar X_{ki}))\big),
        \
        \Lambda_i := \Sigma_i^{-1} .
    \end{equation}
    The matrix $\Sigma_i$ is diagonal and positive, so $\bar x_i$ is nondegenerate exactly when $\Lambda_i - W$ is invertible.

    (i) The matrix $T(\Lambda_i-W)$ is symmetric, because $T\Lambda_i$ is diagonal and $TW$ is symmetric by detailed balance. Writing $P := \Sigma_iT^{-1}$, which is diagonal and positive, \eqref{eq:jacobian-prescribed} exhibits $Df(\bar x_i) = -P\cdot T(\Lambda_i-W)$ as similar to the symmetric matrix $-P^{1/2}\,T(\Lambda_i-W)\,P^{1/2}$.

    (ii) By~\eqref{eq:delta-i} we have $\Lambda_i - W = \Delta_i + bu^\top$ with $\Delta_i := \Lambda_i - A = \diag_k(\delta^i_k)$. Expanding $\det(\Delta_i + bu^\top)$ by multilinearity in the columns, the terms taking two or more columns from $bu^\top$ vanish, those columns being proportional, and the remaining ones are the left-hand side of~\eqref{eq:det-nondeg}. When $\Delta_i$ is invertible, that expression is the one given by the Matrix Determinant Lemma and factors as $\det(\Delta_i)\big(1+\sum_k u_kb_k/\delta^i_k\big)$, whence~\eqref{eq:nondeg-criterion}. Finally, the $\delta^i_k$ and the $u_k$ are determined by $\bar X$ alone, so the left-hand side of~\eqref{eq:det-nondeg} is an affine function of $b$, with constant term $\prod_k\delta^i_k$ and coefficients $u_k\prod_{j\neq k}\delta^i_j$. If two of the $\delta^i_k$ vanish, all of these vanish as well. If at most one does, at least one coefficient is nonzero, because $u_k>0$ by~\eqref{eq:u-positive}, and the zero set of a nonconstant affine function is a hyperplane.

    (iii) Comparing~\eqref{eq:delta-i} with~\eqref{eq:eigen-sigmaA} gives $\delta^i_k = (1-\lambda_k)/\Sigma_{i,kk}$, so that the sign of $\delta^i_k$ is that of $1-\lambda_k$, determined in the proof of Lemma~\ref{lem:eigenvalues}. In particular no $\delta^i_k$ vanishes, so~\eqref{eq:nondeg-criterion} is the criterion, and $\alpha_k = u_kb_k\Sigma_{i,kk}$ turns its left-hand side into $1+\sum_k\alpha_k/(1-\lambda_k) = \varphi(1)$. Discarding the positive terms of that sum shows that $\varphi(1)>0$ under~\eqref{eq:nondeg-sufficient}, and Remark~\ref{rem:phi-one} then locates the spectrum.
\end{proof}

\begin{remark}
    Part~(iii) also follows from the factorization $p(\lambda) = \varphi(\lambda+1)\det(\Sigma A - (\lambda+1)\idty)$ of the characteristic polynomial obtained in Section~\ref{app:counting-eigenvalues}. No $\lambda_j$ equals $1$ there, so the second factor does not vanish at $\lambda = 0$, and $0$ is an eigenvalue of $Df(\bar x_i)$ if and only if $\varphi(1)=0$. Degeneracy thus means that the populations activated by the $i$-th pattern balance exactly, in the sense of~\eqref{eq:nondeg-criterion}, those that it leaves silent.
\end{remark}

\section{Uniform estimates for the local passage}
\label{app:local-passage}
In this section we prove Lemma~\ref{lem:passage}. The argument is a uniform version of the standard local analysis near a hyperbolic saddle. Compare with the discussion of local--global Poincaré maps and Shilnikov variables in \cite[Sections~2.3 and~3]{homburg2010homoclinic}.

\begin{proof}[Proof of Lemma~\ref{lem:passage}]
Fix \(i\). All constants below are uniform over $f$. This follows from the uniform \(C^3(K)\)-bound, the \(C^1(K)\)-closeness to \(g\), and the
finiteness of the set of equilibria.

\smallskip
\noindent
\emph{Step 1. Hyperbolic continuation and adapted coordinates.}
Since \(Dg(\bar x_i)\) is invertible, the implicit function theorem gives a unique equilibrium \(\bar x_i(f)\) near \(\bar x_i\). Moreover,
\[
    \bar x_i(f)\longrightarrow \bar x_i,
    \qquad
    Df(\bar x_i(f))\longrightarrow Dg(\bar x_i),
\]
as \(f\to g\) in \(C^1(K)\). The eigenvalues therefore remain real, simple, and ordered as in Lemma~\ref{lem:passage}\textup{(i)}.

The stable-, unstable-, and strong-stable manifold theorems, applied uniformly to the admissible vector fields $f$, give the corresponding local invariant manifolds on a common neighborhood, with uniform $C^3$ bounds \cite[Chapters~4--5]{HirschPughShub1977}. Here uniformity follows from the uniform $C^3(K)$ bound and the uniform spectral gaps.
Since \(W^u_{\rm loc}\) and \(W^s_{\rm loc}\) are transverse and \(W^{ss}_{\rm loc}\subset W^s_{\rm loc}\), these manifolds can be straightened simultaneously. We first straighten the transverse pair \(W^u_{\rm loc},W^s_{\rm loc}\), and then straighten \(W^{ss}_{\rm loc}\) inside \(W^s_{\rm loc}\), extending the latter coordinate change constantly in the unstable direction. This yields the coordinates in \eqref{eq:straightened}. On a common smaller neighborhood, the charts and their inverses have uniform $C^3$ bounds and depend continuously on $f$ in the $C^1$ topology. Consequently, the transformed vector fields have uniform $C^2$ bounds.

Put \(v=(s,w)\). Both coordinate subspaces \(\{u=0\}\) and \(\{v=0\}\) are invariant. Hadamard's lemma therefore puts the vector field in the exact form
\begin{equation}
    \label{eq:app-factored}
    \dot u=u\,a^f(u,v),
    \qquad
    \dot v=B^f(u,v)v,
\end{equation}
with
\[
    a^f(0,0)=\lambda_u^i(f),
    \qquad
    B^f(0,0)
    =
    \operatorname{diag}\bigl(-\lambda_s^i(f),-\lambda_{ss}^i(f)\bigr).
\]
The coefficients $a^f$ and $B^f$ have uniform $C^1$ bounds.

Recall that $\kappa<\lambda_u^i$ and $\beta_i>1$ by the choice preceding~\eqref{eq:kappa-choice}. After reducing \(\delta_0\) and the \(C^1\)-neighborhood of \(g\), if necessary, we have throughout the local box that
\begin{equation}
    \label{eq:app-hyperbolic-bounds}
    \lambda_u^i-\kappa
    \le a^f(u,v)\le
    \lambda_u^i+\kappa,
    \qquad
    \langle B^f(u,v)v,v\rangle
    \le-(\lambda_s^i-2\kappa)|v|^2.
\end{equation}
Moreover, \(B^f\) is uniformly close to \(\operatorname{diag}(-\lambda_s^i,-\lambda_{ss}^i)\) in the maximum norm. The uniform \(C^1\)-bounds on the coefficients further give
\begin{equation}
    \label{eq:app-structural-bounds}
    \begin{aligned}
        |a^f(u,v)-\lambda_u^i(f)|
        &\le C(|u|+|v|),\\
        |\partial_u(u a^f)-a^f|+\|D_v(u a^f)\|
        &\le C|u|,\\
        \|\partial_u(B^f v)\|
        &\le C|v|,
    \end{aligned}
\end{equation}
and, if necessary, after one more reduction of the box,
\begin{equation}
    \label{eq:app-linear-stable}
    \bigl\langle D_v(B^f(u,v)v)\zeta,\zeta\bigr\rangle
    \le-(\lambda_s^i-2\kappa)|\zeta|^2.
\end{equation}

\smallskip
\noindent
\emph{Step 2. Existence, size, and duration of the passage.}
The closeness of \(B^f\) to its diagonal linear part implies that the vector field points inward on the four stable faces \(s=\pm\delta\) and \(w=\pm\delta\). On the other hand, the first inequality in \eqref{eq:app-hyperbolic-bounds} gives
\[
    \dot u\ge(\lambda_u^i-\kappa)u>0.
\]
Consequently, every orbit starting at \((u_0,\delta,w_0)\), with \(0<u_0\le\delta\) and \(|w_0|\le\delta\), remains in the local box until it reaches \(u=\delta\), and it does so in finite time. Denote this time by \(T=\tau_i^f(u_0,w_0)\).

Set \(c_s:=\lambda_s^i-2\kappa\). From \eqref{eq:app-hyperbolic-bounds},
\begin{equation}
\label{eq:decay-v}
    |v(t)|\le |v(0)|e^{-c_st}
    \le \sqrt2\,\delta e^{-c_st}.
\end{equation}
The upper bound for \(a^f\) and the identity \(u(T)=\delta\) yield
\[
    T\ge
    \frac{1}{\lambda_u^i+\kappa}
    \log\frac{\delta}{u_0}.
\]
Hence,
\begin{equation}
    \label{eq:app-size}
    |z_i^f(u_0,w_0)|=|v(T)|
    \le
    \sqrt2\,\delta (\tfrac{u_0}{\delta})^{\beta_i}.
\end{equation}
This proves \eqref{eq:local-size}, with \(C_{0,i}=\sqrt{2}\). The strict inequality \(|z_i^f(u_0,w_0)|>0\) follows from uniqueness, since the invariant solution \(v=0\) cannot meet a solution with \(v(0)=(\delta,w_0)\ne0\).

To estimate the passage time, write
\[
    a^f(u,v)=\lambda_u^i(f)+r^f(u,v).
\]
By \eqref{eq:app-structural-bounds}, \(|r^f(u,v)|\le C(|u|+|v|)\). Moreover,
\[
    \int_0^T u(t)\,dt\le C\delta,
    \qquad
    \int_0^T|v(t)|\,dt\le C\delta,
\]
uniformly in \(f,u_0,w_0\). Integrating \(\dot u/u=a^f(u,v)\) therefore gives
\[
    \log\frac{\delta}{u_0}
    =
    \lambda_u^i(f)T+O(\delta),
\]
which is \eqref{eq:passage-time}.

\smallskip
\noindent
\emph{Step 3. Derivative of the passage map.}
Fix an initial point $y_0=(u_0,\delta,w_0)\in\Sigma_i^{\rm in}$ and a tangent vector $h_0=(\xi_0,0,\eta_0) \in T_{y_0}\Sigma_i^{\rm in}.$ Writing \(h_0=(\xi_0,\zeta_0)\), with \(\zeta_0=(0,\eta_0)\), let $(\xi(t),\zeta(t)):=D\varphi_t^f(y_0)h_0$ be the corresponding variation of the trajectory at fixed time. It satisfies the variational equation
\begin{equation}
    \label{eq:app-variational}
    \begin{aligned}
        \dot\xi
        &=\partial_u(u a^f)\xi+D_v(u a^f)\zeta,\\
        \dot\zeta
        &=\partial_u(B^f v)\xi+D_v(B^f v)\zeta.
    \end{aligned}
\end{equation}
Since \(u(t)>0\), put \(p(t):=\xi(t)/u(t)\). From \eqref{eq:app-structural-bounds}--\eqref{eq:app-linear-stable},
\begin{equation}
    \label{eq:app-variation-inequalities}
    |\dot p|
    \le C\bigl(u|p|+|\zeta|\bigr),
    \qquad
    \frac{d}{dt}|\zeta|
    \le-c_s|\zeta|+Cu|v||p|.
\end{equation}
The variation of constants formula, \eqref{eq:decay-v}, and Gronwall's inequality give
\begin{equation}
    \label{eq:app-p-zeta}
    \sup_{0\le t\le T}|p(t)|
    \le
    C\left(\frac{|\xi_0|}{u_0}+|\zeta_0|\right),
    \qquad
    |\zeta(T)|
    \le
    C(\tfrac{u_0}{\delta})^{\beta_i-1}(|\xi_0|+|\zeta_0|).
\end{equation}
Indeed, one first inserts the estimate for \(p\) into the variation-of-constants formula for \(\zeta\) and uses \(\int_0^T u(t)\,dt\le C\delta\). The factor \((\tfrac{u_0}{\delta})^{\beta_i-1}\) then follows from \(e^{-c_sT}\le (\tfrac{u_0}{\delta})^{\beta_i}\).

We now account for the dependence of the hitting time on the initial condition. Differentiating the identity \(u(T(u_0,w_0))=\delta\) gives
\[
    DT
    =
    -\frac{p(T)}{a^f(\delta,v(T))},
\]
and hence, by \eqref{eq:app-hyperbolic-bounds} and \eqref{eq:app-p-zeta},
\[
    |DT|
    \le
    C\left(\frac{|\xi_0|}{u_0}+|\zeta_0|\right).
\]
Since the stable component of the passage map is \(v(T(u_0,w_0))\),
\[
    Dz_i^f(u_0,w_0)(\xi_0,\zeta_0)
    =
    \zeta(T)+B^f(\delta,v(T))v(T)\,DT.
\]
Combining this identity with \eqref{eq:app-size} and \eqref{eq:app-p-zeta} proves
\[
    \|Dz_i^f(u_0,w_0)\|
    \le
    C_{1,i}(\tfrac{u_0}{\delta})^{\beta_i-1},
\]
which is \eqref{eq:local-derivative}.

\smallskip
\noindent
\emph{Step 4. Exit direction.}

It remains to prove \eqref{eq:exit-direction}. Introduce
\[
    \sigma:=\log\frac{u}{\delta},
    \qquad
    \sigma_0:=\log\frac{u_0}{\delta}<0,
\]
so that \(\sigma=0\) at exit. Since \(d\sigma/dt=a^f(u,v)>0\), the stable equation becomes
\begin{equation}
    \label{eq:app-log-stable}
    \frac{dv}{d\sigma}
    =
    \widehat B^f(\sigma,v)v,
    \qquad
    \widehat B^f(\sigma,v)
    :=
    \frac{B^f(\delta e^\sigma,v)}
         {a^f(\delta e^\sigma,v)}.
\end{equation}
Along the unstable branch \(v=0\),
\[
    \widehat B^f(\sigma,0)
    =
    \operatorname{diag}
    \bigl(-\nu_i(f),-\chi_i(f)\bigr)
    +O(\delta e^\sigma),
\]
uniformly in \(f\), where $\nu_i(f):=\frac{\lambda_s^i(f)}{\lambda_u^i(f)}$ and $ \chi_i(f):=\frac{\lambda_{ss}^i(f)}{\lambda_u^i(f)}$.
Since \(\lambda_{ss}^i>\lambda_s^i\), after reducing the neighborhood of \(g\) there exists \(\gamma_i>0\) such that
\begin{equation}
    \label{eq:app-projective-gap}
    \chi_i(f)-\nu_i(f)\ge2\gamma_i
\end{equation}
for every admissible \(f\).

Write
\[
    \widehat B^f
    =
    \begin{pmatrix}
        b_{11}&b_{12}\\
        b_{21}&b_{22}
    \end{pmatrix}.
\]
Whenever \(s\neq0\), set \(q=w/s\). Since \(v=s(1,q)\), equation \eqref{eq:app-log-stable} gives
\begin{equation}
    \label{eq:app-projective}
    q'
    =
    \Phi^f(\sigma,s,q)
    :=
    b_{21}
    +(b_{22}-b_{11})q
    -b_{12}q^2,
\end{equation}
where the coefficients are evaluated at \(v=s(1,q)\). Although \(q\) is not defined at \(v=0\), the right-hand side extends smoothly to \(s=0\) and describes the projectivized stable variational equation along the unstable branch.

By \eqref{eq:app-projective-gap}, the equation at \(s=0\) has the form
\[
    q'
    =
    -\bigl(\chi_i(f)-\nu_i(f)\bigr)q
    +R^f(\sigma,q),
    \qquad
    |R^f|+|\partial_qR^f|
    \le
    C\delta e^\sigma
\]
for bounded \(q\). A standard contraction argument gives a unique solution \(q_*^f\) satisfying
\begin{equation}
    \label{eq:app-qstar}
    q_*^f(\sigma)\longrightarrow0
    \quad\text{as }\sigma\to-\infty,
    \qquad
    |q_*^f(\sigma)|
    \le
    C\delta e^\sigma.
\end{equation}
Define
\begin{equation}
    \label{eq:app-eout}
    e_i^{\rm out}(f)
    :=
    \frac{(1,q_*^f(0))}
         {\sqrt{1+|q_*^f(0)|^2}},
\end{equation}
with the orientation induced by \(+e_i^s\).

We now compare \(q_*^f\) with the projective coordinate \(q\) of a genuine trajectory. Choose \(\rho_i>0\) sufficiently small. The spectral gap and the uniform \(C^1\)-bounds imply, for \(|q|\le2\rho_i\),
\[
    \partial_q\Phi^f(\sigma,0,q)\le-\gamma_i,
    \qquad
    |\Phi^f(\sigma,s,q)-\Phi^f(\sigma,0,q)|
    \le C|v|.
\]
Hence, setting \(r=q-q_*^f\),
\begin{equation}
    \label{eq:app-r}
    \frac{d}{d\sigma}|r|
    \le
    -\gamma_i|r|+C|v(\sigma)|.
\end{equation}
The stable decay estimate obtained above gives
\begin{equation}
    \label{eq:app-v-sigma}
    |v(\sigma)|
    \le
    C\delta e^{-\beta_i(\sigma-\sigma_0)},
    \qquad
    \sigma_0\le\sigma\le0.
\end{equation}
Since \(|q(\sigma_0)|=|w_0|/\delta\le\rho_i\), a continuation argument using \eqref{eq:app-qstar}--\eqref{eq:app-v-sigma} shows that
\[
    s(\sigma)>0,
    \qquad
    |q(\sigma)|<2\rho_i,
    \qquad
    \sigma_0\le\sigma\le0,
\]
provided \(\delta_0\) is sufficiently small. Thus \(q\) remains well defined throughout the passage.

Variation of constants in \eqref{eq:app-r} now yields, for some \(\mu_i>0\),
\begin{equation}
    \label{eq:app-q-exit}
    |q(0)-q_*^f(0)|
    \le
    C_{\mathrm{dir},i}
    \left(\frac{u_0}{\delta}\right)^{\mu_i}.
\end{equation}

Finally, since \(s(0)>0\),
\[
    \frac{z_i^f(u_0,w_0)}
         {|z_i^f(u_0,w_0)|}
    =
    \frac{(1,q(0))}
         {\sqrt{1+q(0)^2}}.
\]
The map \(q\mapsto(1,q)/\sqrt{1+q^2}\) is Lipschitz with constant at most one. Therefore \eqref{eq:app-eout} and \eqref{eq:app-q-exit} prove \eqref{eq:exit-direction}.

We conclude with the continuous dependence of $e_i^{\rm out}(f)$ on $f$. Fix an admissible $\delta>0$ and let $v_i^f(u):=z_i^f(u,0)/\|z_i^f(u,0)\|$ for $0<u<\delta$. For each fixed $u$, the map $f\mapsto v_i^f(u)$ is continuous in the $C^1(K)$ topology. Indeed, the adapted coordinates depend continuously on $f$ by Step 1, and the local passage takes place in finite time, with a transverse crossing of the outgoing section. Moreover, $\|z_i^f(u,0)\|>0$ by \eqref{eq:local-size}, so the normalization is well defined. By \eqref{eq:exit-direction}, for any two admissible fields $f$ and $f_0$,
\[
\|e_i^{\rm out}(f)-e_i^{\rm out}(f_0)\|
\le 2C_{\mathrm{dir},i}\left(\frac{u}{\delta}\right)^{\mu_i}
+\|v_i^f(u)-v_i^{f_0}(u)\|.
\]
Given $\varepsilon>0$, we first choose $u>0$ sufficiently small that the first term is less than $\varepsilon/2$, uniformly in $f$ and $f_0$. Keeping this $u$ fixed, the continuity of $v_i^f(u)$ makes the second term less than $\varepsilon/2$ whenever $f$ is sufficiently close to $f_0$ in $C^1(K)$. Thus $e_i^{\rm out}(f)$ depends continuously on $f$, as claimed.
\end{proof}

\section{Details of the focused attention meditation case study}
\label{sec:FAM-details}

\subsection{Proof of Proposition~\ref{prop:LV-hypotheses}}
\label{sec:proof-LV-hypotheses}

    Since every component \(g_k(x)\) contains the factor \(x_k\), the coordinate planes and the positive octant are invariant. A direct differentiation at \(\bar x_i=a_i e_i\) gives $ \operatorname{spec}Dg(\bar x_i)=\{-1,-\lambda_s,\lambda_u^i\}.$ Hence \(\bar x_i\) is a hyperbolic saddle satisfying Assumption~\ref{A:spectrum}. Moreover, $\nu_i=\frac{\lambda_s}{\lambda_u^i}>1,$ so Assumption~\ref{A:product} also holds.

    To prove the existence of the connections, fix \(i\) and consider
    the invariant coordinate plane $P_i:=\{x_{i+2}=0\}$.
    In the normalized coordinates
    \[
        y:=\frac{x_i}{a_i},
        \qquad
        z:=\frac{x_{i+1}}{a_{i+1}},
    \]
    the restriction of \(g\) to \(P_i\) is
    \begin{equation}
        \label{eq:LV-plane}
        \dot y
        =
        y\bigl(1-y-(1+\lambda_s)z\bigr),
        \qquad
        \dot z
        =
        z\bigl(1-(1-\lambda_u^i)y-z\bigr).
    \end{equation}
    In the open quadrant \(y,z>0\),
    \begin{equation}
        \label{eq:LV-ratio}
        \frac{d}{dt}\log\frac{z}{y}
        =
        \lambda_u^i y+\lambda_s z
        >0.
    \end{equation}
    Thus \eqref{eq:LV-plane} has neither an equilibrium nor a periodic orbit in the open quadrant.
    The square \([0,1]^2\) is positively invariant, so the branch of \(W^u(\bar x_i)\) entering its open quadrant is bounded. By the Poincaré--Bendixson theorem and \eqref{eq:LV-ratio}, its only possible forward limit is \((0,1)\), corresponding to \(\bar x_{i+1}\). The boundary equilibria $(0,0)$ and $(1,0)$ cannot be forward limits of this branch, and the boundary dynamics admits no cycle of connections. Hence there is a heteroclinic connection
    \[
        \Gamma_i^+
        \subset
        W_+^u(\bar x_i)\cap W^s(\bar x_{i+1}).
    \]
    The three connections form the heteroclinic cycle $\Gamma$. The coordinate planes are invariant, and the planar dynamics described above shows that $\Gamma$ is asymptotically stable relative to the boundary of the positive octant. Moreover, at each $\bar x_i$ the two external eigenvalues are $\lambda_u^i>0$ and $-\lambda_s<0$, and
    \[
    \prod_{i=1}^3\lambda_u^i<\lambda_s^3,
    \]
    by \eqref{eq:LV-parameters}. Hence the stability criterion for planar heteroclinic cycles in \cite[Corollary~2]{hofbauer1994heteroclinic} implies that $\Gamma$ is asymptotically stable relative to $\mathbb R_{\geq0}^3$.

    The eigendirection corresponding to \(-1\) is the \(x_i\)-axis. Since this axis is invariant, it coincides locally with \(W^{ss}(\bar x_i)\). The incoming connection \(\Gamma_{i-1}^+\), however, lies in the open quadrant of \(P_{i-1}\) and is not contained in the \(x_i\)-axis. Therefore Assumption~\ref{A:incoming} holds.

    Finally, \(P_i\) is invariant and, near \(\bar x_{i+1}\), coincides with \(W^s_{\rm loc}(\bar x_{i+1})\). Hence \(U_i^g\) vanishes on \(P_i\cap\Sigma_i^{\rm out}\), and therefore $T_0(P_i\cap\Sigma_i^{\rm out})\subset \ker DU_i^g(0)$. Since \(\Psi_i^g\) is a local diffeomorphism, \(DU_i^g(0)\neq0\), so both spaces are one-dimensional and the inclusion is an equality.

    The leading stable eigendirection at \(\bar x_i\) is transverse to \(P_i\). Since \(P_i\) is invariant, the differential of the flow maps its tangent spaces onto themselves along \(\Gamma_i^+\). Therefore the transported leading stable direction, and hence \(e_i^{\rm out}(g)\), remains transverse to \(P_i\). In particular, \(e_i^{\rm out}(g)\) does not belong to \(\ker DU_i^g(0)\), and thus $ DU_i^g(0)[e_i^{\rm out}(g)]\neq0.$

    Since $P_i=\{x_{i+2}=0\}$ is invariant, the linearized flow preserves the sign of the normal component along $\Gamma_i^+$. With the orientations induced by the positive octant, this gives
    \[
    DU_i^g(0)[e_i^{\rm out}(g)]>0,
    \]
    and Assumption~\ref{A:transit} follows.

     \qed

\subsection{Proof of Proposition~\ref{prop:LV-positive-splittings}}
\label{sec:proof-LV-positive-splittings}

By the positive diagonal change of coordinates $x_j=a_j z_j$, we may assume without loss of generality that $a_j=1$ for $j=1,2,3$, so that $\bar x_j=e_j$.

By hyperbolic persistence and continuous dependence of local invariant manifolds and finite-time transition maps, choose $\varepsilon_0>0$ sufficiently small that the continued saddles $\bar x_i(f)$, adapted charts, and global transition maps are defined as in Section~\ref{sec:sequential-periodic}.
Fix $i$, with indices taken modulo $3$, and set
\[
    k=i+2,\qquad q=\bar x_i(f),\qquad p=\bar x_{i+1}(f).
\]

Condition~\eqref{eq:positive-plane-condition} makes $\{x_k\ge0\}$ forward invariant as long as the orbit remains in $K$.

The signs $\partial_k g_k(e_i)=-\lambda_s^i<0$ and $\partial_k g_k(e_{i+1})=\lambda_u^{i+1}>0$, together with~\eqref{eq:positive-plane-condition}, place the perturbed saddles on opposite sides of $\{x_k=0\}$:
\begin{equation}
\label{eq:perturbed-saddle-sides}
    q_k>0,\qquad p_k<0.
\end{equation}
Indeed, let $p^0=p-p_ke_k$ and define $ \phi(t)=f_k(p^0+te_k).$ Since $p\to e_{i+1}$ as $f\to g$, the segment joining $p^0$ to $p$ lies in a fixed neighborhood of $e_{i+1}$ contained in $K$. By $C^1$ closeness, after decreasing $\varepsilon_0$ if necessary, $\phi'(t)=\partial_k f_k(p^0+te_k) \ge \frac12\lambda_u^{i+1}>0$ on this segment. Moreover, $\phi(0)>0$ and $\phi(p_k)=f_k(p)=0$. Thus $p_k<0$. The same argument at $q$, using $\partial_k f_k\le-\lambda_s^i/2<0$ near $e_i$, gives $q_k>0$.

Let $y$ be the intersection of the continued unstable branch of $q$ with the extended incoming section at $p$. Since this branch converges to $q$ in backward time and $q_k>0$, its points sufficiently close to $q$ have $x_k>0$. The branch segment up to $y$ remains in $K$, so forward invariance yields $y_k>0$.
On the other hand, $ W^s_{\mathrm{loc}}(p)\subset\{x_k<0\}.$ Indeed, the forward orbit of a point in this local stable manifold remains in the chosen neighborhood contained in $K$ and converges to $p$. If its initial $k$-th coordinate were nonnegative, forward invariance and convergence would give $p_k\ge0$, contradicting \eqref{eq:perturbed-saddle-sides}.

Finally, write the incoming section in adapted coordinates as $\widehat\Sigma_{i+1}^{\mathrm{in}} =\{s=\delta,\ |u|<\delta,\ |w|<\delta\},$ and denote the physical coordinate $x_k$ on this section by $X_k(u,w)$.
For $g$, the local stable manifold coincides with $\{x_k=0\}$, and the oriented unstable coordinate vector at the saddle has positive $k$-th component. Thus, after shrinking the chart, $\partial_u X_k>0$ throughout the section for $g$. By $C^1$ continuous dependence of the charts on $f$, this inequality persists for all sufficiently small perturbations.

Write $y=(u_y,\delta,w_y)$ in adapted coordinates. Since $(0,\delta,w_y)\in W^s_{\mathrm{loc}}(p)$, $X_k(0,w_y)<0$ and $X_k(u_y,w_y)=y_k>0$. Strict monotonicity in $u$ implies $u_y>0$. By definition, $d_i(f)=u_y$, proving the claim. \qed

\subsection{Training procedure}
\label{sec:training}

We implement $f_\theta$ as a feedforward neural network with a single nonlinear hidden layer and hyperbolic tangent activation $\sigma$. The construction and training are implemented in Python~\cite{bolelli2026code}, and the prescribed block structure of $P$ is kept throughout the optimization.
The procedure is summarized in Algorithm~\ref{alg:approx}, where the loss is the one of~\eqref{eq:numerical-training-loss}.

\begin{algorithm}[htb]
\caption{Approximation of the target dynamics}
\label{alg:approx}
\begin{algorithmic}[1]

\State Define the target field: implement $g$ of Section~\ref{sec:target} on $K$.

\State Generate the training data: sample $S=\{x^{(n)}\}_{n=1}^{D}\subset K$ and compute $g(x^{(n)})$ and $Dg(x^{(n)})$.

\State Construct the neural architecture: fix $N\gg3$ and define $f_\theta(x)=-x+P\sigma(Wx+b),$ with parameters $\theta=(P,W,b)$. Prescribe three disjoint population blocks in $P$.

\State Initialize the network: randomly initialize the trainable entries of $P$, $W$, and $b$, keeping the remaining entries of $P$ equal to zero.

\State Train the parameters: at each optimization step, sample additional points on the coordinate planes and use Adam to minimize the total loss $\mathcal L=\mathcal L_{C^1}+\mathcal L_T$ defined in~\eqref{eq:numerical-training-loss}.

\State Return the trained dynamics: the resulting system is $\dot x=f_\theta(x)$.

\end{algorithmic}
\end{algorithm}

For each of the two configurations of Section~\ref{sec:simulation}, we train a separate network $f_\theta$ with $N=90$ hidden units organized into three blocks of $30$ on $D=8000$ samples uniformly distributed in $K$, using both the values and the Jacobians of $g$. We train with the Adam optimizer for $3000$ epochs, with a batch size of $512$ and a learning rate decreasing geometrically from $10^{-2}$ to $10^{-4}$.
We set $\lambda_J=1$, $\lambda_T=100$, and $m=0.005$, so that violations of the positive normal component weigh heavily in the loss. At each optimization step, the plane penalty is evaluated on $192$ newly sampled points, with $64$ points on each coordinate plane.
No weight regularization is used.

For reproducibility, we use the random seed $34554354353$ in both configurations. The computational domain is $K=[-0.05,1.15]^3$. The projection matrix is parametrized as $P=M\operatorname{diag}(p)$, where $M\in\{0,1\}^{3\times90}$ is a fixed block-membership matrix assigning $30$ populations to each macroscopic variable, and $p\in\mathbb R^{90}$ contains the trainable projection weights. Before training, the entries of $W$, $b$, and $p$ are independently initialized from normal distributions with zero mean and standard deviations $0.7$, $0.1$, and $0.5$, respectively. Adam is implemented with $\beta_1=0.9$, $\beta_2=0.999$, and numerical constant $10^{-8}$.

For the population-level simulations, the initial condition is chosen as $y(0)=P^\dagger x(0)$, where $P^\dagger$ denotes the Moore--Penrose pseudoinverse, so that $Py(0)=x(0)$. The reduced and population systems are integrated independently using the same numerical settings. The population trajectories in Figure~\ref{fig:neural} are displayed over the final portions of the simulations, approximately $[\lvWindowStartSym,500]$ in the symmetric case and $[\lvWindowStartAsym,500]$ in the asymmetric case.

To assess the approximation accuracy, we evaluate the trained vector fields on a test set of $\lvTestPoints$ points sampled in $K$. For the symmetric and asymmetric configurations, respectively, the maximum sampled errors between $f_\theta$ and $g$ are $\lvTestFieldMaxSym$ and $\lvTestFieldMaxAsym$, with RMS errors of $\lvTestFieldRmsSym$ and $\lvTestFieldRmsAsym$. The corresponding maximum errors between their Jacobians, measured in Frobenius norm, are $\lvTestJacMaxSym$ and $\lvTestJacMaxAsym$, with RMS errors of $\lvTestJacRmsSym$ and $\lvTestJacRmsAsym$. These quantities measure the empirical approximation accuracy and do not provide uniform $C^1(K)$ error bounds. Similarly, the positivity condition \eqref{eq:positive-plane-condition} is encouraged through the training penalty but is not certified on the entire coordinate planes, although it holds on the grid of Section~\ref{sec:checks}.

\subsection{Additional numerical checks}
\label{sec:checks}

Beyond the single trajectory of Figure~\ref{fig:results1}, we checked the two trained fields against the hypotheses of Section~\ref{sec:sequential-periodic}.
We give the values for the symmetric configuration, with those for the asymmetric one in parentheses.
On the sample points within distance $\lvGammaRadius$ of $\Gamma$, the maximum errors are $\lvGammaFieldMaxSym$ ($\lvGammaFieldMaxAsym$) between $f_\theta$ and $g$, and $\lvGammaJacMaxSym$ ($\lvGammaJacMaxAsym$) between their Jacobians in Frobenius norm.
Near each saddle $e_i$, a search by Newton's method from random starting points finds a single equilibrium of $f_\theta$, at distance at most $\lvSaddleDistMaxSym$ ($\lvSaddleDistMaxAsym$) from $e_i$, and the eigenvalues of $Df_\theta$ there have the same signs as those of $Dg(e_i)$ and stay close to them.
The normal component $(f_\theta)_k$ is positive at every point of a $\lvPlaneGrid\times\lvPlaneGrid$ grid on each face $K\cap\{x_k=0\}$, with minimum $\lvPlaneMinSym$ ($\lvPlaneMinAsym$).
All $\lvSweepCount$ trajectories of a sweep of initial conditions in $[\lvSweepLow,\lvSweepHigh]^3$ converge to the same periodic orbit.
Their periods agree to a relative precision of $\lvSweepPeriodRangeSym$ ($\lvSweepPeriodRangeAsym$), and for $t\ge\lvSweepTailStart$ they lie within $\lvSweepDistanceSym$ ($\lvSweepDistanceAsym$) of the orbit through the trajectory of Figure~\ref{fig:results1}.
The nontrivial Floquet multipliers of this orbit are $\lvFloquetSecondSym$ and $\lvFloquetThirdSym$ ($\lvFloquetSecondAsym$ and $\lvFloquetThirdAsym$), so the orbit is linearly stable.

\clearpage
\bibliographystyle{elsarticle-num}
\bibliography{references}

\end{document}